\documentclass[11pt]{article}
\usepackage[utf8]{inputenc}
\usepackage{geometry}
\usepackage{bbm}
\usepackage{amsmath,amsthm, amssymb, amsfonts}
\usepackage[colorlinks, linkcolor=blue]{hyperref}
\newtheorem{theorem}{Theorem}
\newtheorem{lemma}[theorem]{Lemma}
\newtheorem{remark}[theorem]{Remark}
\newtheorem{proposition}[theorem]{Proposition}

\numberwithin{equation}{section}
\numberwithin{theorem}{section}
\usepackage{graphicx}
\providecommand{\keywords}[1]
{
  \small	
  \textbf{Keywords } #1
}
\providecommand{\thank}[1]
{
  \small	
  \textbf{Acknowledgment } #1
}

\allowdisplaybreaks
\title{\textbf{Existence of modified wave operators and infinite cascade result for a half wave Schrödinger equation on the plane}}
\author{Xi Chen}
\date{}

\begin{document}

\maketitle

\begin{abstract}
We consider the following half wave Schrödinger equation,
\begin{center}
$\left(i \partial_{t}+\partial_{x }^2-\left|D_{y}\right|\right) U=|U|^{2} U$
\end{center}
on the plane $\mathbb{R}_x \times \mathbb{R}_y$.  We prove the existence of modified wave operators between small decaying solutions to this equation and small decaying solutions to the non chiral cubic Szeg\H o equation, which is similar to the the existence result of modified wave operators on $\mathbb{R}_x \times \mathbb{T}_y$ obtained by H. Xu [\ref{[2]}]. We then combine our modified wave operators result with a recent cascade result [\ref{[[1]]}] for the cubic Szeg\H o equation by P. Gérard and A. Pushnitski to deduce that there exists solutions $U$ to the half wave Schrödinger equation such that $\|U(t)\|_{L_{x}^{2} H_{y}^{1}}$ tends to infinity as $\log t$ when $t \rightarrow +\infty$. It indicates that the half wave Schrödinger equation on the plane is one of the very few dispersive equations admitting global solutions with small and smooth data such that the $H^s$ norms are going to infinity as $t$ tends to infinity.
\end{abstract}
\\\\
\thank{I am currently a PhD student under the supervision of Patrick Gérard, and I would like to thank him for his supervision of this paper. I also would like to thank the authors of all the articles I cited in this paper, their work has been a great help and inspiration to me. Finally, I would like to thank every friend who has helped me in my research career. }
\tableofcontents{}

\section{Introduction}
Consider the following half wave Schrödinger equation on the plane,
\begin{equation}
\label{1.1}
\left(i \partial_{t}+\mathcal{A}\right) U=|U|^{2} U, \qquad (x, y) \in \mathbb{R} \times \mathbb{R},
\end{equation}
where
\begin{align*}
\mathcal{A}:=\partial_{x }^2-\left|D_{y}\right|.
\end{align*}
The corresponding Hamiltonian function is 
\begin{align*}
H(U)=\frac{1}{2} \int_{\mathbb{R} \times \mathbb{R}}\left(\left|\partial_{x} U(x, y)\right|^{2}+\left|D_{y}\right| U(x, y) \overline{U}(x, y)\right) d x d y+\frac{1}{4} \int_{\mathbb{R} \times \mathbb{R}}|U(x, y)|^{4} d x d y.
\end{align*}
We observe that this equation enjoys the mass conservation 
\begin{align*}
\int_{\mathbb{R} \times \mathbb{R}}|U(t, x, y)|^{2} d x d y=\int_{\mathbb{R} \times \mathbb{R}}|U(0, x, y)|^{2} d x d y.
\end{align*}
The local well-posedness of the Cauchy problem in the energy space $H_{x}^{1} L_{y}^{2} \cap L_{x}^{2} H_{y}^{\frac{1}{2}}$ is still an open problem. Moreover, Y. Bahri, S. Ibrahim and H. Kikuchi have proved the local well-posedness of the Cauchy problem in higher regularity spaces $H_{x}^{2s} L_{y}^{2} \cap L_{x}^{2} H_{y}^{s} \text{ with } s > \frac{1}{2}$ [\ref{[2.5]}], but we still do not have the global existence in these higher regularity spaces. In lower regularity space, as N. Burq, P. Gérard and N. Tzvetkov did in [\ref{[2.1]}], one can prove that the time flow map on $H_{x}^{2s} L_{y}^{2} \cap L_{x}^{2} H_{y}^{s} \text{ with } \frac{1}{4} \leq s  < \frac{1}{2}$ is not $C^{3}$ at the origin. Also, an adaptation of the arguments from [\ref{[2.1]}] implemented in I. Kato [\ref{[2.2]}] implies the ill-posedness in $H_{x}^{2s} L_{y}^{2} \cap L_{x}^{2} H_{y}^{s} \text{ with } s  < \frac{1}{4}$. In this paper, we prove the existence of modified wave operators and corresponding cascade result for (\ref{1.1}) with small decaying data.
\begin{remark}
The sign in front of the nonlinearity is not relevant as far as we are dealing with small data, so the conclusions we obtain in Theorem \ref{Theorem 1.5} and Theorem \ref{Theorem 1.6} are also available in the focusing case. 
\end{remark}
\subsection{Some previous results}
In this subsection, we introduce some important results for the cubic Szeg\H o equation, the half wave equation and the half wave Schrödinger equation.
\subsubsection{The cubic Szeg\H o equation}
The cubic Szeg\H o equation on the circle reads as follows,
\begin{equation}
\label{1.505}
\left\{\begin{array}{l}
i \partial_{t} u=\Pi_{+}\left(|u|^{2} u\right), \quad (t,x) \in \mathbb{R} \times \mathbb{T},\\ u(0) = u_0,
\end{array}\right.
\end{equation}
where 
\begin{equation}
\label{1.405}
\widehat{\Pi_{+} u}(p)=\left\{\begin{array}{l}\widehat{u}(\xi), \text { if } p \in \mathbb{Z}_{\geq 0}, \\ 0, \text { if } p \in \mathbb{Z}-\mathbb{Z}_{\geq 0}.\end{array}\right.
\end{equation}
Also, the cubic Szeg\H o equation on the line is stated as follows,
\begin{equation}
\label{1.51}
\left\{\begin{array}{l}
i \partial_{t} u=\Pi_{+}\left(|u|^{2} u\right), \quad (t,x) \in \mathbb{R} \times \mathbb{R},\\ u(0) = u_0,
\end{array}\right.
\end{equation}
where
\begin{equation}
\label{1.41}
\widehat{\Pi_{+} u}(\xi)=\left\{\begin{array}{l}\widehat{u}(\xi), \text { if } \xi \geq 0, \\ 0, \quad \text { if } \xi<0.\end{array}\right.
\end{equation}
For the cubic Szeg\H o equation on the circle (\ref{1.505}), P. Gérard and S. Grellier have proved the global well-posedness in $H_{+}^{s}(\mathbb{T}) \text { for } s \geq \frac{1}{2}$ [\ref{[3.5]}]. With slight modifications of the proof in [\ref{[3.5]}], O. Pocovnicu has shown the global well-posedness of the cubic Szeg\H o equation (\ref{1.51}) in $H_{+}^{s}(\mathbb{R}) \text { for } s \geq \frac{1}{2}$ [\ref{[3.6]}].
\begin{theorem}
[\ref{[3.6]},  Theorem 1.1]
The cubic Szeg\H o equation (\ref{1.51}) is globally well-posed in $H_{+}^{s}(\mathbb{R})$ for $s \geq \frac{1}{2}$, i.e. given $u_{0} \in H_{+}^{s}(\mathbb{R})$ with $s\geq\frac{1}{2}$, there exists a unique global-in-time solution $u \in C\left(\mathbb{R} ; H_{+}^{s}(\mathbb{R})\right)$ of (\ref{1.51}).
\end{theorem}
See also [\ref{[3.61]}] for the explicit formula of the solution to the Szegö equation (\ref{1.51}) with some special initial data.\\

In fact, P. Gérard and S. Grellier have also studied the large time behavior of solutions to the cubic Szeg\H o equation with some initial data in [\ref{[4]}], and this result is stated as follows.
\begin{proposition}
[\ref{[4]}, Theorem 1]
\label{Proposition 1.39}
There exist initial data $u_0 \in C_{+}^{\infty}(\mathbb{T}):=\bigcap_{s} H^{s}(\mathbb{T})$ and sequences $\left(\overline{t}_{n}\right),\left(\underline{t}^{n}\right)$ tending to infinity such that
$\forall s>\frac{1}{2}, \forall M \in \mathbb{Z}_{+}$, the corresponding solution to (\ref{1.505}) satifies 
\begin{align*}
\frac{\left\|u(\overline{t}_{n})\right\|_{H^{s}}}{\left|\overline{t}_{n}\right|^{M}} \underset{n \rightarrow \infty} \rightarrow \infty,
\end{align*}
and $u(\underline{t}^{n}) \underset{n \rightarrow \infty}{\longrightarrow} u_{0} \text { in } C_{+}^{\infty}$. Furthermore, the set of such initial data is a dense $G_{\delta}$ subset of $C_{+}^{\infty}(\mathbb{T})$.
\end{proposition}
\begin{remark}
As shown in Proposition \ref{Proposition 1.39}, for $s > \frac{1}{2}$, there exist solutions to (\ref{1.505}) with the initial data in a dense $G_\delta$ subset of $C_{+}^{\infty}(\mathbb{T})$ which satisfy
\begin{align*}
\limsup_{|t| \to \infty} \frac{\|u(t)\|_{H^{s}}}{|t|^{M}} & = \infty, \\
\liminf_{|t| \to \infty} \|u(t)\|_{H^{s}} & < \infty.
\end{align*}
However, we do not know if there exists a solution to (\ref{1.505}) such that the $H^s$ norm of this solution tends to infinity as $t$ tends to infinity.
\end{remark}
Before introducing the following proposition, we recall the definition (\ref{4.401}) of Hankel operator $H_u$ in Section 4, and we say that $\lambda > 0$ is a singular value of $H_u$ if the corresponding Schmidt subspace
\begin{align*}
E_{H_{u}}(\lambda)=\operatorname{Ker}\left(H_{u}^{2}-\lambda^{2} I\right)
\end{align*}
is not $\{ 0 \}$. 

Recently, P. Gérard and A. Pushnitski have shown the cascade result for the equation (\ref{1.51}) [\ref{[[1]]}], and this result is stated as follows.
\begin{proposition}
[\ref{[[1]]}, Proposition 9.3]
\label{Proposition 1.4}
Let $u \in H_{+}^1(\mathbb{R})$ be a rational solution of the cubic Szeg\H o equation on the line (\ref{1.51})  such that the Hankel operator $H_u$ has singular values $\lambda_{1}, \cdots, \lambda_{N}$, with $\lambda_1$ being multiple and $\lambda_j$ being simple for every $j \geq 2$. Then
\begin{align*}
0<\liminf_{t \rightarrow\infty} \frac{\left\|\partial_{x} u(t)\right\|_{L^{2}}}{|t|} \leq \limsup_{t \rightarrow \infty} \frac{\left\|\partial_{x} u(t)\right\|_{L^{2}}}{|t|} < +\infty.
\end{align*}
\end{proposition}
\subsubsection{The half wave equation}
The half wave equation reads as follows,
\begin{equation}
\label{1.35}
\left\{\begin{array}{l}i \partial_{t} v-|D| v=|v|^{2} v, \\ v(0)=v_{0}.\end{array}\right.
\end{equation}
The equation (\ref{1.35}) is usually studied on $\mathbb{T}$ or $\mathbb{R}$. P. Gérard and S. Grellier have proved the global well-posedness of (\ref{1.35}) in $H^s(\mathbb{T})$ with $s\geq\frac{1}{2}$ [\ref{[5.5]}], and one can use the analogous method to deduce the global well-posedness of (\ref{1.35}) in $H^s(\mathbb{R})$ with $s\geq \frac{1}{2}$.
\begin{theorem}
[\ref{[5.5]}, Proposition 1] 
Given $u_0 \in H^{s}(\mathbb{T})$ with $s\geq\frac{1}{2}$, there exists a unique solution $u \in C\left(\mathbb{R}, H^{s}(\mathbb{T})\right)$ to (\ref{1.35}). Also, for $u_0 \in H^{s}(\mathbb{R})$ with $s\geq\frac{1}{2}$, there exists a unique solution $u \in C\left(\mathbb{R}, H^{s}(\mathbb{R})\right)$ to (\ref{1.35}).
\end{theorem}
Moreover, O. Pocovnicu has studied partially about its long time behavior for the problem (\ref{1.35}), she has proved that if the initial condition is of order $O(\varepsilon)$ and supported on positive frequencies only, then the corresponding solution can be approximated by the solution of the Szeg\H o equation [\ref{[6]}].

\subsubsection{The half wave Schrödinger equation on the plane}
For (\ref{1.1}), by using the endpoint Strichartz estimate, Y. Bahri, S. Ibrahim and H. Kikuchi have deduced the local well-posedness of the Cauchy problem in $H_{x}^{2 s} L_{y}^{2} \cap L_{x}^{2} H_{y}^{s}$ with $s>\frac{1}{2}$ [\ref{[2.5]}]. We state this result as follows.
\begin{theorem}
[\ref{[2.5]}, Theorem 1.6]
For any $U_0 \in H_{x}^{2s} L_{y}^{2} \cap L_{x}^{2} H_{y}^{s}$ with $s>\frac{1}{2}$, there exists $T_{\max} > 0$ and a unique local solution $U \in C\left(\left(-T_{\max }, T_{\max }\right) ; H_{x}^{2 s} L_{y}^{2} \cap L_{x}^{2} H_{y}^{s}\right)$ to (\ref{1.1}) with the initial data $U_0$. 
\end{theorem}
An adaptation of the arguments from [\ref{[2.1]}] implemented in [\ref{[2.2]}] implies the following ill-posedness result for (\ref{1.1}).
\begin{theorem}
[\ref{[2.2]}, Norm inflation]
Let $s<\frac{1}{4}$ and $\mathcal{H}^s : = H_{x}^{2 s} L_{y}^{2} \cap L_{x}^{2} H_{y}^{s}$. There exists a positive sequence $(t_n)_{n \in \mathbb{N}}$ tending to zero and a sequence $(u_n(t))_{n \in \mathbb{N}}$ of $C^{\infty}\left(\mathbb{R}^{2}\right)$ solutions to (\ref{1.1}) defined for $t \in\left[0, t_{n}\right]$, which satisfy
\begin{align*}
\left\|u_{n}(0)\right\|_{\mathcal{H}^s} \rightarrow 0
\end{align*}
and
\begin{align*}
\left\|u_{n}\left(t_{n}\right)\right\|_{\mathcal{H}^s} \rightarrow \infty
\end{align*}
as $n \to \infty$.
\end{theorem}
\subsubsection{The half wave Schrödinger equation on the cylinder}
Consider the half wave Schrödinger equation on the cylinder,
\begin{equation}
\label{1.2}
\left(i \partial_{t}+\mathcal{A}\right) U=|U|^{2} U, \qquad (x, y) \in \mathbb{R} \times \mathbb{T}.
\end{equation}
Inspired by the work of Kato-Pusateri [\ref{[2.201]}], Z. Hani, B. Pausader, N. Tzvetkov and N. Visciglia have proved modified scattering for the cubic Schrödinger equation on $\mathbb{R}\times\mathbb{T}^d$ [\ref{[7]}]. H. Xu has adapted the method in [\ref{[7]}] to establish a modified scattering theory between small decaying solutions to (\ref{1.2}) and small decaying solutions to the non-chiral cubic Szeg\H o equation (\ref{1.4}).

Before introducing H. Xu's result, we define the following norms,
\begin{align*}
\|F\|_{V}:=\|F\|_{H_{x, y}^{N'}}+\|x F\|_{L_{x, y}^{2}}, \quad\|F\|_{V^{+}}:=\|F\|_{V}+\left\|\left(1-\partial_{x}^2\right)^{4} F\right\|_{V}+\|x F\|_{V}. 
\end{align*}
The non chiral cubic Szeg\H o equation on $\mathbb{R} \times \mathbb{T}$ is stated as follows,
\begin{equation}
\label{1.4}
\begin{array}{l}i \partial_{t} G(t)=\mathcal{R}\left[G(t), G(t), G(t)\right], \\ \mathcal{F}_{x\rightarrow\xi} \mathcal R[G,G,G]:=
\Pi_+(|\widehat{G}_+|^2 \widehat{G}_+)+
\Pi_-(|\widehat {G}_-|^2\widehat{G}_-). \end{array}
\end{equation}
Here $\widehat{G}(\xi, \cdot)=\mathcal{F}_{\mathbb{R}_x} G(\xi, \cdot)$ with $\xi \in \mathbb{R}$, $\Pi_{+}$ is the Szeg\H o projector onto the non-negative Fourier modes on the variable $y$, $\Pi_{-}:=\mathrm{Id}-\Pi_{+}$, and $G_{\pm}:=\Pi_{\pm}(G)$. 

To obtain the following results, H. Xu assumed that the initial data satisfies
\begin{equation}
\label{1.71}
U_{0}(x, y+\pi)=-U_{0}(x, y).
\end{equation}
In fact, the modified scattering consists of two parts: the existence of modified wave operators and the asymptotic completeness. We introduce the existence of modified wave operators obtained by H. Xu in [\ref{[2]}] as follows.
\begin{theorem}
[\ref{[2]}, Theorem 1.4]
\label{theorem 1.2}
Given $N' \geq 13$, there exists $\varepsilon=\varepsilon(N')>0$ such that if $G_{0} \in V^{+}$ satisfies
\begin{align*}
\left\|G_{0}\right\|_{V^{+}} \leq \varepsilon
\end{align*}
and if $G(t)$ solves (\ref{1.4}) with initial data $G_0$, then there exists $U \in C([0, \infty): V)$ a solution of (\ref{1.2}) such that 
\begin{align*}
\left\|\mathrm{e}^{-i t \mathcal{A}} U(t) - G(\pi \ln t)\right\|_{V} \rightarrow 0 \text { as } t \rightarrow \infty.
\end{align*}
\end{theorem}
Also, H. Xu has obtained the following asymptotic completeness in [\ref{[2]}].
\begin{theorem}
[\ref{[2]}, Theorem 1.3]
\label{theorem 1.1}
Given $N' \geq 13$, there exists $\varepsilon=\varepsilon(N')>0$ such that if $U_{0} \in V^{+}$ satisfies
\begin{align*}
\left\|U_{0}\right\|_{V^{+}} \leq \varepsilon,
\end{align*}
and if $U(t)$ solves (\ref{1.2}) with initial data $U_0$, then $U \in C([0,+\infty): V)$ exists globally and exhibits modified scattering to its resonant dynamics (\ref{1.4}) in the following sense: there exists $G_{0} \in V$ such that if $G(t)$ is the solution of (\ref{1.4}) with initial data $G(0)=G_{0}$, then
\begin{align*}
\left\|\mathrm{e}^{-i t \mathcal{A}} U(t)-G(\pi \ln t)\right\|_{V} \rightarrow 0 \text { as } t \rightarrow \infty.
\end{align*}
\end{theorem}
Theorem {\ref{theorem 1.2}} and {\ref{theorem 1.1}} make up the modified scattering theory of (\ref{1.2}). 
H. Xu has also combined Theorem \ref{theorem 1.2} with Proposition \ref{Proposition 1.39} by P. Gérard and S. Grellier, and she has obtained the following infinite cascade result.
\begin{theorem}
[\ref{[2]}, Corollary 5.1]
\label{theorem 1.3}
Given $N' \geq 13$, then for any $\varepsilon > 0$, there exists $U_{0} \in V^{+}$ with $\left\|U_{0}\right\|_{V^{+}} \leq \varepsilon$, such that the corresponding solution to (\ref{1.2}) satisfies
\begin{align*}
\limsup _{t \rightarrow \infty} \frac{\|U(t)\|_{L_{x}^{2} H_{y}^{s}}}{(1+\log |t|)^{M}}=\infty, \quad \forall s>1 / 2, \quad \forall M \geq 0.
\end{align*}
\end{theorem}
\begin{remark}
\label{remark 1.12}
Theorem \ref{theorem 1.3} shows that the limit superior at infinity of $\|U(t)\|_{L_{x}^{2} H_{y}^{s}}$ with $s>\frac{1}{2}$ is infinity. However, we cannot deduce $ \liminf _{t \rightarrow \infty} \|U(t)\|_{L_{x}^{2} H_{y}^{s}} < \infty$ for any $s>0$ directly from Proposition \ref{Proposition 1.39}. Here we show the sketch of the proof of Theorem \ref{theorem 1.3} and we show also why we cannot infer $\liminf _{t \rightarrow \infty} \|U(t)\|_{L_{x}^{2} H_{y}^{s}} < \infty$.  

We choose a cutoff function $\psi \in C_c^\infty(\mathbb{R})$ as follows:  $\psi(\xi) = 1 $ when $|\xi| \leq 1$, $0<\psi(\xi)<1 $ when $1< |\xi| < 2 $ and $\psi(\xi)=0$ when $|\xi| \geq 2$. Let $\rho> 0$ be a constant to be determined, and we construct $\widehat{G}_0(\xi,y) : = \rho\, \psi(\xi) u_0(y)$ with $u_0(y)$ in a dense $G_\delta$ subset of $C_{+}^{\infty}(\mathbb{T})$ as in Proposition \ref{Proposition 1.39}. We verify that $G_0(x,y) \in V^{+}$, and for any $\varepsilon > 0$, we choose $\rho_{\varepsilon}$ small enouth to have $\|G_0(x,y)\|_{V^{+}} \leq \varepsilon$, which satisfies the condition of Theorem {\ref{theorem 1.2}}. Let $u(t,y)$ be the corresponding solution to (\ref{1.505}) with the initial data $u_0(y)$, then $G(t,x,y) : = \mathcal{F}_{\xi\to x}^{-1}\left(\rho_{\varepsilon}\psi(\xi) u(\rho_{\varepsilon}^2\psi(\xi)^2 t, y)\right)$ is the solution to
\begin{align*}
\left\{\begin{array}{l}i \partial_{t} G=\mathcal{R}[G, G, G], \quad (x,y) \in \mathbb{R} \times \mathbb{T}, \\ G(0)=G_{0}.\end{array}\right.
\end{align*}
For $s>\frac{1}{2}$ and $M > 0$, by Proposition \ref{Proposition 1.39}, we verify that
\begin{equation}
\label{1.105}
\begin{aligned}
\limsup_{t\to\infty}\frac{\|G(t,x,y)\|_{L_{x}^{2} H_{y}^{s}}}{|t|^M} & \geq \limsup_{t\to\infty} \frac{\left(\int_{|\xi|\leq 1}\|\widehat{G}(t, \xi,y)\|_{H_{y}^{s}}^{2} d \xi\right)^{\frac{1}{2}}}{|t|^M}\\ & =\sqrt{2}\rho_{\varepsilon}^{2M+1} \limsup_{t\to\infty}\frac{\|u(\rho_{\varepsilon}^2t,y)\|_{H_{y}^{s}}}{\rho_{\varepsilon}^{2M} |t|^M} = \infty.
\end{aligned}
\end{equation}
Combining (\ref{1.105}) with Theoerem \ref{theorem 1.2}, we deduce Theorem \ref{theorem 1.3}. 

Yet, the above method is not sufficient to prove $\liminf _{t \rightarrow \infty} \|G(t,x,y)\|_{L_{x}^{2} H_{y}^{s}} < \infty$. From Proposition \ref{Proposition 1.39}, we know that there exists a sequence \{$\underline{t}^n\}_n$ with $\underline{t}^n \to \infty$ such that $\lim_{n\to\infty}\|u(\underline{t}^n, y)\|_{H_{y}^s} < \infty$ for any $s > 0$. Then we have
\begin{equation}
\label{1.106}
\begin{aligned}
\liminf_{n\to\infty}\|G(\frac{\underline{t}^n}{\rho_{\varepsilon}^2},\xi,y)\|_{L_x^2 H_{y}^{s}}& = \liminf_{n\to\infty} \left(\int_{|\xi|< 2} \|\widehat{G}(\frac{\underline{t}^n}{\rho_{\varepsilon}^2},\xi,y)\|_{H_{y}^{s}}^2 d\xi\right)^{\frac{1}{2}}  \\ & \leq  \rho_{\varepsilon}\liminf_{n\to\infty} \left(\int_{|\xi|< 2}\|u(\psi(\xi)^2\underline{t}^n,y)\|_{H_{y}^{s}}^{2} d \xi\right)^{\frac{1}{2}} \\ &  \leq \sqrt{2}\rho_{\varepsilon} \lim_{n\to\infty} \|u(\underline{t}^n,y)\|_{H_{y}^{s}} + \rho_{\varepsilon} \liminf_{n\to\infty}\left(\int_{1<|\xi|< 2}\|u(\psi(\xi)^2\underline{t}^n,y)\|_{H_{y}^{s}}^{2} d \xi\right)^{\frac{1}{2}}  .
\end{aligned}
\end{equation}
From (\ref{1.106}), we cannot deduce $\liminf_{n\to \infty} \|G(\frac{\underline{t}^n}{\rho_{\varepsilon}^2},x,y)\|_{L_{x}^{2} H_{y}^{s}} < \infty$ because we do not even know whether  $\liminf_{n\to\infty}\|u(\psi(\xi)^2\underline{t}^n,y)\|_{H_{y}^{s}}^{2} < \infty$ when $1<|\xi|<2$.
\end{remark}
\begin{remark}
Let $u(t,y)$ be the corresponding solution to (\ref{1.505}) with the initial data $u_0(y)$ in a dense $G_\delta$ subset of $C_{+}^{\infty}(\mathbb{T})$ as in Proposition \ref{Proposition 1.39}. In fact, from [\ref{[10]}], we have
\begin{equation}
\label{1.10611}
\limsup _{T \rightarrow+\infty} \frac{1}{T} \int_{0}^{T}\|u(t)\|_{H^{1}} d t=+\infty.
\end{equation}
Formula (\ref{1.10611}) indicates the relative length of the time intervals where Sobolev norms of $u$ are large is large enough. Let $\varrho>0$ be a constant to be determined, and we take $\chi(\xi):= \varrho\, \mathrm{e}^{-\xi^2/2}$. We construct $\widehat{G}_0(\xi,y) : = \chi(\xi) u_0(y)$, and we verify that $G_0(x,y) \in V^{+}$. For any $\varepsilon > 0$, we choose $\varrho_{\varepsilon}$ small enouth to satisfy $\|G_0(x,y)\|_{V^{+}} \leq \varepsilon$. Then we have $G(t,x,y) : = \mathcal{F}_{\xi\to x}^{-1}\left(\chi(\xi) u(\chi(\xi)^2 t, y)\right)$ is the solution to
\begin{align*}
\left\{\begin{array}{l}i \partial_{t} G=\mathcal{R}[G, G, G], \quad (x,y) \in \mathbb{R} \times \mathbb{T}, \\ G(0)=G_{0}.\end{array}\right.
\end{align*}
We have
\begin{align*}
& \frac{1}{T} \int_{0}^{T}\|G(t)\|_{L_x^2 H_y^{1}}^2 d t\\ = & \frac{1}{T} \int_{0}^{T} \int_{\mathbb{R}}\chi(\xi)^2 \|u(\chi(\xi)^2 t, y)\|_{H_y^1}^2 d\xi dt \\= & \frac{1}{T}  \int_{0}^{\varrho_{\varepsilon}^2 T} \|u(s, y)\|_{H_y^1}^2 \left(\int_{\chi(\xi)^2\geq\frac{s}{T}} d\xi\right)ds \\ = &  \frac{2}{T} \int_{0}^{\varrho_{\varepsilon}^2 T} \|u(s, y)\|_{H_y^1}^2 \left(\log \left(\frac{ \varrho_{\varepsilon}^2 T}{s}\right)\right)^{\frac{1}{2}}ds \\ \gtrsim & \frac{1}{T} \int_{0}^{\frac{\varrho_{\varepsilon}^2 T}{2}} \|u(s, y)\|_{H_y^1}^2ds \\ \gtrsim & \, \varrho_{\varepsilon}^2\left(\frac{2}{\varrho_{\varepsilon}^2 T}  \int_{0}^{\frac{\varrho_{\varepsilon}^2 T}{2}} \|u(s, y)\|_{H_y^1} ds\right)^2.
\end{align*}
By (\ref{1.10611}), we deduce that 
\begin{equation}
\label{1.10612}
\limsup _{T \rightarrow+\infty} \frac{1}{T} \int_{0}^{T}\|G(t)\|_{L_x^2 H_y^{1}}^2 d t = +\infty.
\end{equation}
Formula (\ref{1.10612}) indicates the relative length of the time intervals where Sobolev norms of $G$ are large is large enough. However, (\ref{1.10612}) is not sufficient to prove $\liminf _{t \rightarrow \infty} \|G(t)\|_{L_{x}^{2} H_{y}^{s}} = \infty$ for any $s > 0$, so we cannot deduce that the corresponding solution $U$ to (\ref{1.2}) satisfies $ \liminf _{t \rightarrow \infty} \|U(t)\|_{L_{x}^{2} H_{y}^{s}} = \infty$ for any $s>0$ either.
\end{remark}
\subsection{Our main results}
The aim of this paper is to prove the existence of modified wave operators and corresponding cascade result for the half wave Schrödinger equation (\ref{1.1}) with small decaying data. In our case, we set $N \geq 3$ is an arbitrary integer, and we define the following norms,
\begin{align*}
& \|F\|_{Z}:=\|\widehat{F}(\xi, y)\|_{L_{\xi}^\infty\mathcal{G}_y},\\ & \|F\|_{S}:=\|F\|_{H_{x, y}^{N}} + \|x F\|_{L_x^2 H_y^2}, \\ & \|F\|_{\mathcal{Y}}:= \|F\|_{S} + \| F\|_{Z}, \\ & \|F\|_{S^{+}}:=\|F\|_{S}+ \|xF\|_{L_x^2 H_y^3} +\|\left(1+|D_x|\right) F\|_{S}+\|x F\|_{S},
\\ & \|F\|_{\mathcal{Y}^{+}}:=\|F\|_{S^+}+\|F\|_Z + \|\left(1+|D_x|\right) F\|_{Z}+\| x F\|_{Z} +\| x^2 F\|_{Z},
\end{align*}
where $\mathcal{G}:= L^2(\mathbb{R}) \bigcap \dot{B}_{1,1}^{1}(\mathbb{R})$. \\\\
We introduce the following non chiral cubic Szeg\H o equation on $\mathbb{R} \times \mathbb{R}$, 
\begin{equation}
\label{1.7}
\begin{array}{l}i \partial_{t} G(t)=\mathcal{R}\left[G(t), G(t), G(t)\right], \\ \mathcal{F}_{x\rightarrow\xi} \mathcal R[G,G,G]:=
\Pi_+(|\widehat{G}_+|^2 \widehat{G}_+)+
\Pi_-(|\widehat {G}_-|^2\widehat{G}_-). \end{array}
\end{equation}
Here $\widehat{G}(\xi, \cdot)=\mathcal{F}_{\mathbb{R}_x} G(\xi, \cdot)$ with $\xi \in \mathbb{R}$, $\Pi_{+}$ is the Szeg\H o projector onto the non-negative Fourier modes in the variable $y$, $\Pi_{-}:=\mathrm{Id}-\Pi_{+}$, and $G_{\pm}:=\Pi_{\pm}(G)$.\\\\
Our first result provides the existence of modified wave operators.
\begin{theorem}
\label{Theorem 1.5}
Given $N \geq 3$, there exists $\varepsilon(N) > 0$ such that if $G_{0} \in \mathcal{Y}^{+}$ satisfies 
\begin{equation}
\left\|G_{0}\right\|_{\mathcal{Y}^{+}} \leq \varepsilon.
\end{equation}
If $G$ is the solution of (\ref{1.7}) with initial data $G_0$, then there exists a unique solution $U$ of (\ref{1.1}) such that $\mathrm{e}^{-i t \mathcal{A}} U(t) \in C([0, \infty): \mathcal{Y})$ and 
\begin{align*}
\left\|\mathrm{e}^{-i t \mathcal{A}} U(t)-G(\pi \ln t)\right\|_{\mathcal{Y}} \rightarrow 0 \text { as } t \rightarrow \infty.
\end{align*}
\end{theorem}
Combining Theorem {\ref{Theorem 1.5}} with the Proposition \ref{Proposition 1.4}, we deduce the following cascade result.
\begin{theorem}
\label{Theorem 1.6}
Given $N \geq 3$, then for any $\varepsilon > 0$, there exists $U_{0} \in \mathcal{Y}^{+}$ with $\left\|U_{0}\right\|_{\mathcal{Y}^{+}} \leq \varepsilon$, such that the corresponding solution to (\ref{1.1}) satisfies
\begin{equation}
0<\liminf _{t \rightarrow \infty} \frac{\|U(t)\|_{L_{x}^{2} H_{y}^{1}}}{(1+\log |t|)} \leq  \limsup _{t \rightarrow \infty} \frac{\|U(t)\|_{L_{x}^{2} H_{y}^{1}}}{(1+\log |t|)} < +\infty.
\end{equation}
\end{theorem}
\begin{remark}
\label{remark 1.15}
Unlike Theorem \ref{theorem 1.3}, Theorem \ref{Theorem 1.6} shows that $\|U(t)\|_{L_{x}^{2} H_{y}^{1}}$ tends to infinity as $\log t$ when $t\to\infty$. This shows that the half wave Schrödinger equation on the plane is one of the very few dispersive equations admitting global solutions with small and smooth data such that the $H^s$ norms are going to infinity at infinity.
\end{remark}
\begin{remark}
In fact, $N \geq 3$ in Theorem \ref{Theorem 1.5} and Theorem \ref{Theorem 1.6} is the optimal restriction on $N$ in our proof.
\end{remark}The proof of Theorem \ref{Theorem 1.5} is not only an adaptation of methods in [\ref{[2]}], there are many essential differences between the proof of Theorem \ref{Theorem 1.5} and the proof of Theorem \ref{theorem 1.2}. In fact, H. Xu considered the $B_{1,1}^1(\mathbb{T})$ norm on the variable $y$ because the $B_{1,1}^1(\mathbb{T})$ norm is a conserved quantity for the resonant system on $\mathbb{T}$. Also, the property $\Pi_{\pm}(B_{1,1}^1(\mathbb{T})) \subset B_{1,1}^1(\mathbb{T})$ is used in the decomposition on the frequencies with the variable $y$ in [\ref{[2]}]. However, we cannot consider $B_{1,1}^1(\mathbb{R})$ on the variable $y$ because the $B_{1,1}^1(\mathbb{R})$ norm on the variable $y$ is not a conserved quantity for the resonant system on $\mathbb{R}$. By Peller's theorem in [\ref{[5]}], we know that the $\dot{B}_{1,1}^1(\mathbb{R})$ seminorm is a conserved quantity for the cubic Szeg\H o equation on $\mathbb{R}$. To have a conserved norm for the resonant system, we consider the norm $\mathcal{G} = L^2(\mathbb{R}) \bigcap \dot{B}_{1,1}^{1}(\mathbb{R}) $ on the variable $y$, which is a conserved quantity for the resonant system. This means we deal with a essentially different norm from the norm in [\ref{[2]}].  The space $\mathcal{G}$ satisfies $\Pi_{\pm}(\mathcal{G}) \subset \mathcal{G} \subset L^{\infty}(\mathbb{R})$, and we use this property in the decomposition on the frequencies in the variable $y$ in our proof. Also, the $B_{1,1}^1(\mathbb{T})$ norm can be estimated directly by the $H^s(\mathbb{T})$ norm with $s>1$ because $L^2(\mathbb{T}) \subset L^1(\mathbb{T})$, and this is the reason why H. Xu only needs to estimate the Sobolev norm and the $x$ weighted $L_{x,y}^2$ norm of the non-linearity. But we do not have $L^2(\mathbb{R}) \subset L^1(\mathbb{R})$, so we cannot estimate $\mathcal{G}$ directly by $H^s(\mathbb{R})$ for any $s>0$. In this situation, we should estimate the $Z$ norm of the non-linearity directly. To control the $Z$ norm of the non-linearity, we need to use the $x$ weighted $L_x^2 H_y^2$ norm, so we also need to estimate the $x$ weighted $L_x^2 H_y^2$ norm of the non-linearity. One can see the proof of Lemma \ref{Lemma 2.1}, Lemma \ref{Lemma 3.5} and Lemma \ref{Lemma 3.7} for the direct estimates on the $Z$ norm and on the $x$ weighted $L_x^2 H_y^2$ norm. Meanwhile, since we do not have $L^2(\mathbb{R}) \subset L^1(\mathbb{R})$, we cannot construct an intermediary norm to control the $Z$ norm as H. Xu did in [\ref{[2]}], and it may be the reason why we could not get the asymptotic completeness as Theorem \ref{theorem 1.1} for this time, which is another part of modified scattering. In Section 3, we decompose the non-linearity $\mathcal{N}^{t}$ into a combination of the resonance zero part $\mathcal{N}_0^{t}$ and resonance non-zero part $\widetilde{\mathcal{N}}^{t}$, just as H. Xu did in [\ref{[2]}]. When H. Xu estimated the part $\widetilde{\mathcal{N}}^{t}$, she used the $L^\infty$ norm estimate to get directly the boundedness because the resonance level changes discretely on $\mathbb{T}$. But in our case, we cannot use the same method because the resonance level changes continuously on $\mathbb{R}$, so we should use the $L^2$ norm estimate and we get the extra $t^{\frac{1}{2}}$ in the estimate. Thanks to the dispersive estimate on the $x$ direction, we get $t^{-\frac{1}{2}}$ decay in general which is sufficient for our estimate. One can see the proof of Lemma \ref{Lemma 3.5} for details, which is quite different from the proof in [\ref{[2]}]. Furthermore, we improve the restriction on the regularity of Sobolev norm $H_{x,y}^N$, which is $N \geq 3$ and which is better than the restriction on the regularity of $H_{x,y}^{N'}$($N' \geq 13$) in [\ref{[2]}]. Finally, as we mention in Remark \ref{remark 1.12} and Remark \ref{remark 1.15}, the cascade result we obtain shows that the half wave Schrödinger equation on the plane is one of the few dispersive equations admitting global solutions with small and smooth data such that the $H^s$ norm are going to infinity at infinity, but it is still unknown if there exists such a global solution to the half wave Schrödinger equation on the cylinder satisfying this property.
\subsection{Structure of the paper}
In Section 2, we introduce the notation in this paper. In Section 3, we establish the decomposition of the non-linearity $\mathcal{N}^{t}$, 
\begin{align*}
\mathcal{N}^{t}[F, G, H]=\frac{\pi}{t} \mathcal{R}[F, G, H]+\mathcal{E}^{t}[F, G, H],
\end{align*}
where $\mathcal{R}$ is the resonant part and $\mathcal{E}^{t}$ is the remainder. We give the decay estimate of $\mathcal{E}^{t}$, which is fundamental in the proof of Theorem \ref{Theorem 1.5}. In Section 4, we study the solution to the resonant system, which is equivalent to the non-chiral cubic Szeg\H o equation (\ref{1.7}). By the estimate of the resonant part $\mathcal{R}$, we give the estimate of the solution to the resonant system with respect to the initial data, which is also fundamental in the proof of Theorem \ref{Theorem 1.5}. In Section 5, we construct the modified wave operators and prove Theorem \ref{Theorem 1.5}. Later in this section, we prove the corresponding cascade result, Theorem \ref{Theorem 1.6}. In the Appendix, we introduce the transfer lemma which allows us to transfer $L_{x,y}^2$ estimates on operators into $S'$ estimates on operators. Furthermore, we introduce a lemma which allows us to transfer $L_{x}^2$ estimates on operators into $S$ estimates on operators.
\section{Preliminaries}
\subsection{Notation}
We define the Fourier trasform on $\mathbb{R}$ by
\begin{align*}
\widehat{g}(\xi):=\mathcal{F}_{x}(g)(\xi)=\frac{1}{2\pi}\int_{\mathbb{R}} \mathrm{e}^{-i x \xi} g(x) d x.
\end{align*}
Similarly, we also define the Fourier transform on $\mathbb{R}_y$ by
\begin{align*}
h_{\eta}:=\mathcal{F}_{y}(h)(\eta)=\frac{1}{2\pi}\int_{\mathbb{R}} e^{-i\eta y} h(y) d y.
\end{align*}
Then we define the full Fourier transform on $\mathbb{R}_x \times \mathbb{R}_y$ by
\begin{align*}
(\mathcal{F} U)(\xi, \eta)=\frac{1}{2\pi}\int_{\mathbb{R}} \widehat{U}(\xi, y) \mathrm{e}^{-i \eta y} d y=\widehat{U}_{\eta}(\xi).
\end{align*}
We also introduce Littlewood-Paley projections. We define Littlewood-Paley projections on the full frequencies by
\begin{align*}
\left(\mathcal{F} P_{\leq 2^k} U\right)(\xi, \eta)=\varphi\left(\frac{\xi}{2^k}\right) \varphi\left(\frac{\eta}{2^k}\right)(\mathcal{F} U)(\xi, \eta), 
\end{align*}
where $k \in \mathbb{Z}$, and $\varphi \in C_{c}^{\infty}(\mathbb{R})$ with $\varphi(x)=1$ when $|x| \leq 1$ and $\varphi(x)=0 \text { when }|x| \geq 2$. We also define
\begin{equation}
\phi(x)=\varphi(x)-\varphi(2 x)
\end{equation}
and
\begin{equation}
P_{2^k}=P_{\leq 2^k}-P_{\leq 2^{k-1}}, \quad P_{\geq 2^k}=1-P_{\leq 2^k}.
\end{equation}
Sometimes we only treat the frequency in $x$, so we define
\begin{align*}
(\mathcal{F} Q_ {\leq 2^k} U)(\xi, \eta)=\varphi\left(\frac{\xi}{2^k}\right)(\mathcal{F} U)(\xi, \eta),
\end{align*}
and we define $Q_{2^k}$ similarly. We also define Littlewood-Paley projections on the frequency in $y$
by
\begin{equation}
\left(\mathcal{F}_{y} \Delta_{2^k} h\right)(\eta)=\phi\left(\frac{\eta}{2^k}\right) h_{\eta}.
\end{equation}
\subsection{Duhamel formula}
We define 
\begin{align*}
F := \mathrm{e}^{-i t \mathcal{A}} U(x, y, t),
\end{align*}
where $U$ is a solution to (\ref{1.1}) and $\mathcal{A}=\partial_{x }^2-\left|D_{y}\right|$. We observe that $U$ solves (\ref{1.1}) if and only if $F$ solves 
\begin{equation}
\label{2.5}
i \partial_{t} F(t)=\mathrm{e}^{-i t \mathcal{A}}\left(\mathrm{e}^{i t \mathcal{A}} F(t) \cdot \mathrm{e}^{-i t \mathcal{A}} \overline{F(t)} \cdot \mathrm{e}^{i t \mathcal{A}} F(t)\right).
\end{equation}
We denote the non-linearity in (\ref{2.5}) by 
\begin{align*}
\mathcal{N}^{t}[F, G, H]:=\mathrm{e}^{-i t \mathcal{A}}\left(\mathrm{e}^{i t \mathcal{A}} F \cdot \mathrm{e}^{-i t \mathcal{A}} \overline{G} \cdot \mathrm{e}^{i t \mathcal{A}} H\right),
\end{align*}
where $\mathcal{N}^{t}$ is a trilinear operator.\\\\
Then we have the following Fourier transform expression
\begin{equation}
\label{2.6}
\mathcal{F} \mathcal{N}^{t}[F, G, H](\xi, \eta)=\int_{\mathbb{R}^2} \mathrm{e}^{i t(|\eta|-|\eta-\eta_1|+|\eta_2-\eta_1|-|\eta_2|)} \mathcal{F}_{x}\left(\mathcal{I}^{t}\left[F_{\eta-\eta_1}, G_{\eta_2-\eta_1}, H_{\eta_2}\right]\right)(\xi)d\eta_1 d\eta_2,
\end{equation}
where
\begin{equation}
\label{2.65}
\mathcal{I}^{t}[f, g, h]:=\mathcal{U}(-t) \left(\mathcal{U}(t) f \, \overline{\mathcal{U}(t) g}\, \mathcal{U}(t) h \right), \quad \mathcal{U}(t)=\mathrm{e}^{i t \partial_{x}^2}.
\end{equation}
We observe that
\begin{align*}
\mathcal{F}_{x}\left(\mathcal{I}^{t}[f, g, h]\right)(\xi)=\int_{\mathbb{R}^{2}} \mathrm{e}^{i t 2 \mu \kappa} \widehat{f}(\xi-\mu) \overline{\widehat{g}}(\xi-\mu-\kappa) \widehat{h}(\xi-\kappa) d \kappa d \mu,
\end{align*}
and $\mathcal{I}^{t}$ is also a trilinear operator.\\\\
We also have
\begin{equation}
\label{2.7}
\mathcal{N}^t[F,G,H] = \mathrm{e}^{i t |D_y|} \mathcal{I}^t [\mathrm{e}^{-i t |D_y|}F,\mathrm{e}^{-i t |D_y|}G,\mathrm{e}^{-i t |D_y|}H].
\end{equation}
\begin{remark}
We observe that all trilinear operators that we consider in this paper saitisfy (\ref{6.2}), so we can use Lemma \ref{Lemma 6.2} and Lemma \ref{Lemma 6.21} to estimate these trilinear operators.
\end{remark}
\subsection{Norms}
The homogeneous Besov space $\dot{B}^1(\mathbb{R}): = \dot{B}_{1,1}^{1}(\mathbb{R})$ is defined as the set of all $f \in \mathcal{S}^{\prime}\left(\mathbb{R}^{n}\right)$ such that $\|f\|_{\dot{B}_{1,1}^{1}(\mathbb{R})}$ is finite, where
\begin{align*}
\|f\|_{\dot{B}_{1,1}^{1}(\mathbb{R})}=\sum_{k\in\mathbb{Z}} 2^k \left\|\Delta_{2^k} f\right\|_{L^{1}(\mathbb{R})}.
\end{align*}
We notice that $\|\cdot\|_{\dot{B}_{1,1}^{1}(\mathbb{R})}$ is a norm on $\mathcal{S}^{\prime}\left(\mathbb{R}^{n}\right) / \mathcal{P}\left(\mathbb{R}^{n}\right)$, where $\mathcal{P}\left(\mathbb{R}^{n}\right)$ denotes the set of polynomials on $\mathbb{R}^n$. Then we define the following space
\begin{equation}
\mathcal{G}:= L^2(\mathbb{R}) \bigcap \dot{B}_{1,1}^{1}(\mathbb{R})
\end{equation}
with the norm
\begin{align*}
\|f\|_{\mathcal{G}} : = \|f\|_{L^2(\mathbb{R})} + \|f\|_{\dot{B}_{1,1}^{1}(\mathbb{R})}.
\end{align*}
We observe that $\mathcal{G}$ is a Banach space, and $\Pi_{\pm}(\mathcal{G}) \subset \mathcal{G} \subset L^{\infty}$.
For functions $F$ defined on $\mathbb{R} \times \mathbb{R}$, we will use mainly the following five norms:
\begin{equation}
\|F\|_{Z}:=\|\widehat{F}(\xi, y)\|_{L_{\xi}^\infty\mathcal{G}_y},
\end{equation}
\begin{equation}
\|F\|_{S'}:=\|F\|_{H_{x, y}^{N}} + \|x F\|_{L_{x,y}^2},
\end{equation}
\begin{equation}
\|F\|_{S}:=\|F\|_{H_{x, y}^{N}} + \|x F\|_{L_x^2 H_y^2},
\end{equation}
\begin{equation}
\|F\|_{\mathcal{Y}}:=\|F\|_{S} + \| F\|_{Z},
\end{equation}
\begin{equation}
\label{2.12}
\|F\|_{S^{+}}:=\|F\|_{S}+ \|xF\|_{L_x^2 H_y^3} +\|\left(1+|D_x|\right) F\|_{S}+\|x F\|_{S},
\end{equation}
\begin{equation}
\|F\|_{\mathcal{Y}^{+}}:=\|F\|_{S^+}+\|F\|_Z + \|\left(1+|D_x|\right) F\|_{Z}+\| x F\|_{Z} +\| x^2 F\|_{Z},
\end{equation}
with $N \geq 3$.

We also define the following space-time norm
\begin{equation}
\begin{aligned}\|F\|_{X_{T}} &:=\sup _{0 \leq t \leq T}\left\{\|F(t)\|_{Z}+(1+|t|)^{-\delta}\|F(t)\|_{\mathcal{Y}}+(1+|t|)^{1-3 \delta}\left\|\partial_{t} F(t)\right\|_{\mathcal{Y}}\right\}, \\\|F\|_{X_{T}^{+}} &:=\|F\|_{X_{T}}+\sup _{0 \leq t \leq T}(1+|t|)^{- \delta}\|F(t)\|_{\mathcal{Y}^{+}}, \end{aligned}
\end{equation}
with $\delta < 10^{-4}$.

Now we introduce a lemma which will be very useful in the latter parts.
\begin{lemma}
\label{Lemma 2.1}
For $|t|\geq \frac{1}{2}$, we have
\begin{equation}
\label{2.16}
\left\|\mathcal{N}^{t}[F, G, H]\right\|_{\mathcal{Y}} \lesssim \frac{1}{|t|}\|F\|_{S}\|G\|_{S}\|H\|_{S}.
\end{equation}
In particular,
\begin{equation}
\label{2.16001}
\begin{aligned}
\left\|x \, \mathcal{N}^{t}[F, G, H]\right\|_{L_x^2 H_y^2} \lesssim & \,\frac{1}{|t|} \|xF\|_{L_x^2 H_y^2} \|G\|_{L_x^2 H_y^2 }^{\frac{1}{2}}\|xG\|_{L_x^2 H_y^2 }^{\frac{1}{2}} \|H\|_{L_x^2 H_y^2 }^{\frac{1}{2}} \|xH\|_{ L_x^2 H_y^2 }^{\frac{1}{2}} \\ + & \,\frac{1}{|t|} \|xG\|_{L_x^2 H_y^2} \|F\|_{L_x^2 H_y^2 }^{\frac{1}{2}}\|xF\|_{L_x^2 H_y^2 }^{\frac{1}{2}} \|H\|_{L_x^2 H_y^2 }^{\frac{1}{2}} \|xH\|_{ L_x^2 H_y^2 }^{\frac{1}{2}}\\ + & \,\frac{1}{|t|} \|xH\|_{L_x^2 H_y^2} \|F\|_{L_x^2 H_y^2 }^{\frac{1}{2}}\|xF\|_{L_x^2 H_y^2 }^{\frac{1}{2}} \|G\|_{L_x^2 H_y^2 }^{\frac{1}{2}} \|xG\|_{ L_x^2 H_y^2 }^{\frac{1}{2}} \\ \lesssim & \, \frac{1}{|t|} \|F\|_S \|G\|_S \|H\|_S,
\end{aligned}
\end{equation}
\begin{equation}
\label{2.1601}
\begin{aligned}
& \left\|\mathcal{N}^{t}[F, G, H]\right\|_{Z}\\  \lesssim & \, \frac{1}{|t|} \|F\|_{L_x^2 H_y^2 }^{\frac{1}{4}}\|xF\|_{L_x^2 H_y^2 }^{\frac{1}{4}} \|G\|_{L_x^2 H_y^2 }^{\frac{1}{4}}\|xG\|_{L_x^2 H_y^2 }^{\frac{1}{4}} \|H\|_{L_x^2 H_y^2 }^{\frac{1}{4}} \|xH\|_{ L_x^2 H_y^2 }^{\frac{1}{4}} \|F\|_S^{\frac{1}{2}}\|G\|_S^{\frac{1}{2}}\|H\|_S^{\frac{1}{2}}\\ \lesssim & \, \frac{1}{|t|} \|F\|_S \|G\|_S \|H\|_S.
\end{aligned}
\end{equation}
\end{lemma}
\begin{proof}
Let
\begin{align*}
F = F^1, G = F^2, H = F^3.
\end{align*}
Then we recall (\ref{2.7}),
\begin{align*}
\mathcal{N}^t[F^1,F^2,F^3] = \mathrm{e}^{i t |D_y|} \mathcal{I}^t [\mathrm{e}^{-i t |D_y|}F^1,\mathrm{e}^{-i t |D_y|}F^2,\mathrm{e}^{-i t |D_y|}F^3].
\end{align*}
In fact, we have
\begin{equation}
\label{2.21}
\begin{aligned}
& \left\|\mathcal{I}^{t}\left[F^{1}, F^{2}, F^{3}\right]\right\|_{L_{x}^{2}} \\ = &  \left\| \mathrm{e}^{-i t \partial_x^2}\left( \mathrm{e}^{i t \partial_x^2} F^1 \cdot \mathrm{e}^{-i t \partial_x^2} \overline{F^2} \cdot \mathrm{e}^{i t \partial_x^2} F^3\right)\right\|_{L_{x}^2} \\ \lesssim & \left\|\mathrm{e}^{i t \partial_{x}^2} F^1 \cdot \mathrm{e}^{-i t \partial_{x }^2}\overline{F^2} \cdot \mathrm{e}^{i t \partial_{x}^2} F^3 \right\|_{L_{x}^2} \\ \lesssim & \min _{\{j, k, \ell\}=\{1,2,3\}} \left\| F^j \right\|_{L_{x}^2} \left\|\mathrm{e}^{i t \partial_{x }^2}F^k\right\|_{L_{x}^\infty} \left\| \mathrm{e}^{i t \partial_{x}^2} F^\ell \right\|_{L_{x}^\infty}.
\end{aligned}
\end{equation}
For $|t|\geq \frac{1}{2}$, we use the dispersive estimate and we have
\begin{equation}
\label{2.20}
\left\|\mathrm{e}^{i t \partial_{x}^2} f\right\|_{L_{x}^{\infty}} \lesssim|t|^{-\frac{1}{2}}\|f\|_{L_{x}^{1}} \lesssim|t|^{-\frac{1}{2}}\|f\|_{L_{x}^{2}}^{\frac{1}{2}}\|x f\|_{L_{x}^{2}}^{\frac{1}{2}}.
\end{equation}
So we have 
\begin{equation}
\label{2.201}
\left\|\mathcal{I}^{t}\left[F^{1}, F^{2}, F^{3}\right]\right\|_{L_{x}^{2}} \lesssim \frac{1}{|t|}\min _{\{j, k, \ell\}=\{1, 2, 3\}}\left\|F^{j}\right\|_{L_{x}^{2}}\left\|F^{k}\right\|_{L_x^2}^{\frac{1}{2}}\left\|x F^{k}\right\|_{L_x^2}^{\frac{1}{2}} \left\|F^{\ell}\right\|_{L_x^2}^{\frac{1}{2}}\left\|x F^{\ell}\right\|_{L_x^2}^{\frac{1}{2}}.
\end{equation}
Then by Lemma \ref{Lemma 6.21}, we have 
\begin{equation}
\label{2.1701}
\begin{aligned}
& \left\|\mathcal{N}^{t}\left[F^{1}, F^{2}, F^{3}\right]\right\|_{L_{x, y}^{2}} \\ \lesssim & \, \left\|\mathcal{I}^t [\mathrm{e}^{-i t |D_y|}F^1,\mathrm{e}^{-i t |D_y|}F^2,\mathrm{e}^{-i t |D_y|}F^3]\right\|_{L_{x,y}^2}\\ \lesssim & \, \frac{1}{|t|} \min _{\{j, k, \ell\}=\{1, 2, 3\}}\left\|F^{j}\right\|_{L_{x,y}^{2}}\left\|F^{k}\right\|_{L_x^2H_y^1}^{\frac{1}{2}}\left\|x F^{k}\right\|_{L_x^2H_y^1}^{\frac{1}{2}} \left\|F^{\ell}\right\|_{L_x^2H_y^1}^{\frac{1}{2}}\left\|x F^{\ell}\right\|_{L_x^2H_y^1}^{\frac{1}{2}} \\ \lesssim & \, \frac{1}{|t|}\min _{\{j, k, \ell\}=\{1,2,3\}} \left\| F^j \right\|_{L_{x,y}^2} \left\|F^k\right\|_S\left\|F^\ell\right\|_S
\end{aligned}
\end{equation}
and
\begin{equation}
\label{2.1702}
\left\|\mathcal{N}^{t}\left[F^{1}, F^{2}, F^{3}\right]\right\|_{S} \lesssim \frac{1}{|t|} \left\| F^1 \right\|_{S} \left\|F^2\right\|_S\left\|F^3\right\|_S.
\end{equation}
In particular, by Lemma \ref{Lemma 6.21}, we also have
\begin{align*}
& \left\|x \, \mathcal{N}^{t}\left[F, G, H\right]\right\|_{L_x^2 H_y^2} \\ \lesssim & \left\|x\,\mathcal{I}^t [\mathrm{e}^{-i t |D_y|}F,\mathrm{e}^{-i t |D_y|}G,\mathrm{e}^{-i t |D_y|}H]\right\|_{L_x^2 H_y^2} \\ \lesssim &  \,\frac{1}{|t|} \|xF\|_{L_x^2 H_y^2} \|G\|_{L_x^2 H_y^2 }^{\frac{1}{2}}\|xG\|_{L_x^2 H_y^2 }^{\frac{1}{2}} \|H\|_{L_x^2 H_y^2 }^{\frac{1}{2}} \|xH\|_{ L_x^2 H_y^2 }^{\frac{1}{2}} \\ + & \,\frac{1}{|t|} \|xG\|_{L_x^2 H_y^2} \|F\|_{L_x^2 H_y^2 }^{\frac{1}{2}}\|xF\|_{L_x^2 H_y^2 }^{\frac{1}{2}} \|H\|_{L_x^2 H_y^2 }^{\frac{1}{2}} \|xH\|_{ L_x^2 H_y^2 }^{\frac{1}{2}}\\ + & \,\frac{1}{|t|} \|xH\|_{L_x^2 H_y^2} \|F\|_{L_x^2 H_y^2 }^{\frac{1}{2}}\|xF\|_{L_x^2 H_y^2 }^{\frac{1}{2}} \|G\|_{L_x^2 H_y^2 }^{\frac{1}{2}} \|xG\|_{ L_x^2 H_y^2 }^{\frac{1}{2}}    \\ \lesssim & \, \frac{1}{|t|} \|F\|_S \|G\|_S \|H\|_S,
\end{align*}
which implies (\ref{2.16001}).

Then we estimate $\left\|\mathcal{N}^{t}[F, G, H]\right\|_{Z}$. We firstly estimate $\left\|\mathcal{F}_{x\rightarrow\xi}\mathcal{N}^{t}[F, G, H]\right\|_{L_{\xi}^\infty L_y^2}$. By (\ref{2.201}) and Remark \ref{remark 6.4}, we have
\begin{equation}
\label{2.241}
\begin{aligned}
& \left\|\mathcal{F}_{x\rightarrow\xi}\mathcal{N}^{t}[F, G, H]\right\|_{L_{\xi}^\infty L_y^2}  \\ \lesssim & \,\frac{1}{|t|}
\|F\|_{L_x^2 H_y^2 }^{\frac{1}{4}}\|xF\|_{L_x^2 H_y^2 }^{\frac{1}{4}} \|G\|_{L_x^2 H_y^2 }^{\frac{1}{4}}\|xG\|_{L_x^2 H_y^2 }^{\frac{1}{4}} \|H\|_{L_x^2 H_y^2 }^{\frac{1}{4}} \|xH\|_{ L_x^2 H_y^2 }^{\frac{1}{4}} \|F\|_S^{\frac{1}{2}}\|G\|_S^{\frac{1}{2}}\|H\|_S^{\frac{1}{2}} \\ \lesssim & \, \frac{1}{|t|} \|F\|_S \|G\|_S \|H\|_S.
\end{aligned}
\end{equation}
Then we estimate the term $\|\mathcal{F}_{x\rightarrow\xi}\mathcal{N}^{t}[F, G, H]\|_{L_{\xi}^\infty \dot{B}_y^1}$. By (\ref{2.7}), (\ref{2.201}) and (\ref{6.21}) in Remark \ref{remark 6.6}, we have
\begin{equation}
\label{2.251}
\begin{aligned}
 &  \|\mathcal{F}_{x\rightarrow\xi}{\mathcal{N}^{t}}[F, G, H](\xi,y)\|_{L_{\xi}^\infty \dot{B}_y^1 } \\ \lesssim & \, \frac{1}{|t|} \left(\sum_{k \leq 0} 2^k \left\|\int_{\mathbb{R}}\mathrm{e}^{i y \eta} \mathrm{e}^{it |\eta|} \phi(\frac{\eta}{2^k}) d\eta \right\|_{L_y^1} + \sum_{k > 0} 2^{-k} \left\|\int_{\mathbb{R}}\mathrm{e}^{i y \eta} \mathrm{e}^{it |\eta|} \phi(\frac{\eta}{2^k}) d\eta \right\|_{L_y^1}\right) \times \\ & \|F\|_{L_x^2 H_y^2 }^{\frac{1}{4}}\|xF\|_{L_x^2 H_y^2 }^{\frac{1}{4}} \|G\|_{L_x^2 H_y^2 }^{\frac{1}{4}}\|xG\|_{L_x^2 H_y^2 }^{\frac{1}{4}} \|H\|_{L_x^2 H_y^2 }^{\frac{1}{4}} \|xH\|_{ L_x^2 H_y^2 }^{\frac{1}{4}} \|F\|_S^{\frac{1}{2}}\|G\|_S^{\frac{1}{2}}\|H\|_S^{\frac{1}{2}} \\ \lesssim & \, \frac{1}{|t|} \|F\|_{L_x^2 H_y^2 }^{\frac{1}{4}}\|xF\|_{L_x^2 H_y^2 }^{\frac{1}{4}} \|G\|_{L_x^2 H_y^2 }^{\frac{1}{4}}\|xG\|_{L_x^2 H_y^2 }^{\frac{1}{4}} \|H\|_{L_x^2 H_y^2 }^{\frac{1}{4}} \|xH\|_{ L_x^2 H_y^2 }^{\frac{1}{4}} \|F\|_S^{\frac{1}{2}}\|G\|_S^{\frac{1}{2}}\|H\|_S^{\frac{1}{2}}.
\end{aligned}
\end{equation}
Then by (\ref{2.241}) and (\ref{2.251}), we have
\begin{align*}
& \left\|\mathcal{N}^{t}[F, G, H]\right\|_{Z} \\ \lesssim & \, \frac{1}{|t|} \|F\|_{L_x^2 H_y^2 }^{\frac{1}{4}}\|xF\|_{L_x^2 H_y^2 }^{\frac{1}{4}} \|G\|_{L_x^2 H_y^2 }^{\frac{1}{4}}\|xG\|_{L_x^2 H_y^2 }^{\frac{1}{4}} \|H\|_{L_x^2 H_y^2 }^{\frac{1}{4}} \|xH\|_{ L_x^2 H_y^2 }^{\frac{1}{4}} \|F\|_S^{\frac{1}{2}}\|G\|_S^{\frac{1}{2}}\|H\|_S^{\frac{1}{2}}\\ \lesssim & \, \frac{1}{|t|} \, \|F\|_{S} \|G\|_S \|H\|_S,
\end{align*}
which imples (\ref{2.1601}). Finally, by  (\ref{2.1601}) and (\ref{2.1702}), we obtain (\ref{2.16}). The proof is complete. 
\end{proof}
\section{Structure of the nonlinearity}
In this section, we are to obtain the decomposition for the full non-linearity in (\ref{2.5}), which can be written as
\begin{equation}
\mathcal{N}^{t}[F, G, H]=\frac{\pi}{t} \mathcal{R}[F, G, H]+\mathcal{E}^{t}[F, G, H],
\end{equation}
where $\mathcal{R}$ is the resonant part,
\begin{equation}
\mathcal{F} \mathcal{R}[F, G, H](\xi, \eta)=\int_{\omega(\eta, \eta_1, \eta_2) = 0 }  \widehat{{F}}_{\eta-\eta_1}(\xi)\overline{\widehat{G}}_{\eta_2-\eta_1}(\xi)\widehat{{H}}_{\eta_2}(\xi)d\eta_1 d\eta_2,
\end{equation}
and
\begin{align*}
\omega(\eta, \eta_1, \eta_2): = |\eta|-|\eta-\eta_1|+|\eta_2-\eta_1|-|\eta_2|.
\end{align*}
We call $\mathcal{E}^t$ the remainder term, which will be estimated in Proposition \ref{Proposition 3.1}.
\begin{remark}
\label{remark 3.1}
We have that $\omega\left(\eta, \eta_{1}, \eta_{2}\right)=0$ in the following cases (see the proof of Proposition \ref{Proposition 4.2} for details):
\begin{center}
$\begin{array}{l}\text { If } \eta>0 \text { and }(\eta_1,\eta_2) \in\left\{ \eta \geq \eta_1, \eta_2 \geq \eta_1, \eta_2 \geq 0 \right\} \cup\{\eta_1 = 0\}\cup \{\eta = \eta_2\}, \\ \text { If } \eta<0 \text { and }(\eta_1, \eta_2) \in\left\{ \eta \leq \eta_1, \eta_2 \leq \eta_1, \eta_2 \leq 0 \right\} \cup\{\eta_1 = 0\}\cup \{\eta = \eta_2\}. \end{array}$
\end{center}
Here for any $\eta \in \mathbb{R}$, the sets $\{(\eta_1,\eta_2) \in \mathbb{R}^2 | \eta_1 = 0 \} $ and $\{(\eta_1,\eta_2) \in \mathbb{R}^2 | \eta = \eta_2\}$ are of measure zero in $\mathbb{R}^2$, they do not interfere in the integration in equation, so we can neglect them.
\end{remark}
Then we introduce our main result in this section.
\begin{proposition}
\label{Proposition 3.1}
$\text { We assume for } T \geq 1, F, G, H: \mathbb{R} \rightarrow \mathcal{Y} \text { satisfy }$
\begin{equation}
\label{3.3}
\|F\|_{X_{T}}+\|G\|_{X_{T}}+\|H\|_{X_{T}} \lesssim 1.
\end{equation}
Then for $t \in [\frac{T}{2}, T]$ we write
\begin{align*}
\mathcal{E}^{t}[F(t), G(t), H(t)] & =\mathcal{E}_{1}^{t}[F(t), G(t), H(t)]+\mathcal{E}_{2}^{t}[F(t), G(t), H(t)], \\
\mathcal{E}_{2}^{t}[F(t), G(t), H(t)] & = \partial_{t} \mathcal{E}_{3}^{t}[F(t), G(t), H(t)].
\end{align*}
$\text {We note } \mathcal{E}_{j}(t):=\mathcal{E}_{j}^{t}[F(t), G(t), H(t)]$ for $j=1,2,3$, then we have the following estimates which hold uniformly in $T \geq 1$,
\begin{align*}
& \quad T^{-3\delta}\left\|\int_{\frac{T}{2}}^{T} \mathcal{E}_{j}(t) d t\right\|_{\mathcal{Y}} \lesssim 1, \quad j=1,2, \\ & \sup _{\frac{T}{2} \leq t \leq T}(1+|t|)^{1+20\delta}\left\|\mathcal{E}_{1}(t)\right\|_{Z} \lesssim 1, \\ & \sup _{\frac{T}{2} \leq t \leq T}(1+|t|)^{\frac{1}{3}}\left\|\mathcal{E}_{3}(t)\right\|_{\mathcal{Y}} \lesssim 1.
\end{align*}
Moreover, with the assumption
\begin{equation}
\label{3.4}
\|F\|_{X_{T}^{+}}+\|G\|_{X_{T}^{+}}+\|H\|_{X_{T}^{+}} \leq 1,
\end{equation}
we also have the following estimate which holds uniformly in $T \geq 1$,
\begin{align*}
T^{20 \delta}\left\|\int_{\frac{T}{2}}^{T} \mathcal{E}_{j}(t) d t\right\|_{\mathcal{Y}} \lesssim 1, \quad j=1,2.
\end{align*}
\end{proposition}
\subsection{The high frequency estimates for $\mathcal{Y}$}
In this subsection, we are to obtain the decay estimate for the nonlinearity in the regime with at least one high frequency. We adapt the energy estimate when two inputs have high frequencies. On the other hand, we use the bilinear refinements of the Strichartz estimate on $\mathbb{R}$. Firstly we introduce the following lemma which will be used in the proof of Lemma \ref{Lemma 3.3}.
\begin{lemma}
[\ref{[3]}]
\label{lemma 3.2}
 Assume that $\lambda / 10 \geq \mu \geq 1$ and that $u(t)=\mathrm{e}^{i t \partial_{x x}} u_{0}, v(t)=\mathrm{e}^{i t \partial_{x x}} v_{0}$. Then we have the bound
 \begin{equation}
\left\|Q_{\lambda} u \overline{Q_{\mu} v}\right\|_{L_{x, t}^{2}(\mathbb{R} \times \mathbb{R})} \lesssim \lambda^{-\frac{1}{2}}\left\|u_{0}\right\|_{L_{x}^{2}(\mathbb{R})}\left\|v_{0}\right\|_{L_{x}^{2}(\mathbb{R})}.
 \end{equation}
\end{lemma}
We refer to [\ref{[3]}] for the proof.\\

Then we prove the following lemma which gives a decay estimate on $\mathcal{N}^{t}\left[Q_{A} F(t),Q_B G(t), Q_C H(t)\right]$ in the regime $\max (A, B, C) \geq T^{\frac{1}{6}}$ with $T \geq 1$.
\begin{lemma}
\label{Lemma 3.3}
The following estimates hold for $T \geq 1$,
\begin{equation}
\label{3.6}
\left\|\sum_{\max (A, B, C) \geq T^{\frac{1}{6}}} \mathcal{N}^{t}\left[Q_{A} F, Q_{B} G, Q_{C} H\right]\right\|_{Z} \lesssim T^{-\frac{25}{24}}\|F\|_{S}\|G\|_{S}\|H\|_{S}, \quad \forall\, t\geq\frac{T}{2},
\end{equation}
\begin{equation}
\left\|\sum_{\max (A, B, C) \geq T^{\frac{1}{6}}} \int_{\frac{T}{2}}^T \mathcal{N}^{t}\left[Q_{A} F, Q_{B} G, Q_{C} H\right]\right\|_{S} \lesssim T^{-\frac{1}{20}}\|F\|_{X_T}\|G\|_{X_T}\|H\|_{X_T}.
\end{equation}
\end{lemma}
\begin{proof}
Let $t\geq\frac{T}{2}$. For (\ref{3.6}), according to (\ref{2.1601}) and Bernstein's inequality, we have
\begin{align*}
& \left\| \sum_{\max (A, B, C) \geq T^{\frac{1}{6}}} \mathcal{N}^{t}[F, G, H]\right\|_{Z} \\ \lesssim & \quad T^{-1} \sum_{\max (A, B, C) \geq T^{\frac{1}{6}}} \|Q_A F\|_{L_x^2 H_y^2 }^{\frac{1}{4}} \|Q_B G\|_{L_x^2 H_y^2 }^{\frac{1}{4}} \|Q_C H\|_{L_x^2 H_y^2 }^{\frac{1}{4}} \| F\|_{S}^{\frac{3}{4}}\| G\|_{S}^{\frac{3}{4}} \| H\|_{S}^{\frac{3}{4}}   \\ \lesssim & \quad T^{-1} \sum_{\max (A, B, C) \geq T^{\frac{1}{6}}} (ABC)^{-\frac{1}{4}} \|F\|_{S} \|G\|_S \|H\|_S \\ \lesssim & \quad T^{-\frac{25}{24}} \|F\|_S \|G\|_S \|H\|_S.
\end{align*}
For another estimate, firstly we split the set $\left\{(A, B, C): \max (A, B, C) \geq T^{\frac{1}{6}}\right\}$ into two parts $\Lambda$ and $\Lambda^c$. Here $\Lambda:=\left\{(A, B, C): \operatorname{med}(A, B, C) \leq T^{\frac{1}{6}} / 16, \max (A, B, C) \geq T^{\frac{1}{6}}\right\}$, with $\operatorname{med}(A, B, C)$ denotes the second largest number among $(A, B, C)$.

We start with the case $(A, B, C) \in \Lambda^{c}$, we are to prove
\begin{equation}
\label{3.9}
\sum_{(A, B, C) \in \Lambda^{c}} \left\|\mathcal{N}^{t}\left[Q_{A} F, Q_{B} G, Q_{C} H\right]\right\|_{S} \lesssim T^{-\frac{13}{12}}\left\| F\right\|_{S}\left\| G\right\|_{S}\left\| H \right\|_{S},\quad \forall t\geq\frac{T}{2}.
\end{equation}
Firstly, we give the estimate for $\left\|\sum_{(A, B, C) \in \Lambda^{c}} \mathcal{N}^{t}\left[Q_{A} F, Q_{B} G, Q_{C} H\right]\right\|_{L_{x,y}^{2}}$. For $t\geq\frac{T}{2}$, by (\ref{2.1701}), we have
\begin{align*}
& \left\|\sum_{(A, B, C) \in \Lambda^{c}} \mathcal{N}^{t}\left[Q_{A} F, Q_{B} G, Q_{C} H\right]\right\|_{L_{x,y}^{2}} \\ \lesssim & \quad T^{-1} \sum_{(A, B, C) \in \Lambda^{c}} \|F\|_{L_{x,y}^2} \|Q_B G\|_{L_x^2 H_y^1}^{\frac{1}{2}} \|x Q_B G\|_{L_x^2 H_y^1}^{\frac{1}{2}} \|Q_C H\|_{L_x^2 H_y^1}^{\frac{1}{2}} \|x Q_C H\|_{L_x^2 H_y^1}^{\frac{1}{2}} \\ \lesssim & \quad T^{-1} \sum_{(A, B, C) \in \Lambda^{c}}(B C)^{-1}\left\|F\right\|_{L_{x, y}^{2}}\left\| G\right\|_{S}\left\|H\right\|_{S} \\ \lesssim & \quad T^{-1} \left(\sum_{(A, B, C) \in \Lambda^{c}}(\operatorname{med}(A, B, C))^{-1}\right)\|F\|_{L_{x, y}^{2}}\|G\|_{S}\|H\|_{S} \\ \lesssim & \quad T^{-\frac{7}{6}}\|F\|_{L_{x, y}^{2}}\|G\|_{S}\|H\|_{S}.
\end{align*}
The inequality above holds by replacing $F$ with $G, H$. According to Lemma {\ref{Lemma 6.2}}, we have
\begin{align*}
\left\|\sum_{(A, B, C) \in \Lambda^{c}} \mathcal{N}^{t}\left[Q_{A} F, Q_{B} G, Q_{C} H\right]\right\|_{S'} \lesssim T^{-\frac{7}{6}} \|F\|_{S} \|G\|_S \|H\|_S.
\end{align*}
Then we estimate $\left\|\sum_{(A, B, C) \in \Lambda^{c}} x \, \mathcal{N}^{t}\left[Q_{A} F, Q_{B} G, Q_{C} H\right]\right\|_{L_{x}^{2} H_y^2 }$. According to (\ref{2.16001}), we have
\begin{align*}
& \left\|\sum_{(A, B, C) \in \Lambda^{c}} x \, \mathcal{N}^{t}\left[Q_{A} F, Q_{B} G, Q_{C} H\right]\right\|_{L_{x}^{2} H_y^2 }\\ \lesssim & \quad T^{-1} \sum_{(A, B, C)\in \Lambda^{c}} \|xQ_A F\|_{L_x^2 H_y^2} \|Q_B G\|_{L_x^2 H_y^2 }^{\frac{1}{2}}\|xQ_B G\|_{L_x^2 H_y^2 }^{\frac{1}{2}} \|Q_C H\|_{L_x^2 H_y^2 }^{\frac{1}{2}} \|xQ_C H\|_{ L_x^2 H_y^2 }^{\frac{1}{2}} \\ + & \quad  T^{-1} \sum_{(A, B, C)\in \Lambda^{c}} \|xQ_B G\|_{L_x^2 H_y^2} \|Q_A F\|_{L_x^2 H_y^2 }^{\frac{1}{2}}\|xQ_A F\|_{L_x^2 H_y^2 }^{\frac{1}{2}} \|Q_C H\|_{L_x^2 H_y^2 }^{\frac{1}{2}} \|xQ_C H\|_{ L_x^2 H_y^2 }^{\frac{1}{2}}\\ + & \quad  T^{-1} \sum_{(A, B, C)\in \Lambda^{c}} \|xQ_C H\|_{L_x^2 H_y^2} \|Q_A F\|_{L_x^2 H_y^2 }^{\frac{1}{2}}\|xQ_A F\|_{L_x^2 H_y^2 }^{\frac{1}{2}} \|Q_B G\|_{L_x^2 H_y^2 }^{\frac{1}{2}} \|xQ_B G\|_{ L_x^2 H_y^2 }^{\frac{1}{2}} \\ \lesssim & \quad T^{-1} \left(\sum_{(A, B, C) \in \Lambda^{c}}(\operatorname{med}(A, B, C))^{-\frac{1}{2}}\right) \|F\|_S \|G\|_S \|H\|_S \\ \lesssim & \quad T^{-\frac{13}{12}} \|F\|_S \|G\|_S \|H\|_S. 
\end{align*}
So we have proved (\ref{3.9}).\\\\
Then we consider the case $(A, B, C) \in \Lambda $, we are to prove
\begin{equation}
\label{3.11}
\begin{aligned}
& \left\| \sum_{\substack{A, B, C \\(A, B, C) \in \Lambda}} \int_{\frac{T}{2}}^{T} \mathcal{N}^{t}\left[Q_{A} F(t), Q_{B} G(t), Q_{C} H(t)\right] d t \right\|_{S} \\ & \lesssim  T^{-\frac{1}{20}}\|F\|_{X_{T}}\|G\|_{X_{T}}\|H\|_{X_{T}}.
\end{aligned}
\end{equation}
We consider a decomposition
\begin{equation}
[T / 4,2 T]=\bigcup_{j \in J} I_{j}, \quad I_{j}=\left[j T^{\frac{9}{10}},(j+1) T^{\frac{9}{10}}\right]=\left[t_{j}, t_{j+1}\right], \quad \# J \lesssim T^{\frac{1}{10}}.
\end{equation}
We also consider $\chi \in C_{c}^{\infty}(\mathbb{R}), \chi \geq 0$ such that $\chi(s)=0 \text { if }|s| \geq 2$ and 
\begin{align*}
\sum_{k \in \mathbb{Z}} \chi(s-k) \equiv 1.
\end{align*}
We can estimate the left hand side of (\ref{3.11}) by $C(E_1+E_2)$, where
\begin{align*} E_{1}=\bigg\| \sum_{j \in J} \sum_{{(A, B, C)} \in \Lambda} \int_{\frac{T}{2}}^{T} \chi\left(\frac{t}{T^{\frac{9}{10}}}- j\right) & \bigg(\mathcal{N}^{t}\left[Q_{A} F(t), Q_{B} G(t), Q_{C} H(t)\right]\\ &   -\mathcal{N}^{t}\left[Q_{A} F\left(t_{j}\right), Q_{B} G\left(t_{j}\right), Q_{C} H\left(t_{j}\right)\right] \bigg) d t \bigg\|_{S} \end{align*}
and 
\begin{align*}
E_{2}=\left\|\sum_{j \in J} \sum_{(A, B, C) \in \Lambda} \int_{\frac{T}{2}}^{T} \chi\left(\frac{t}{T^{\frac{9}{10}}}-j\right) \mathcal{N}^{t}\left[Q_{A} F\left(t_{j}\right), Q_{B} G\left(t_{j}\right), Q_{C} H\left(t_{j}\right)\right] d t\right\|_{S}.
\end{align*}
We start by estimating $E_1$, we have
\begin{equation}
E_{1} \leq \sum_{j \in J} \int_{\frac{T}{2}}^{T} \chi\left(\frac{t}{T^{\frac{9}{10}}}-j\right) E_{1, j}(t) d t,
\end{equation}
with 
\begin{align*}
E_{1, j}(t):=  \left\|\sum_{(A, B, C) \in \Lambda}\left(\mathcal{N}^{t}\left[Q_{A} F(t), Q_{B} G(t), Q_{C} H(t)\right]-\mathcal{N}^{t}\left[Q_{A} F\left(t_{j}\right), Q_{B} G\left(t_{j}\right), Q_{C} H\left(t_{j}\right)\right]\right)\right\|_{S}.
\end{align*}
We note $Q_{+}:=Q_{\geq T^{\frac{1}{6}}} \text { and } Q_{-}:=Q_{\leq T^{\frac{1}{6}} / 16}$, then we have
\begin{align*}
\sum_{(A, B, C) \in \Lambda} \mathcal{N}^{t}\left[Q_{A} F, Q_{B} G, Q_{C} H\right]& =\mathcal{N}^{t}\left[Q_{+} F, Q_{-} G, Q_{-} H\right] +\mathcal{N}^{t}\left[Q_{-} F, Q_{+} G, Q_{-} H\right]\\  & +  \mathcal{N}^{t}\left[Q_{-} F, Q_{-} G, Q_{+} H\right]. 
\end{align*}
We rerrange the terms in $E_{1,j}$ two by two and we rewrite the first pair as follows,
\begin{align*}
\mathcal{N}^{t} & {\left[Q_{+} F(t), Q_{-} G(t), Q_{-} H(t)\right]-\mathcal{N}^{t}\left[Q_{+} F\left(t_{j}\right), Q_{-} G\left(t_{j}\right), Q_{-} H\left(t_{j}\right)\right] } \\ & =\mathcal{N}^{t}\left[Q_{+}\left(F(t)-F\left(t_{j}\right)\right), Q_{-} G(t), Q_{-} H(t)\right]+\mathcal{N}^{t}\left[Q_{+} F\left(t_{j}\right), Q_{-}\left(G(t)-G \left(t_{j}\right)\right), Q_{-} H(t)\right] \\ &+\mathcal{N}^{t}\left[Q_{+} F\left(t_{j}\right), Q_{-} G\left(t_{j}\right), Q_{-}\left(H(t)-H\left(t_{j}\right)\right)\right]. 
\end{align*}
Then by Lemma \ref{Lemma 2.1}, we have
\begin{align*}
\left\|\mathcal{N}^{t}\left[Q_{+}\left(F(t)-F\left(t_{j}\right)\right), Q_{-} G(t), Q_{-} H(t)\right]\right\|_{S} \lesssim(1+|t|)^{-1}\left\|F(t)-F\left(t_{j}\right)\right\|_{S}\|G(t)\|_{S}\|H(t)\|_{S}.
\end{align*}
Other terms can be estimated similarly, so we have
\begin{equation}
\begin{aligned} E_{1, j}(t) \lesssim &(1+|t|)^{-1}\left[\left\|F(t)-F\left(t_{j}\right)\right\|_{S}\|G(t)\|_{S}\|H(t)\|_{S}\right.\\ &+\left\|F\left(t_{j}\right)\right\|_{S}\left\|G(t)-G\left(t_{j}\right)\right\|_{S}\|H(t)\|_{S} \\ &\left.+\left\|F\left(t_{j}\right)\right\|_{S}\left\|G\left(t_{j}\right)\right\|_{S}\left\|H(t)-H\left(t_{j}\right)\right\|_{S}\right]. \end{aligned}
\end{equation}
Since $\left|t-t_{j}\right| \leq T^{\frac{9}{10}}$, we have
\begin{align*}
\left\|F(t)-F\left(t_{j}\right)\right\|_{S} \leq \int_{t_{j}}^{t}\left\|\partial_{t} F(\theta)\right\|_{S} d \theta \leq T^{\frac{9}{10}} \sup _{t}\left\|\partial_{t} F(t)\right\|_{S}.
\end{align*}
Then by the definition of $X_T$, we have
\begin{align*}
\left\|F(t)-F\left(t_{j}\right)\right\|_{S} \leq T^{-\frac{1}{10}+3 \delta}\|F\|_{X_{T}},
\end{align*}
and
\begin{align*}
\|F(t)\|_{S} \leq T^{\delta}\|F\|_{X_{T}}.
\end{align*}
Thus,
\begin{align*}
E_{1, j} \lesssim T^{-\frac{11}{10}+5 \delta}\|F\|_{X_{T}}\|G\|_{X_{T}}\|H\|_{X_{T}}.
\end{align*}
So we have
\begin{align*}
E_{1} \lesssim \int_{T / 2}^{T} \sum_{j \in J} \chi\left(\frac{t}{T \frac{9}{10}}-j\right) E_{1, j(t)} d t \lesssim T^{-\frac{1}{10}+5 \delta}\|F\|_{X_{T}}\|G\|_{X_{T}}\|H\|_{X_{T}}.
\end{align*}
Then we estimate $E_2$. We denote
\begin{align*}
E_{2, j}^{A, B, C}=\left\|\int_{\frac{T}{2}}^{T} \chi\left(\frac{t}{T^{\frac{9}{10}}}-j\right) \mathcal{N}^{t}\left[Q_{A} F\left(t_{j}\right), Q_{B} G\left(t_{j}\right), Q_{C} H\left(t_{j}\right)\right] d t\right\|_{S},
\end{align*}
we have
\begin{align*}
E_{2} \leq \sum_{j \in J} \sum_{(A, B, C) \in \Lambda} E_{2, j}^{A, B, C}.
\end{align*}
Let
\begin{align*}
F = F^a,  G = F^b, H = F^c,
\end{align*}
then we are to prove 
\begin{equation}
\label{3.15}
\begin{aligned}
& \left\|\int_{\frac{T}{2}}^{T} \chi\left(\frac{t}{T^{\frac{9}{10}}}-j\right) \mathcal{I}^{t}\left[Q_{A} F^{a}(t_j), Q_{B} F^{b}(t_j), Q_{C} F^{c}(t_j)\right] d t\right\|_{L_{x}^{2}}  \\ \lesssim & \, (\max (A, B, C))^{-1} \left\|F^{a}\right\|_{L_{x}^{2}}\left\|F^{b}\right\|_{L_{x}^{2}}\left\|F^{c}\right\|_{L_{x}^{2}}.\end{aligned}
\end{equation}
To prove (\ref{3.15}), we take $K \in L_{x}^2$ and $F^a, F^b, F^c$ independent on $t$, then we have
\begin{align*}
I_K & = \left\langle K, \int_{\frac{T}{2}}^{T} \chi\left(\frac{t}{T^{\frac{9}{10}}}-j\right) \mathcal{I}^{t}\left[Q_{A} F^{a}, Q_{B} F^{b}, Q_{C} F^{c}\right] d t\right\rangle_{L_{x}^{2}} \\ & =    \int_{\frac{T}{2}}^{T}\int_{\mathbb{R}} \chi\left(\frac{t}{T^{\frac{9}{10}}}-j\right) \mathrm{e}^{i t \partial_{x}^2}(Q_{A} F^{a})  \overline{\mathrm{e}^{i t \partial_{x}^2}(Q_{B} F^{b}}) \mathrm{e}^{i t \partial_{x}^2}(Q_{C} F^{c})\overline{\mathrm{e}^{i t \partial_{x}^2} K} dxdt.
\end{align*}
Without loss of generality, we assume that $K=Q_{D} K, D \simeq \max (A, B, C)$ and $A=\max (A, B, C)$. By Cauchy-Schwarz inequality, we have
\begin{align*} & \left|\int_{\frac{T}{2}}^{T} \int_{\mathbb{R}} \chi\left(\frac{t}{T^{\frac{9}{10}}}-j\right) \mathrm{e}^{i t \partial_{x}^2}\left(Q_{A} F^{a}\right) \overline{\mathrm{e}^{i t \partial_{x}^2}\left(Q_{B} F^{b}\right)} \mathrm{e}^{i t \partial_{x}^2}\left(Q_{C} F^{c}\right) \overline{\mathrm{e}^{i t \partial_{x}^2} Q_{D} K} d x d t\right| \\ \leq &\left\|\mathrm{e}^{i t \partial_{x}^2}\left(Q_{A} F^{a}\right) \overline{\mathrm{e}^{i t \partial_{x}^2}\left(Q_{B} F^{b}\right)}\right\|_{L_{x, t}^{2}}\left\|\mathrm{e}^{i t \partial_{x}^2}\left(Q_{C} F^{c}\right) \overline{\mathrm{e}^{i t \partial_{x}^2} Q_{D} K}\right\|_{L_{x, t}^{2}}.\end{align*}
Since $A \geq 16 B, D \geq 16 C$, by Lemma \ref{lemma 3.2}, we have
\begin{align*}
\left\|\mathrm{e}^{i t \partial_{x}^2}\left(Q_{A} F^{a}\right) \overline{\mathrm{e}^{i t \partial_{x}^2}\left(Q_{B} F^{b}\right)}\right\|_{L_{x, t}^{2}} & \lesssim A^{-1 / 2}\left\|F^{a}\right\|_{L_{x}^{2}}\left\|F^{b}\right\|_{L_{x}^{2}},\\
\left\|\mathrm{e}^{i t \partial_{x}^2}\left(Q_{C} F^{c}\right) \overline{\mathrm{e}^{i t \partial_{x}^2} Q_{D} K}\right\|_{L_{x, t}^{2}} & \lesssim D^{-1 / 2}\left\|F^{c}\right\|_{L_{x}^{2}}\left\|K\right\|_{L_{x}^{2}}.
\end{align*}
By duality, we deduce (\ref{3.15}). Then by Remark \ref{remark 6.41}, we have
\begin{align*}
E_{2, j}^{A, B, C} \lesssim(\max (A, B, C))^{-1}\|F\|_{S}\|G\|_{S}\|H\|_{S}.
\end{align*}
Then we deduce
\begin{align*}
E_{2} \leq \sum_{j \in J} \sum_{(A, B, C) \in \Lambda} E_{2, j}^{A, B, C} \lesssim \# J \sum_{(A, B, C) \in \Lambda}(\max (A, B, C))^{-1}\|F\|_{S}\|G\|_{S}\|H\|_{S}.
\end{align*}
We recall that $A=\max (A, B, C)$, so we have
\begin{align*}
\sum_{(A, B, C) \in \Lambda}(\max (A, B, C))^{-1}=\left(\sum_{A \geq T^{1 / 6}} A^{-1}\right)\left(\#\left\{B: B \leq T^{1 / 6} / 16\right\}\right)^{2} \lesssim T^{-1 / 6+\delta}.
\end{align*}
By the definition (\ref{2.12}), we have
\begin{align*}
\left\|F\left(t_{j}\right)\right\|_{S}\left\|G\left(t_{j}\right)\right\|_{S}\left\|H\left(t_{j}\right)\right\|_{S} \leq T^{3 \delta}\|F\|_{X_{T}}\|G\|_{X_{T}}\|H\|_{X_{T}},
\end{align*}
thus
\begin{equation}
E_{2} \lesssim T^{-1 / 15+4\delta}\|F\|_{X_{T}}\|G\|_{X_{T}}\|H\|_{X_{T}}.
\end{equation}
Thus we complete the proof of (\ref{3.11}). We have finished the proof of Lemma \ref{Lemma 3.3}.
\end{proof}
Thus afterwards we can suppose that the $x$ frequencies of $F,G,H$ are $\leq T^{\frac{1}{6}}$. To go to the next subsection, we introduce the following decomposition
\begin{equation}
\mathcal{N}^{t}[F, G, H]=\mathcal{N}_{0}^{t}[F, G, H]+\widetilde{\mathcal{N}}^{t}[F, G, H],
\end{equation}
where
\begin{equation}
\label{3.18}
\mathcal{F} \mathcal{N}_{0}^{t}(\xi, \eta):=  \int_{\omega(\eta,\eta_1,\eta_2) = 0} \mathcal{F}_{x}\left(\mathcal{I}^{t}\left[F_{\eta-\eta_{1}}, G_{\eta_2 -\eta_1}, H_{\eta_2}\right]\right)(\xi) d \eta_{1} d \eta_{2}.
\end{equation}
\subsection{The fast oscillations}
In this subsection, we treat the contribution of $\widetilde{\mathcal{N}}^t[F,G,H]$.
We present firstly two elementary estimates here.
\begin{lemma}
\label{Lemma 3.4}
$\text { Let } \frac{1}{p}=\frac{1}{q}+\frac{1}{r}+\frac{1}{s} \text { with } 1 \leq p, q, r, s \leq \infty, \text { then it holds }$
\begin{equation}
\label{3.181}
\begin{aligned}
&\left\|\int_{\mathbb{R}^{3}} \mathrm{e}^{i x \xi} m(\mu, \kappa) \widehat{f}(\xi-\mu) \overline{\widehat{g}}(\xi-\mu-\kappa) \widehat{h}(\xi-\kappa) d \mu d \kappa d \xi\right\|_{L_{x}^{p}}\\ \lesssim & \int_{\mathbb{R}^{2}}\left|\mathcal{F}_{\mu,\kappa}^{-1} m(y, z)\right|\|f(x-z) \overline{g}(x-y-z) h(x-y)\|_{L_{x}^{p}} d y d z \\ \lesssim & \left\|\mathcal{F}_{\mu,\kappa}^{-1} m\right\|_{L^{1}\left(\mathbb{R}^{2}\right)}\|f\|_{L^{q}}\|g\|_{L^{r}}\|h\|_{L^{s}}.
\end{aligned}
\end{equation}
\end{lemma}
\begin{proof}
We have
\begin{align*}  I =& \int_{\mathbb{R}^{3}} \mathrm{e}^{i x \xi} m(\mu, \kappa) \widehat{f}(\xi-\mu) \overline{\widehat{g}}(\xi-\mu-\kappa) \widehat{h}(\xi-\kappa) d \mu d \kappa d \xi \\=& \frac{1}{4\pi^2}\int_{\mathbb{R}^2} \mathcal{F}_{\mu,\kappa}^{-1} m(y,z) \int_{\mathbb{R}^3} \mathrm{e}^{ix\xi}\mathrm{e}^{-iz(\xi-\mu)}\widehat{f}(\xi-\mu) \mathrm{e}^{i(z+y)(\xi-\mu-\kappa)}\overline{\widehat{g}}(\xi-\mu-\kappa)\mathrm{e}^{-iy(\xi-\kappa)}\widehat{h}(\xi-\kappa)d\mu d\kappa d\xi dy dz \\  = & \frac{1}{4\pi^2}\int_{\mathbb{R}^{2}} \mathcal{F}_{\mu,\kappa}^{-1} m(y, z) f(x-z) \overline{g}(x-y-z) h(x-y) d y d z,
\end{align*}
then by Hölder’s inequality, we have
\begin{align*}\|I\|_{L_{x}^{p}} & \lesssim \int_{\mathbb{R}^{2}}\left|\mathcal{F}_{\mu,\kappa}^{-1} m(y, z)\right|\|f(x-z) \overline{g}(x-y-z) h(x-y)\|_{L_{x}^{p}} d y d z \\ &\lesssim \left(\int_{\mathbb{R}^{2}}\left|\mathcal{F}_{\mu,\kappa}^{-1} m(y, z)\right| d y d z \right) \|f\|_{L^{q}}\|g\|_{L^{r}}\|h\|_{L^{s}}  \\& = \left\|\mathcal{F}_{\mu,\kappa}^{-1} m\right\|_{L^{1}\left(\mathbb{R}^{2}\right)}\|f\|_{L^{q}}\|g\|_{L^{r}}\|h\|_{L^{s}}. 
\end{align*}
\end{proof}
To simplify the notation afterwards, we define the following shift operator,
\begin{equation}
T_y f(x) : = f(x-y).
\end{equation}
\begin{remark}
In fact, the inequality we use in the proof of Lemma \ref{Lemma 3.5} is the inequality (\ref{3.181}) with $p=2$. In this case we have 
\begin{equation}
\label{3.191}
\begin{aligned}
&\left\|\int_{\mathbb{R}^2} m(\mu, \kappa) \widehat{f}(\xi-\mu) \overline{\widehat{g}}(\xi-\mu-\kappa) \widehat{h}(\xi-\kappa)d\mu d\kappa\right\|_{L_{\xi}^2} \\ \lesssim &\left\|\int_{\mathbb{R}^{3}} \mathrm{e}^{i x \xi} m(\mu, \kappa) \widehat{f}(\xi-\mu) \overline{\widehat{g}}(\xi-\mu-\kappa) \widehat{h}(\xi-\kappa) d \mu d \kappa d \xi\right\|_{L_{x}^{2}} \\ \lesssim & \int_{\mathbb{R}^{2}}\left|\mathcal{F}_{\mu,\kappa}^{-1} m(y, z)\right|\|f(x-z) \overline{g}(x-y-z) h(x-y)\|_{L_{x}^{2}} d y d z \\ \lesssim & \left\|\mathcal{F}_{\mu,\kappa}^{-1} m\right\|_{L^{1}\left(\mathbb{R}^{2}\right)}\|f\|_{L^{q}}\|g\|_{L^{r}}\|h\|_{L^{s}}.
\end{aligned}
\end{equation}
\end{remark}
\begin{remark}
\label{remark 3.7}
Let 
\begin{equation}
\mathcal{M}^t[F,G,H](\xi) : = \int_{\mathbb{R}^{2}} \mathrm{e}^{2it \mu \kappa} m(\mu, \kappa)  \widehat{F}(\xi-\mu) \overline{\widehat{G}}(\xi-\mu-\kappa) \widehat{H}(\xi-\kappa) d \mu d \kappa.
\end{equation}
We observe that 
\begin{align*}
\mathcal{M}^t[F,G,H] (\xi) = \mathrm{e}^{it|\xi|^2}\int_{\mathbb{R}^2}m(\mu,\kappa) \widehat{\mathrm{e}^{it\partial_x^2}F}(\xi-\mu) \overline{\widehat{\mathrm{e}^{it\partial_x^2} G}}(\xi-\mu-\kappa) \widehat{\mathrm{e}^{it\partial_x^2} H}(\xi-\kappa) d \mu d \kappa.
\end{align*}
Then by (\ref{2.21}) and (\ref{3.191}), we have
\begin{equation}
\begin{aligned}
& \left\|\mathcal{M}^t[f^1,f^2,f^3]\right\|_{L_\xi^2} \\  \lesssim & \left\|\int_{\mathbb{R}^2}m(\mu,\kappa) \widehat{\mathrm{e}^{it\partial_x^2}f^1}(\xi-\mu) \overline{\widehat{\mathrm{e}^{it\partial_x^2} f^2}}(\xi-\mu-\kappa) \widehat{\mathrm{e}^{it\partial_x^2} f^3}(\xi-\kappa) d \mu d \kappa\right\|_{L_\xi^2} \\  \lesssim & \int_{\mathbb{R}^{2}}\left|\mathcal{F}_{\mu,\kappa}^{-1} m(y, z)\right|\left\|I^t[T_z f^1, T_{y+z} f^2 , T_y f^3]\right\|_{L_{x}^{2}} d y d z   \\ \lesssim & \left\|\mathcal{F}_{\mu,\kappa}^{-1} m\right\|_{L^{1}\left(\mathbb{R}^{2}\right)}\min _{\{j, k, \ell\}=\{1,2,3\}} \left\|f^j\right\|_{L_x^2} \left\|\mathrm{e}^{it\partial_x^2} f^k\right\|_{L_x^\infty}  \left\|\mathrm{e}^{it\partial_x^2} f^\ell\right\|_{L_x^\infty} \\  \lesssim &  \frac{\left\|\mathcal{F}_{\mu,\kappa}^{-1} m\right\|_{L^{1}\left(\mathbb{R}^{2}\right)}}{|t|} \min _{\{j, k, \ell\}=\{1,2,3\}}\left\|f^{j}\right\|_{L_{x}^{2}}\left\|f^{k}\right\|_{L_{x}^{2}}^{\frac{1}{2}}\left\|x f^{k}\right\|_{L_{x}^{2}}^{\frac{1}{2}}\left\|f^{\ell}\right\|_{L_{x}^{2}}^{\frac{1}{2}}\left\|x f^{\ell}\right\|_{L_{x}^{2}}^{\frac{1}{2}}.
\end{aligned}
\end{equation}
\end{remark}
Then we give an estimate to an auxiliary function which will be used in the proof of Lemma \ref{Lemma 3.5}.
\begin{lemma}
\label{Lemma 3.6}
$\text { For } T \geq 1, \varphi \in C_{c}^{\infty}(\mathbb{R}), \varphi(x)=1 \text { when }|x| \leq 1$ when $|x| \leq 1$ and $\varphi(x)=0$ when $|x| \geq 2$. We define for $T / 2 \leq t \leq T$, 
\begin{align*}
\widetilde{m}(\mu, \kappa):=\varphi\left(t^{\frac{3}{4}} \mu \kappa\right) \varphi\left(\frac{T^{-\frac{1}{6}}}{4} \mu\right) \varphi\left(\frac{T^{-\frac{1}{6}}}{4} \kappa\right).
\end{align*}
Then $\left\|\mathcal{F}_{\mu , \kappa}^{-1} \widetilde{m}\right\|_{L^{1}\left(\mathbb{R}^{2}\right)} \lesssim t^{\frac{\delta}{100}}$.
\end{lemma}
\begin{proof}
We have
\begin{align*}
\left\|\mathcal{F}_{\mu, \kappa}^{-1} \widetilde{m}\right\|_{L^{1}\left(\mathbb{R}^{2}\right)}=\left\|I\left(x_{1}, x_{2}\right)\right\|_{L_{x_{1}, x_{2}}^{1}},
\end{align*}
where
\begin{align*}
I\left(x_{1}, x_{2}\right)=\int_{\mathbb{R}^{2}} \mathrm{e}^{i x_{1} \mu} \mathrm{e}^{i x_{2} \kappa} \varphi(S \mu \kappa) \varphi(\mu) \varphi(\kappa) d \mu d \kappa, \quad S \approx T^{\frac{13}{12}}.
\end{align*}
Then we are to show
\begin{align*}
\left|I\left(x_{1}, x_{2}\right)\right|+\left|x_{1} I\left(x_{1}, x_{2}\right)\right|+\left|x_{2} I\left(x_{1}, x_{2}\right)\right| \lesssim 1, \quad\left|x_{1} x_{2} I\left(x_{1}, x_{2}\right)\right| \lesssim \log (1+T).
\end{align*}
For the first inequality above, we estimate $\left|x_{1} I\left(x_{1}, x_{2}\right)\right|$ for example, and other two terms can be estimated in the same way. We have
\begin{align*}\left|x_{1} I\left(x_{1}, x_{2}\right)\right| &=\left|\int_{\mathbb{R}^{2}} \frac{1}{i} \partial_{\mu}\left(\mathrm{e}^{i x_{1} \mu}\right) \mathrm{e}^{i x_{2} \kappa} \varphi(S \mu \kappa) \varphi(\mu) \varphi(\kappa) d \mu d \kappa\right| \\ &=\left|\int_{\mathbb{R}^{2}} \mathrm{e}^{i x_{1} \mu} \mathrm{e}^{i x_{2} \kappa}\left[S \kappa \varphi^{\prime}(S \mu \kappa) \varphi(\mu) \varphi(\kappa)+\varphi(S \mu \kappa) \varphi^{\prime}(\mu) \varphi(\kappa)\right] d \mu d \kappa\right| \\ & \lesssim 1+\left|\int_{\mathbb{R}^{2}} \mathrm{e}^{i x_{1} \mu} \mathrm{e}^{i x_{2} \kappa}\left(S \kappa \varphi^{\prime}(S \mu \kappa) \varphi(\mu) \varphi(\kappa)\right) d \mu d \kappa\right|. 
\end{align*}
We obeserve that $|S \mu \kappa| \leq 2$, then the second term turns out to be
\begin{align*}
\left|\int_{\mathbb{R}^{2}} \mathrm{e}^{i x_{1} \mu} \mathrm{e}^{i x_{2} \kappa}\left(S \kappa \varphi^{\prime}(S \mu \kappa) \varphi(\mu) \varphi(\kappa)\right) d \mu d \kappa\right| \lesssim \int_{D:=\{|S \mu \kappa|,|\mu|,|\kappa| \leq 2\}}|S \kappa| d \mu d \kappa \lesssim 1.
\end{align*}
So we get the first inequality, and we use a similar method to prove the second one. We have
\begin{align*} \left|x_{1} x_{2} I\left(x_{1}, x_{2}\right)\right| & \lesssim \int_{\mathbb{R}^{2}}\left|\partial_{\mu} \partial_{\kappa}(\varphi(S \mu \kappa) \varphi(\mu) \varphi(\kappa))\right| d \mu d \kappa \\ & \lesssim \int_{D}\left|S \kappa \varphi^{\prime}(S \mu \kappa) \varphi(\mu) \varphi^{\prime}(\kappa)\right|+\left|S \kappa S \mu \varphi^{\prime \prime}(S \mu \kappa) \varphi(\mu) \varphi(\kappa)\right| \\ & +\left|S \mu \varphi^{\prime}(S \mu \kappa) \varphi^{\prime}(\mu) \varphi(\kappa)\right|+\left|\varphi(S \mu \kappa) \varphi^{\prime}(\mu) \varphi^{\prime}(\kappa)\right|+\left|S \varphi^{\prime}(S \mu \kappa) \varphi(\mu) \varphi(\kappa)\right| d \mu d \kappa \\ & \lesssim\left(\int_{0}^{T^{-13 / 12}} \int_{0}^{2}+\int_{T^{-13 / 12}}^{2} \int_{0}^{\frac{2 T^{-13 / 12}}{\kappa}}\right)[1+|S \kappa|+|S \mu|+|S \kappa S \mu|+S] d \mu d \kappa \\ & \lesssim \log (1+T).\end{align*}
Then we have
\begin{align*}
\left(1+\left|x_{1}\right|\right)\left(1+\left|x_{2}\right|\right)\left|I\left(x_{1}, x_{2}\right)\right| \lesssim \log (1+T).
\end{align*}
We also have a polynomial in $T$ bound
\begin{align*}
\left(1+\left|x_{1}\right|^{2}\right)\left(1+\left|x_{2}\right|^{2}\right)\left|I\left(x_{1}, x_{2}\right)\right| \lesssim T^{13 / 3}.
\end{align*}
Therefore by interpolation we obtain that for every $0 < \varepsilon < 13/3$, there exists $ k > 1/2$ such that
such that
\begin{align*}
\left|I\left(x_{1}, x_{2}\right)\right| \lesssim(1+T)^{\varepsilon}\left(1+\left|x_{1}\right|^{2}\right)^{-k}\left(1+\left|x_{2}\right|^{2}\right)^{-k}.
\end{align*}
Thus we have 
\begin{align*}
\left\|\mathcal{F}_{\eta, \kappa}^{-1} \widetilde{m}\right\|_{L^{1}\left(\mathbb{R}^{2}\right)} \lesssim t^{\frac{\delta}{100}}.
\end{align*}
\end{proof}
\begin{lemma}
\label{lemma 3.81}
Let $\varpi$ be a function such that $\varpi^{-}(\eta) : = \mathbbm{1}_{\eta \leq 0} \, \varpi(\eta)$ is a Schwartz function. Then for $k \leq 0$, we have
\begin{equation}
\left\|\int_{-\infty}^{0} \mathrm{e}^{i y \eta} \frac{\mathrm{e}^{-2 i t \eta}-1}{i \eta} \varpi\left(\frac{\eta}{2^{k}}\right) d \eta\right\|_{L_{y}^{1}} \lesssim 2^{-\frac{k}{2}} t^{\frac{1}{2}}.
\end{equation}
\end{lemma}
\begin{proof}
We have
\begin{align*}
& \left\|\int_{-\infty}^{0} \mathrm{e}^{i y \eta} \frac{\mathrm{e}^{-2 i t \eta}-1}{i \eta} \varpi\left(\frac{\eta}{2^{k}}\right) d \eta\right\|_{L_{y}^{1}} \\ = & \left\|\int_{-\infty}^{0} \int_{y}^{y-2t} \mathrm{e}^{i\tau\eta} d\tau \, \varpi\left(\frac{\eta}{2^{k}}\right) d \eta\right\|_{L_{y}^{1}} \\ \lesssim & \left\|\int_{2^k y}^{2^k(y-2t)} \mathcal{F}_{\eta\to\tau}^{-1}\varpi^{-}(\tau) d\tau \, \right\|_{L_{y}^{1}},
\end{align*}
where $\varpi^{-}(\eta) = \mathbbm{1}_{\eta \leq 0} \, \varpi(\eta)$.

Since $\varpi^{-}$ is a Schwartz function and $\varpi^{-}(0) = 0$, we have $\mathcal{F}_{\eta\to\tau}^{-1}\varpi^{-}$ is a Schwartz function and $\int_{-\infty}^\infty \mathcal{F}_{\eta\to\tau}^{-1}\varpi^{-}(\tau) d\tau = 0$. So we have
\begin{align*}
\left|\mathcal{F}_{\eta\to\tau}^{-1}\varpi^{-}(\tau)\right| \lesssim \frac{1}{|\tau|^{\frac{3}{2}}}.
\end{align*}
Without loss of generality, we assume here $t\geq 1$. To estimate the $L_y^1$ norm of $\int_{2^k y}^{2^k(y-2t)} \mathcal{F}_{\eta\to\tau}^{-1}\varpi^{-}(\tau) d\tau$, we decompose the integral interval into five parts. We have
\begin{align*}\left\|\int_{2^k y}^{2^k(y-2t)} \mathcal{F}_{\eta\to\tau}^{-1}\varpi^{-}(\tau) d\tau \, \right\|_{L_{y}^{1}} = & \int_{-\infty}^{-1} \left|\int_{2^k y}^{2^k(y-2t)} \mathcal{F}_{\eta\to\tau}^{-1}\varpi^{-}(\tau) d\tau \, \right|dy \\ + & \int_{-1}^{1} \left|\int_{2^k y}^{2^k(y-2t)} \mathcal{F}_{\eta\to\tau}^{-1}\varpi^{-}(\tau) d\tau \, \right|dy \\ + & \int_{1}^{2t-1} \left|\int_{2^k y}^{2^k(y-2t)} \mathcal{F}_{\eta\to\tau}^{-1}\varpi^{-}(\tau) d\tau \, \right|dy \\ + & \int_{2t-1}^{2t+1} \left|\int_{2^k y}^{2^k(y-2t)} \mathcal{F}_{\eta\to\tau}^{-1}\varpi^{-}(\tau) d\tau \, \right|dy \\ + 
& \int_{2t+1}^{\infty} \left|\int_{2^k y}^{2^k(y-2t)} \mathcal{F}_{\eta\to\tau}^{-1}\varpi^{-}(\tau) d\tau \, \right|dy.
\end{align*}
For the first term above, we have
\begin{align*}
\int_{-\infty}^{-1} \left|\int_{2^k y}^{2^k(y-2t)} \mathcal{F}_{\eta\to\tau}^{-1}\varpi^{-}(\tau) d\tau\right| dy \leq & \int_{-\infty}^{-1} \int_{2^k (y-2t)}^{2^k y} \left|\mathcal{F}_{\eta\to\tau}^{-1}\varpi^{-}(\tau)\right| d\tau dy \\ \lesssim & \int_{-\infty}^{-1} \int_{2^k (y-2t)}^{2^k y} \frac{1}{(-\tau)^{\frac{3}{2}}} d\tau dy \\ \lesssim & \, 2^{-\frac{k}{2}}\int_{-\infty}^{-1} \left(\frac{1}{(-y)^{\frac{1}{2}}}-\frac{1}{(2t-y)^{\frac{1}{2}}}\right) dy \\ \lesssim & \, 2^{-\frac{k}{2}} \int_{-2t-1}^{-1}\frac{1}{(-y)^{\frac{1}{2}}}dy \\ \lesssim & \,  2^{-\frac{k}{2}} t^{\frac{1}{2}}.
\end{align*}
For the second term above, since $\mathcal{F}_{\eta \rightarrow \tau}^{-1} \varpi^{-}$ is a Schwartz function and $\int_{-\infty}^{\infty} \mathcal{F}_{\eta \rightarrow \tau}^{-1} \varpi^{-}(\tau) d \tau=0$, we deduce that $\int_{-\infty}^y \mathcal{F}_{\eta \rightarrow \tau}^{-1} \varpi^{-}(\tau) d \tau$ is also a Schwartz function, thus $\|\int_{-\infty}^y \mathcal{F}_{\eta \rightarrow \tau}^{-1} \varpi^{-}(\tau) d \tau\|_{L_y^\infty} \lesssim 1$. So we have
\begin{align*}
\int_{-1}^{1} \left|\int_{2^k y}^{2^k(y-2t)} \mathcal{F}_{\eta\to\tau}^{-1}\varpi^{-}(\tau) d\tau \, \right|dy \lesssim & \int_{-1}^{1} \left|\int_{-\infty}^{2^k (y-2t)} \mathcal{F}_{\eta\to\tau}^{-1}\varpi^{-}(\tau) d\tau -\int_{-\infty}^{2^k y} \mathcal{F}_{\eta\to\tau}^{-1}\varpi^{-}(\tau) d\tau \right|dy\\ \lesssim & \left\|\int_{-\infty}^{2^k (y-2t)} \mathcal{F}_{\eta\to\tau}^{-1}\varpi^{-}(\tau) d\tau\right\|_{L_y^\infty} + \left\|\int_{-\infty}^{2^k y} \mathcal{F}_{\eta\to\tau}^{-1}\varpi^{-}(\tau) d\tau\right\|_{L_y^\infty}  \\ \lesssim & \, 1.
\end{align*}
For the third term above, since 
\begin{align*}
\int_{-\infty}^{2^k y} \mathcal{F}_{\eta\to\tau}^{-1}\varpi^{-}(\tau)d\tau + \int_{2^k y}^{\infty} \mathcal{F}_{\eta\to\tau}^{-1}\varpi^{-}(\tau)d\tau = \int_{-\infty}^{\infty} \mathcal{F}_{\eta\to\tau}^{-1}\varpi^{-}(\tau)d\tau = 0,
\end{align*}
we have
\begin{align*}
\int_{1}^{2t-1} \left|\int_{2^k y}^{2^k(y-2t)} \mathcal{F}_{\eta\to\tau}^{-1}\varpi^{-}(\tau) d\tau \, \right|dy = & \int_{1}^{2t-1} \left|-\int_{-\infty}^{2^k y} \mathcal{F}_{\eta\to\tau}^{-1}\varpi^{-}(\tau)d\tau + \int_{-\infty}^{2^k (y-2t)} \mathcal{F}_{\eta\to\tau}^{-1}\varpi^{-}(\tau)d\tau \right|dy  \\= & \int_{1}^{2t-1} \left|\int_{2^k y}^{\infty} \mathcal{F}_{\eta\to\tau}^{-1}\varpi^{-}(\tau)d\tau + \int_{-\infty}^{2^k (y-2t)} \mathcal{F}_{\eta\to\tau}^{-1}\varpi^{-}(\tau)d\tau \right|dy \\ \lesssim & \int_{1}^{2t-1} \left(\int_{2^k y}^{\infty} \left|\mathcal{F}_{\eta\to\tau}^{-1}\varpi^{-}(\tau)\right|d\tau + \int_{-\infty}^{2^k (y-2t)} \left|\mathcal{F}_{\eta\to\tau}^{-1}\varpi^{-}(\tau)\right|d\tau \right)dy  \\  \lesssim & \int_{1}^{2t-1} \left(\int_{2^k y}^{\infty} \frac{1}{\tau^{\frac{3}{2}}}d\tau+ \int_{-\infty}^{2^k (y-2t)}\frac{1}{(-\tau)^{\frac{3}{2}}}d\tau \right)dy\\ \lesssim & \, 2^{-\frac{k}{2}} \left(\int_{1}^{2t-1}\frac{1}{y^{\frac{1}{2}}}dy + \int_{1}^{2t-1} \frac{1}{(2t-y)^{\frac{1}{2}}}dy\right)\\ \lesssim & \, 2^{-\frac{k}{2}} t^{\frac{1}{2}} .
\end{align*}
For the fourth term above, as we did to estimate the second term above, we have
\begin{align*}
\int_{2t-1}^{2t+1} \left|\int_{2^k y}^{2^k(y-2t)} \mathcal{F}_{\eta\to\tau}^{-1}\varpi^{-}(\tau) d\tau \, \right|dy \lesssim & \int_{2t-1}^{2t+1} \left|\int_{-\infty}^{2^k (y-2t)} \mathcal{F}_{\eta\to\tau}^{-1}\varpi^{-}(\tau) d\tau -\int_{-\infty}^{2^k y} \mathcal{F}_{\eta\to\tau}^{-1}\varpi^{-}(\tau) d\tau \right|dy\\ \lesssim & \left\|\int_{-\infty}^{2^k (y-2t)} \mathcal{F}_{\eta\to\tau}^{-1}\varpi^{-}(\tau) d\tau\right\|_{L_y^\infty} + \left\|\int_{-\infty}^{2^k y} \mathcal{F}_{\eta\to\tau}^{-1}\varpi^{-}(\tau) d\tau\right\|_{L_y^\infty}  \\ \lesssim & \, 1.
\end{align*}
For the fifth term above, as we did to estimate the first term above, we have
\begin{align*}
\int_{2t+1}^{\infty} \left|\int_{2^k y}^{2^k(y-2t)} \mathcal{F}_{\eta\to\tau}^{-1}\varpi^{-}(\tau) d\tau\right| dy \leq & \int_{2t+1}^{\infty} \int_{2^k (y-2t)}^{2^k y} \left|\mathcal{F}_{\eta\to\tau}^{-1}\varpi^{-}(\tau)\right| d\tau dy \\ \lesssim &  \int_{2t+1}^{\infty} \int_{2^k (y-2t)}^{2^k y} \frac{1}{\tau^{\frac{3}{2}}} d\tau dy \\ \lesssim & \, 2^{-\frac{k}{2}}\int_{2t+1}^{\infty} \left(\frac{1}{(y-2t)^{\frac{1}{2}}}-\frac{1}{y^{\frac{1}{2}}}\right) dy \\ \lesssim & \, 2^{-\frac{k}{2}} \int_{1}^{2t+1}\frac{1}{y^{\frac{1}{2}}}dy \\ \lesssim & \, 2^{-\frac{k}{2}} t^{\frac{1}{2}}.
\end{align*}
So we have
\begin{align*}
\left\|\int_{-\infty}^{0} \mathrm{e}^{i y \eta} \frac{\mathrm{e}^{-2 i t \eta}-1}{i \eta} \left(\frac{\eta}{2^{k}}\right) d \eta\right\|_{L_{y}^{1}} \lesssim 2^{-\frac{k}{2}} t^{\frac{1}{2}}.
\end{align*}
\end{proof}

The main result in this subsection is to estimate $\widetilde{\mathcal{N}}^{t}$ with low frequencies in $x$, here we use quite a different method other than in [\ref{[2]}].
\begin{lemma}
\label{Lemma 3.5}
For $T \geq 1$ ,assume that $F, G, H: \mathbb{R} \rightarrow S$ satisfy (\ref{3.3}) and 
\begin{equation}
\label{3.19}
F=Q_{\leq T^{\frac{1}{6}}} F, \quad G=Q_{\leq T^{\frac{1}{6}}} G, \quad H=Q_{\leq T^{\frac{1}{6}}} H.
\end{equation}
Then for $t \in [\frac{T}{2},T]$ we can write
\begin{align*}
\widetilde{\mathcal{N}}^{t}[F(t), G(t), H(t)]  =\widetilde{\mathcal{E}}_{1}^{t}[F(t), G(t), H(t)]+\mathcal{E}_{2}^{t}[F(t), G(t), H(t)].
\end{align*}
We set $\widetilde{\mathcal{E}}_{1}(t):=\widetilde{\mathcal{E}}_{1}^{t}[F(t), G(t), H(t)]$ and $\mathcal{E}_{2}(t):=\mathcal{E}_{2}^{t}[F(t), G(t), H(t)]$, it holds unifomrly in $T \geq 1$ that
\begin{equation}
\label{3.225}
T^{1+20 \delta} \sup _{\frac{T}{2} \leq t \leq T}\left\|\widetilde{\mathcal{E}}_{1}(t)\right\|_{\mathcal{Y}} \lesssim 1, \quad  T^{\frac{1}{3}}\sup _{\frac{T}{2} \leq t \leq T} \left\|\mathcal{E}_{3}(t)\right\|_{\mathcal{Y}} \lesssim 1,
\end{equation}
where $\mathcal{E}_{2}(t)=\partial_{t} \mathcal{E}_{3}(t)$.
\end{lemma}
\begin{proof}
Let $\frac{T}{2} \leq t \leq T$. From (\ref{3.19}), we set
\begin{align*}
F=F^{a}=Q_{\leq T^{\frac{1}{6}}} F^{a},\quad  G = F^{b}=Q_{\leq T^{\frac{1}{6}}} F^{b}, \quad H = F^{c}=Q_{\leq T^{\frac{1}{6}}} F^{c}.
\end{align*}
Let 
\begin{equation}
\begin{aligned}
& \mathcal{F}\widetilde{\mathcal{N}}^{t}\left[F^{a}, F^{b}, F^{c}\right](\xi,\eta) \\ = & \int_{\omega\left(\eta, \eta_{1} \eta_{2}\right) \neq 0} e^{i t \omega} \left( \mathcal{O}_{1}^{t}\left[F_{\eta-\eta_{1}}^{a}, F_{\eta-\eta_{1}-\eta_{2}}^{b}, F_{\eta-\eta_{2}}^{c}\right] + \mathcal{O}_{2}^{t}\left[F_{\eta-\eta_{1}}^{a}, F_{\eta-\eta_{1}-\eta_{2}}^{b}, F_{\eta-\eta_{2}}^{c}\right] \right) d\eta_1 d\eta_2,
\end{aligned}
\end{equation}
with
\begin{equation}
\mathcal{O}_{1}^{t}\left[f^{a}, f^{b}, f^{c}\right](\xi):=\int_{\mathbb{R}^{2}} \mathrm{e}^{2 i t \mu \kappa}\left(1-\varphi\left(t^{\frac{3}{4}} \mu \kappa\right)\right) \widehat{f^{a}}(\xi-\mu) \overline{\widehat{f^{b}}}(\xi-\mu-\kappa) \widehat{f^{c}}(\xi-\kappa) d \mu d \kappa,
\end{equation}
\begin{equation}
\label{3.22}
\mathcal{O}_{2}^{t}\left[f^{a}, f^{b}, f^{c}\right](\xi):=\int_{\mathbb{R}^{2}} \mathrm{e}^{2 i t \mu \kappa} \varphi\left(t^{\frac{3}{4}} \mu \kappa\right) \widehat{f^{a}}(\xi-\mu) \overline{\widehat{f^{b}}}(\xi-\mu-\kappa) \widehat{f^{c}}(\xi-\kappa) d \mu d \kappa.
\end{equation}
Then for $ \omega(\eta,\eta_1,\eta_2) \neq 0$, we have
\begin{equation}
\label{3.23}
\begin{array}{l}\mathrm{e}^{i t \omega} \mathcal{O}_{2}^{t}\left[f^{a}, f^{b}, f^{c}\right]=\partial_{t}\left(\frac{\mathrm{e}^{i t \omega}-1}{i \omega} \mathcal{O}_{2}^{t}\left[f^{a}, f^{b}, f^{c}\right]\right)-\frac{\mathrm{e}^{i t \omega}-1}{i \omega}\left(\partial_{t} \mathcal{O}_{2}^{t}\right)\left[f^{a}, f^{b}, f^{c}\right] \\ -\frac{\mathrm{e}^{i t \omega}-1}{i \omega} \mathcal{O}_{2}^{t}\left[\partial_{t} f^{a}, f^{b}, f^{c}\right]-\frac{\mathrm{e}^{i t \omega}-1}{i \omega} \mathcal{O}_{2}^{t}\left[f^{a}, \partial_{t} f^{b}, f^{c}\right]-\frac{\mathrm{e}^{i t \omega}-1}{i \omega} \mathcal{O}_{2}^{t}\left[f^{a}, f^{b}, \partial_{t} f^{c}\right] \\ :=\partial_{t}\left(\frac{\mathrm{e}^{i t \omega}-1}{i \omega} \mathcal{O}_{2}^{t}\left[f^{a}, f^{b}, f^{c}\right]\right)+(\mathrm{e}^{i t \omega}-1) \mathcal{L}^{t}\left[f^{a}, f^{b}, f^{c}\right],\end{array}
\end{equation}
where
\begin{align*}
\left(\partial_{t} \mathcal{O}_{2}^{t}\right)\left[f^{a}, f^{b}, f^{c}\right]:=\int_{\mathbb{R}^{2}} \partial_{t}\left(\mathrm{e}^{2 i t \mu \kappa} \varphi\left(t^{\frac{3}{4}} \mu \kappa\right)\right) \widehat{f^{a}}(\xi-\mu) \overline{\widehat{f^{b}}}(\xi-\mu-\kappa) \widehat{f^{c}}(\xi-\kappa) d \mu d \kappa.
\end{align*}
Then we define 
\begin{align*}
\mathcal{E}_{2}^{t}\left[F^{a}, F^{b}, F^{c}\right]=\partial_{t} \mathcal{E}_{3}^{t}\left[F^{a}, F^{b}, F^{c}\right],
\end{align*}
where 
\begin{equation}
\mathcal{F} \mathcal{E}_{3}^{t}\left[F^{a}, F^{b}, F^{c}\right](\xi, \eta):= \int_{\omega(\eta,\eta_1\eta_2) \neq 0} \frac{\mathrm{e}^{i t \omega(\eta,\eta_1\eta_2)}-1}{i \omega(\eta,\eta_1\eta_2)} \mathcal{O}_{2}^{t}\left[F_{\eta-\eta_1}^{a}, F_{\eta_2-\eta_1}^{b}, F_{\eta_2}^{c}\right] d\eta_1 d\eta_2.
\end{equation}
We also define 
\begin{equation}
\label{3.25}
\begin{aligned}
\mathcal{F} \widetilde{\mathcal{E}}_{1}^{t}(\xi, \eta):=  \int_{\omega(\eta,\eta_1\eta_2) \neq 0} \bigg(\mathrm{e}^{i t \omega(\eta,\eta_1\eta_2)} & \mathcal{O}_{1}^{t}\left[F_{\eta-\eta_1}^{a}, F_{\eta_2-\eta_1}^{b}, F_{\eta_2}^{c}\right] \\ & + (\mathrm{e}^{i t \omega(\eta,\eta_1\eta_2)}-1)\mathcal{L}^{t}\left[F_{\eta-\eta_1}^{a}, F_{\eta_2-\eta_1}^{b}, F_{\eta_2}^{c}\right] \bigg)d\eta_1 d\eta_2.
\end{aligned}
\end{equation}
$\textbf { 1. Estimation of } \mathcal{E}_{3}(t) $. We define the multiplier 
\begin{align*}
m(\mu, \kappa):=\varphi\left(t^{\frac{3}{4}} \mu \kappa\right) \varphi\left(\frac{T^{-\frac{1}{6}}}{4} \mu\right) \varphi\left(\frac{T^{-\frac{1}{6}}}{4} \kappa\right).
\end{align*}
From (\ref{3.19}), this multiplier and $\varphi\left(t^{\frac{3}{4}} \mu \kappa\right)$ are equivalent in (\ref{3.22}).\\\\
From Lemma \ref{Lemma 3.6}, we have $\left\|\mathcal{F}_{\mu \kappa} \widetilde{m}\right\|_{L^{1}\left(\mathbb{R}^{2}\right)} \lesssim t^\frac{\delta}{100}$. Then by Remark \ref{remark 3.7}, we have
\begin{equation}
\label{3.35}
\begin{aligned}
& \left\|\mathcal{O}_{2}^{t}\left[f^{a}, f^{b}, f^{c}\right]\right\|_{L_{\xi}^{2}} \\  \lesssim & \, (1+|t|)^{-1 + \frac{\delta}{100}} \min _{\{j, k, \ell\}=\{1,2,3\}}\left\|f^{j}\right\|_{L_{x}^{2}}\left\|f^{k}\right\|_{L_{x}^{2}}^{\frac{1}{2}}\left\|x f^{k}\right\|_{L_{x}^{2}}^{\frac{1}{2}}\left\|f^{\ell}\right\|_{L_{x}^{2}}^{\frac{1}{2}}\left\|x f^{\ell}\right\|_{L_{x}^{2}}^{\frac{1}{2}}.
\end{aligned}
\end{equation}
Since $\omega(\eta,\eta_1,\eta_2) \neq 0$, from Lemma \ref{Lemma 4.1} in Section 4, we deduce that\\
(1). If $\eta \geq 0, \eta-\eta_1 \leq 0, \eta_2-\eta_1 \leq 0, \eta_2 \leq 0$, then $\omega(\eta,\eta_1,
\eta_2) = 2 \eta$.\\
(2). If $\eta \leq 0, \eta-\eta_1 \geq 0, \eta_2-\eta_1 \geq 0, \eta_2 \geq 0$, then $\omega(\eta,\eta_1,\eta_2) = -2 \eta$.\\
(3). If $\eta \geq 0, \eta-\eta_1 \leq 0, \eta_2-\eta_1 \geq 0, \eta_2 \geq 0$, then $\omega(\eta,\eta_1,
\eta_2) = 2 \eta - 2\eta_1$.\\
(4). If $\eta \leq 0, \eta-\eta_1 \geq 0, \eta_2-\eta_1 \leq 0, \eta_2 \leq 0$, then $\omega(\eta,\eta_1,\eta_2) = 2\eta_1-2 \eta$.\\
(5). If $\eta \geq 0, \eta-\eta_1 \geq 0, \eta_2-\eta_1 \leq 0, \eta_2 \geq 0$, then $\omega(\eta,\eta_1,\eta_2) = 2\eta_1-2\eta_2$.\\
(6). If $\eta \leq 0, \eta-\eta_1 \leq 0, \eta_2-\eta_1 \geq 0, \eta_2 \leq 0$, then $\omega(\eta,\eta_1,\eta_2) = 2\eta_2-2\eta_1$.\\
(7). If $\eta \geq 0, \eta-\eta_1 \geq 0, \eta_2-\eta_1 \geq 0, \eta_2 \leq 0$, then $\omega(\eta,\eta_1,\eta_2) = 2\eta_2$.\\
(8). If $\eta \leq 0, \eta-\eta_1 \leq 0, \eta_2-\eta_1 \leq 0, \eta_2 \geq 0$, then $\omega(\eta,\eta_1,\eta_2) = -2\eta_2$.\\
(9). If $\eta \geq 0, \eta-\eta_1 \geq 0, \eta_2-\eta_1 \leq 0, \eta_2 \leq 0$, then $\omega(\eta,\eta_1,\eta_2) = 2\eta_1$.\\
(10). If $\eta \leq 0, \eta-\eta_1 \leq 0, \eta_2-\eta_1 \geq 0, \eta_2 \geq 0$, then $\omega(\eta,\eta_1,\eta_2) = -2\eta_1$.\\
(11). If $\eta \geq 0, \eta-\eta_1 \leq 0, \eta_2-\eta_1 \geq 0, \eta_2 \leq 0$, then $\omega(\eta,\eta_1,\eta_2) = 2\eta-2\eta_1+2\eta_2$.\\
(12). If $\eta \leq 0, \eta-\eta_1 \geq 0, \eta_2-\eta_1 \leq 0, \eta_2 \geq 0$, then $\omega(\eta,\eta_1,\eta_2) = -2\eta+2\eta_1-2\eta_2$.\\
(13). If $\eta \geq 0, \eta-\eta_1 \leq 0, \eta_2-\eta_1 \leq 0, \eta_2 \geq 0$, then $\omega(\eta,\eta_1,\eta_2) = 2\eta-2\eta_2$.\\
(14). If $\eta \leq 0, \eta-\eta_1 \geq 0, \eta_2-\eta_1 \geq 0, \eta_2 \leq 0$, then $\omega(\eta,\eta_1,\eta_2) = 2\eta_2-2\eta$.\\\\
In fact, the above cases include the case $\left\{(\eta_1,\eta_2) \in \mathbb{R}^2|\eta_{1} = 0\right\} \cup\left\{(\eta_1,\eta_2) \in \mathbb{R}^2 | \eta= \eta_2\right\}$, which is the case with $\omega(\eta,\eta_1,\eta_2) = 0$ due to Lemma \ref{Lemma 4.1}. But we can neglect them because they are of measure zero in $\mathbb{R}^2$ and they do not interfere in the integration.\\\\
Without loss of generality, we classify these 14 cases of $\omega\left(\eta, \eta_{1}, \eta_{2}\right) \neq 0$ into the following three cases.
\\\textbf{Case 1: }$\omega\left(\eta, \eta_{1}, \eta_{2}\right)=\pm 2 \eta$. 
\\ \textbf{Case 2: }$\omega\left(\eta, \eta_{1}, \eta_{2}\right)=\pm 2 \eta_{1} \text{ or }\pm (2\eta-2 \eta_{2}) \text{ or } \pm\left(2 \eta-2 \eta_{1}+2 \eta_{2}\right)$. 
\\\textbf{Case 3: }$\omega\left(\eta, \eta_{1}, \eta_{2}\right)=\pm\left(2 \eta-2 \eta_{1}\right) \text{ or } \pm 2 \eta_{2} \text{ or } \pm\left(2 \eta_{1}-2 \eta_{2}\right)$.\\\\ 
Now we analyze each case individually.\\
$\textbf{Case 1}$. For $\omega(\eta,\eta_1,\eta_2) = \pm 2\eta$. We only have to deal with the case $\omega(\eta,\eta_1,\eta_2) = -2 \eta$, and the other case can be treated in the same way. We firstly estimate $\left\|\mathcal{E}_{3}(t)\right\|_{L_{x, y}^{2}}$. By (\ref{3.35}) and Lemma \ref{Lemma 6.21}, we have
\begin{equation}
\label{3.3201}
\begin{aligned}
& \left\|\mathcal{E}_{3}(t)\right\|_{L_{x, y}^{2}}  \\ \lesssim & \left\|\int_{\omega(\eta,\eta_1,\eta_2)= -2\eta} \frac{\mathrm{e}^{-2i t \eta}-1}{i \eta} \mathcal{O}_{2}^{t}\left[F_{\eta-\eta_1}^{a} , F_{\eta_2-\eta_1}^{b} , F_{\eta_2}^{c}\right] d\eta_1 d\eta_2 \right\|_{L_{\xi,\eta}^{2}}  \\  \lesssim & \left\|\int_{\mathbb{R}^2} \frac{\mathrm{e}^{-2i t \eta}-1}{i \eta} \mathcal{O}_{2}^{t}\left[F_{\eta-\eta_1}^{a} \mathbbm{1}_{\eta-\eta_1 \geq 0}, F_{\eta_2-\eta_1}^{b} \mathbbm{1}_{\eta_2-\eta_1 \geq 0}, F_{\eta_2}^{c}\mathbbm{1}_{\eta_2 \geq 0}\right] d\eta_1 d\eta_2 \right\|_{L_{\xi,\eta}^{2}} \\  \lesssim & \left\|\int_{\mathbb{R}^2} \mathcal{O}_{2}^{t}\left[F_{\eta-\eta_1}^{a}\mathbbm{1}_{\eta-\eta_1 \geq 0}, F_{\eta_2-\eta_1}^{b}\mathbbm{1}_{\eta_2-\eta_1 \geq 0}, F_{\eta_2}^{c}\mathbbm{1}_{\eta_2 \geq 0}\right]d\eta_1 d\eta_2\right\|_{L_{\xi}^2 L_\eta^\infty} \left(\int_{-\infty}^{0} \frac{\sin ^{2}(t \eta)}{\eta^{2}}d \eta \right)^{\frac{1}{2}} \\ \lesssim & (1+|t|)^{\frac{1}{2}} \left\|\mathcal{F}_{\xi\to x}^{-1}\mathcal{O}_{2}^{t}\left[F_{+}^{a}, F_{+}^{b}, F_{+}^{c}\right]\right\|_{L_y^1 L_x^2} \\ \lesssim & (1+|t|)^{-\frac{1}{2}+\frac{\delta}{100}} \min _{\{\alpha, \beta, \gamma\}=\{a, b, c\}}\left\|F^{\alpha}\right\|_{L_{x, y}^{2}} \| F^{\beta}\|_S  \|F^{\gamma}\|_{S},
\end{aligned}
\end{equation}
where $\mathcal{F}_{y}(F_{+})(\eta) = \left\{\begin{array}{l}F_{\eta}, \text { if } \eta \geq 0, \\ 0, \text { if } \eta<0, \end{array}\right. $ and $\mathcal{F}_{y}(F_{-})(\eta) =\left\{\begin{array}{l} F_{\eta}, \text { if } \eta \leq 0, \\ 0, \text { if } \eta>0. \end{array}\right.$\\\\
So we have
\begin{equation}
\label{3.29}
\left\|\mathcal{E}_{3}(t)\right\|_{L_{x, y}^{2}} \lesssim (1+|t|)^{-\frac{1}{2}+\frac{\delta}{100}} \min _{\{\alpha, \beta, \gamma\}=\{a, b, c\}}\left\|F^{\alpha}\right\|_{L_{x, y}^{2}} \| F^{\beta}\|_S  \|F^{\gamma}\|_{S}.
\end{equation}
Then by Lemma \ref{Lemma 6.2}, we have
\begin{equation}
\label{3.2501}
\left\|\mathcal{E}_{3}(t)\right\|_{S'} \lesssim (1+|t|)^{-\frac{1}{2}+\frac{\delta}{100}} \max _{\{\alpha, \beta, \gamma\}=\{a, b, c\}}\left\|F^{\alpha}\right\|_{S'} \| F^{\beta}\|_S  \|F^{\gamma}\|_{S}
\end{equation}
and 
\begin{equation}
\label{3.250}
\left\|\mathcal{F}_{x \rightarrow \xi} \mathcal{E}_{3}^{t}[F^a, F^b, F^c]\right\|_{L_{\xi}^{\infty} L_y^2} \lesssim  (1+|t|)^{-\frac{1}{2}+\frac{\delta}{100}} \max _{\{\alpha, \beta, \gamma\}=\{a, b, c\}} \|F^\alpha\|_{S'} \|F^\beta\|_S  \|F^\gamma\|_S. 
\end{equation}
Then we estimate $\left\| x\, \mathcal{E}_{3}(t)\right\|_{L_x^2 H_y^2}$. Similarly, by (\ref{3.3201}) and Lemma \ref{Lemma 6.21}, we can deduce that
\begin{equation}
\label{3.2502}
\begin{aligned}
& \left\|x \, \mathcal{E}_{3}(t)\right\|_{L_x^2 H_y^2} \\  \lesssim & \left\|x \, \mathcal{F}_{\xi,\eta}^{-1}\int_{\omega(\eta,\eta_1,\eta_2)= -2\eta} \frac{\mathrm{e}^{-2i t \eta}-1}{i \eta}  \mathcal{O}_{2}^{t}\left[F_{\eta-\eta_1}^{a} , F_{\eta_2-\eta_1}^{b}, F_{\eta_2}^{c}\right] d\eta_1 d\eta_2 \right\|_{L_{x}^2 H_y^2} \\ \lesssim &   \left\|x \, \mathcal{F}_{\xi,\eta}^{-1}\int_{\omega(\eta,\eta_1,\eta_2)= -2\eta} \frac{\mathrm{e}^{-2i t \eta}-1}{i \eta}  \mathcal{O}_{2}^{t}\left[F_{\eta-\eta_1}^{a} , F_{\eta_2-\eta_1}^{b}, F_{\eta_2}^{c}\right] d\eta_1 d\eta_2 \right\|_{L_{x,y}^2} \\ + & \left\|x \partial_y^2 \mathcal{F}_{\xi,\eta}^{-1}\int_{\omega(\eta,\eta_1,\eta_2)= -2\eta} \frac{\mathrm{e}^{-2i t \eta}-1}{i \eta}  \mathcal{O}_{2}^{t}\left[F_{\eta-\eta_1}^{a} , F_{\eta_2-\eta_1}^{b}, F_{\eta_2}^{c}\right] d\eta_1 d\eta_2 \right\|_{L_{x,y}^2} \\ \lesssim & (1+|t|)^{\frac{1}{2}} \left\|x\,\mathcal{F}_{\xi}^{-1}\mathcal{O}_{2}^{t}\left[F_{+}^{a}, F_{+}^{b}, F_{+}^{c}\right]\right\|_{L_y^1 L_x^2} \\ + & (1+|t|)^{\frac{1}{2}} \left\|x\partial_y^2\mathcal{F}_{\xi}^{-1}\mathcal{O}_{2}^{t}\left[F_{+}^{a}, F_{+}^{b}, F_{+}^{c}\right]\right\|_{L_y^1 L_x^2} \\  \lesssim & (1+|t|)^{-\frac{1}{2} + \frac{\delta}{100}}\left\|F^a\right\|_{S} \| F^b\|_S  \|F^c\|_{S}.
\end{aligned}
\end{equation}
To estimate $\left\|\mathcal{E}_{3}(t)\right\|_{Z}$, since we have (\ref{3.250}),
we only need to estimate $\left\|\mathcal{F}_{x \rightarrow \xi} \mathcal{E}_{3}^{t}[F^a, F^b, F^c]\right\|_{L_{\xi}^{\infty} \dot{B}_{y}^{1}}$. By (\ref{3.35}) and (\ref{6.21}) in Remark \ref{remark 6.6}, we have
\begin{align*} 
&  \|\mathcal{F}_{x\rightarrow\xi}{\mathcal{E}_3^{t}}[F^a, F^b, F^c](\xi,y)\|_{L_{\xi}^\infty \dot{B}_y^1 }\\ \lesssim &  (1+|t|)^{-1+\frac{\delta}{100}} \sum_{k \leq 0} 2^k \left\|\int_{-\infty}^0 \mathrm{e}^{i y \eta} \frac{\mathrm{e}^{-2 i t \eta}-1}{i \eta} \phi(\frac{\eta}{2^k}) d\eta \right\|_{L_y^1} \|F^a\|_S \|F^b\|_S \|F^c\|_S \\ + & (1+|t|)^{-1+\frac{\delta}{100}} \sum_{k > 0} 2^{-k} \left\|\int_{-\infty}^0\mathrm{e}^{i y \eta} \frac{\mathrm{e}^{-2 i t \eta}-1}{i \eta} \phi(\frac{\eta}{2^k}) d\eta \right\|_{L_y^1} \|F^a\|_S \|F^b\|_S \|F^c\|_S.
\end{align*}
In fact, the estimate of the term $\sum_{k > 0} 2^{-k} \left\|\int_{-\infty}^0\mathrm{e}^{i y \eta} \frac{\mathrm{e}^{-2 i t \eta}-1}{i \eta} \phi(\frac{\eta}{2^k}) d\eta \right\|_{L_y^1}$ is very easy. We note that $\phi^{-}(\eta):=\mathbbm{1}_{\eta \leq 0} \phi(\eta)$. We observe that $\int_{-\infty}^y \mathcal{F}_{\eta \rightarrow \tau}^{-1} \phi^{-}(\tau) d \tau$ is a Schwartz function, then for $k>0$, we have
\begin{align*}
& \left\|\int_{-\infty}^{0} \mathrm{e}^{i y \eta} \frac{\mathrm{e}^{-2 i t \eta}-1}{i \eta} \phi\left(\frac{\eta}{2^{k}}\right) d \eta\right\|_{L_{y}^{1}} \\ = & \left\|\int_{-\infty}^{0} \int_{y}^{y-2 t} \mathrm{e}^{i \tau \eta} d \tau \phi\left(\frac{\eta}{2^{k}}\right) d \eta\right\|_{L_{y}^{1}} \\ \lesssim & \left\|\int_{2^{k} y}^{2^{k}(y-2 t)} \mathcal{F}_{\eta \rightarrow \tau}^{-1} \phi^{-}(\tau) d \tau\right\|_{L_{y}^{1}} \\ \lesssim & \left\|\int_{-\infty}^{2^{k}y} \mathcal{F}_{\eta \rightarrow \tau}^{-1} \phi^{-}(\tau) d \tau - \int_{-\infty}^{2^{k}(y-2 t)} \mathcal{F}_{\eta \rightarrow \tau}^{-1} \phi^{-}(\tau) d \tau\right\|_{L_y^1}\\ \lesssim & \left\|\int_{-\infty}^{2^{k}y} \mathcal{F}_{\eta \rightarrow \tau}^{-1} \phi^{-}(\tau) d \tau\right\|_{L_y^1} \\ \lesssim & \, 2^{-k}.
\end{align*}
So we have 
\begin{equation}
\label{3.401}
\sum_{k > 0} 2^{-k} \left\|\int_{-\infty}^0\mathrm{e}^{i y \eta} \frac{\mathrm{e}^{-2 i t \eta}-1}{i \eta} \phi(\frac{\eta}{2^k}) d\eta \right\|_{L_y^1} \lesssim \sum_{k > 0} 2^{-2k} \lesssim 1.
\end{equation}
We observe that $\phi^{-}$ is a Schwartz function, then by Lemma \ref{lemma 3.81}, we have
\begin{align*}
\left\|\int_{-\infty}^{0} \mathrm{e}^{i y \eta} \frac{\mathrm{e}^{-2 i t \eta}-1}{i \eta} \phi\left(\frac{\eta}{2^{k}}\right) d \eta\right\|_{L_{y}^{1}} \lesssim 2^{-\frac{k}{2}} t^{\frac{1}{2}} \quad \text{    for  } k \leq 0.
\end{align*}
So we have
\begin{equation}
\label{3.402}
\sum_{k \leq 0} 2^k \left\|\int_{-\infty}^0 \mathrm{e}^{i y \eta} \frac{\mathrm{e}^{-2 i t \eta}-1}{i \eta} \phi(\frac{\eta}{2^k}) d\eta \right\|_{L_y^1} \lesssim t^{\frac{1}{2}} \sum_{k \leq 0} 2^{\frac{k}{2}} \lesssim t^{\frac{1}{2}}.
\end{equation}
Combining (\ref{3.401}) and (\ref{3.402}), we deduce that
\begin{align*}
\left\|\mathcal{F}_{x \rightarrow \xi} \mathcal{E}_3^{t}\left[F^{a}, F^{b}, F^{c}\right](\xi, y)\right\|_{L_{\xi}^{\infty} \dot{B}_{y}^{1}} \lesssim (1+|t|)^{-\frac{1}{2}+\frac{\delta}{100}} \left\|F^a\right\|_{S} \left\|F^b\right\|_{S}\left\|F^c\right\|_{S}.
\end{align*}
So we get the estimate for $\left\|\mathcal{E}_{3}(t)\right\|_{Z}$,
\begin{equation}
\label{3.2503}
\left\|\mathcal{E}_{3}(t)\right\|_{Z} \lesssim (1+|t|)^{-\frac{1}{2}+\frac{\delta}{100}} \left\|F^a\right\|_{S} \left\|F^b\right\|_{S}\left\|F^c\right\|_{S}.
\end{equation}
According to (\ref{3.2501}), (\ref{3.2502}) and (\ref{3.2503}), we have
\begin{equation}
\label{3.2504}
\left\|\mathcal{E}_{3}(t)\right\|_{\mathcal{Y}} \lesssim (1+|t|)^{-\frac{1}{2}+\frac{\delta}{100}} \left\|F^a\right\|_{S} \left\|F^b\right\|_{S}\left\|F^c\right\|_{S}.
\end{equation}
$\textbf{Case 2}$. For $\omega\left(\eta, \eta_{1}, \eta_{2}\right)=\pm 2 \eta_{1} \text{ or }\pm (2\eta-2 \eta_{2}) \text{ or } \pm\left(2 \eta-2 \eta_{1}+2 \eta_{2}\right)$. We only have to deal with the case $\omega((\eta,\eta_1,\eta_2) = - 2 \eta_1$, and other cases can be treated in the same way. Firstly we estimate $\left\|\mathcal{E}_{3}(t)\right\|_{L_{x, y}^{2}}$. By (\ref{3.191}), we have
\begin{align*}
& \left\|\mathcal{E}_{3}(t)\right\|_{L_{x, y}^{2}} \\  \lesssim & \left\|\int_{\omega(\eta,\eta_1\eta_2) = - 2\eta_1} \frac{\mathrm{e}^{-2i t \eta_1}-1}{i \eta_1} \mathcal{O}_{2}^{t}\left[F_{\eta-\eta_1}^{a}, F_{\eta_2-\eta_1}^{b}, F_{\eta_2}^{c}\right] d\eta_1 d\eta_2 \right\|_{ L_{\xi,\eta}^{2}} \\ \lesssim & \int_{\mathbb{R}^{2}}\left|\mathcal{F}_{\mu,\kappa}^{-1} m(z_1, z_2)\right| \left\|\left\|\mathrm{e}^{it\partial_x^2} T_{z_2}F_{-}^{a}\right\|_{L_x^2} \left\|\frac{\mathrm{e}^{-2itD_y}-1}{iD_y} \left(\overline{\mathrm{e}^{it\partial_x^2}T_{z_1+z_2} F_{+}^b} \mathrm{e}^{it\partial_x^2} T_{z_1} F_{+}^c \right) \right\|_{L_x^\infty}  \right\|_{L_y^2} dz_1 dz_2 \\ \lesssim & \int_{\mathbb{R}^{2}}\left|\mathcal{F}_{\mu,\kappa}^{-1} m(z_1, z_2)\right| \left\|F^a\right\|_{L_{x,y}^2} \left\|\frac{\mathrm{e}^{-2it D_y}-1}{i D_y} \left(\overline{\mathrm{e}^{it\partial_x^2} T_{z_1+z_2}F_{+}^b} \mathrm{e}^{it\partial_x^2} T_{z_1}F_{+}^c \right) \right\|_{L_{x,y}^\infty} dz_1 dz_2 \\ \lesssim & \int_{\mathbb{R}^{2}}\left|\mathcal{F}_{\mu,\kappa}^{-1} m(z_1, z_2)\right| \|F^a\|_{L_{x,y}^2}\left\|\frac{\mathrm{e}^{-2it \eta}-1}{i \eta} \right\|_{ L_\eta^2} \left\|\left(\overline{\mathrm{e}^{it\partial_x^2} T_{z_1+z_2}F_{+}^b} \mathrm{e}^{it\partial_x^2} T_{z_1}F_{+}^c \right) \right\|_{L_y^2L_x^\infty}dz_1 dz_2 \\ \lesssim & (1+|t|)^{-\frac{1}{2}+ \frac{\delta}{100}} \|F^a\|_{L_{x,y}^2} \|F^b\|_S \|F^c\|_S,
\end{align*}
and we also have
\begin{align*}
& \left\|\mathcal{E}_{3}(t)\right\|_{L_{x, y}^{2}} \\  \lesssim & \left\|\int_{\omega(\eta,\eta_1\eta_2) = - 2\eta_1} \frac{\mathrm{e}^{-2i t \eta_1}-1}{i \eta_1} \mathcal{O}_{2}^{t}\left[F_{\eta-\eta_1}^{a}, F_{\eta_2-\eta_1}^{b}, F_{\eta_2}^{c}\right] d\eta_1 d\eta_2 \right\|_{L_{\xi,\eta}^{2}} \\ \lesssim & \int_{\mathbb{R}^{2}}\left|\mathcal{F}_{\mu,\kappa}^{-1} m(z_1, z_2)\right| \left\|\left\|\mathrm{e}^{it\partial_x^2} T_{z_2}F_{-}^{a}\right\|_{L_x^\infty} \left\|\frac{\mathrm{e}^{-2it D_y}-1}{i D_y} \left(\overline{\mathrm{e}^{it\partial_x^2} T_{z_1+z_2}F_{+}^b} \mathrm{e}^{-it\partial_x^2} T_{z_1} F_{+}^c \right) \right\|_{L_x^2}  \right\|_{L_y^2} dz_1 dz_2 \\ \lesssim & \int_{\mathbb{R}^{2}}\left|\mathcal{F}_{\mu,\kappa}^{-1} m(z_1, z_2)\right| \left\|\mathrm{e}^{it\partial_x^2} T_{z_2} F_{-}^a\right\|_{L_{x,y}^\infty} \left\|\frac{\mathrm{e}^{-2it D_y}-1}{i D_y} \left(\overline{\mathrm{e}^{it\partial_x^2} T_{z_1+z_2}F_{+}^b} \mathrm{e}^{it\partial_x^2} T_{z_1}F_{+}^c \right) \right\|_{L_{x,y}^2} dz_1 dz_2 \\ \lesssim & \int_{\mathbb{R}^{2}}\left|\mathcal{F}_{\mu,\kappa}^{-1} m(z_1, z_2)\right| \left\|\mathrm{e}^{it\partial_x^2} T_{z_2} F_{-}^a\right\|_{L_{x,y}^\infty} \left\|\frac{\mathrm{e}^{-2it \eta}-1}{i \eta} \right\|_{L_\eta^2} \left\|\left(\overline{\mathrm{e}^{it\partial_x^2} T_{z_1+z_2}F_{+}^b} \mathrm{e}^{it\partial_x^2} T_{z_1}F_{+}^c \right) \right\|_{L_x^2 L_y^1}dz_1 dz_2 \\ \lesssim &(1+|t|)^{\frac{1}{2}} \int_{\mathbb{R}^{2}}\left|\mathcal{F}_{\mu,\kappa}^{-1} m(z_1, z_2)\right| \|\mathrm{e}^{it\partial_x^2} T_{z_2}F_{-}^a\|_{L_{x,y}^\infty} \left\|\left(\overline{\mathrm{e}^{it\partial_x^2} T_{z_1+z_2}F_{+}^b} \mathrm{e}^{it\partial_x^2} T_{z_1}F_{+}^c \right) \right\|_{L_y^1 L_x^2} dz_1 dz_2\\ \lesssim & (1+|t|)^{-\frac{1}{2}+ \frac{\delta}{100}} \|F^b\|_{L_{x,y}^2} \|F^a\|_S \|F^c\|_S.
\end{align*}
The inequality above holds by replacing $F^b$ with $F^c$. So we have
\begin{equation}
\label{3.34}
\|\mathcal{E}_3(t) \|_{L_{x,y}^2} \lesssim (1+|t|)^{-\frac{1}{2}+\frac{\delta}{100}}  \min_{\{\alpha,\beta,\gamma\} = \{a, b, c\} } \|F^\alpha\|_{L_{x,y}^2} \|F^\beta\|_S \|F^\gamma\|_S.
\end{equation}
According to Lemma {\ref{Lemma 6.2}}, we have
\begin{equation}
\label{3.2505}
\left\|\mathcal{E}_{3}(t)\right\|_{S'} \lesssim (1+|t|)^{-\frac{1}{2}+\frac{\delta}{100}} \max _{\{\alpha, \beta, \gamma\}=\{a, b, c\}}\left\|F^{\alpha}\right\|_{S'} \| F^{\beta}\|_S  \|F^{\gamma}\|_{S}
\end{equation}
and
\begin{equation}
\label{3.26}
\left\|\mathcal{F}_{x \rightarrow \xi} \mathcal{E}_{3}^{t}[F^a, F^b, F^c]\right\|_{L_{\xi}^{\infty} L_{y}^{2}} \lesssim   (1+|t|)^{-\frac{1}{2}+ \frac{\delta}{100}} \max _{\{\alpha, \beta, \gamma\}=\{a, b, c\}} \|F^\alpha\|_{S'} \|F^\beta\|_S \|F^\gamma\|_S.
\end{equation}
Then we estimate $\left\|x \, \mathcal{E}_{3}(t)\right\|_{L_x^2 H_y^2}$. Similarly as we did in the estimate of (\ref{3.34}), we have
\begin{equation}
\label{3.2506}
\begin{aligned}
& \left\|x\, \mathcal{E}_{3}(t)\right\|_{L_{x}^{2} H_y^2} \\  \lesssim & \left\| x \, \mathcal{F}_{\xi, \eta}^{-1}\int_{\omega(\eta,\eta_1\eta_2) = - 2\eta_1} \frac{\mathrm{e}^{-2i t \eta_1}-1}{i \eta_1} \mathcal{O}_{2}^{t}\left[F_{\eta-\eta_1}^{a}, F_{\eta_2-\eta_1}^{b}, F_{\eta_2}^{c}\right] d\eta_1 d\eta_2 \right\|_{L_x^2 H_y^2} \\ \lesssim & (1+|t|)^{-\frac{1}{2}+ \frac{\delta}{100}} \|F^a\|_{S} \|F^b\|_S \|F^c\|_S.
\end{aligned}
\end{equation}
To estimate $\left\|\mathcal{E}_{3}(t)\right\|_{Z}$, since we have (\ref{3.26}), we only need to estimate $\left\|\mathcal{F}_{x \rightarrow \xi} \mathcal{E}_{3}^{t}[F^a, F^b, F^c]\right\|_{L_{\xi}^{\infty} \dot{B}_{y}^{1}}$. By (\ref{6.14}), we have
\begin{align*} 
&  \|\mathcal{F}_{x\rightarrow\xi}{\mathcal{E}_{3}^{t}}[F^a, F^b, F^c](\xi,y)\|_{L_{\xi}^\infty \dot{B}_y^1 } \\ \lesssim & \left\|\mathcal{F}_{\xi,\eta}^{-1} \int_{\mathbb{R}^2}\frac{\mathrm{e}^{-2 i t \eta_1}-1}{i \eta_1} \mathcal{O}_{2}^{t}\left[F_{\eta-\eta_1}^{a}\mathbbm{1}_{\eta-\eta_1 \leq 0}, F_{\eta_2-\eta_1}^{b}\mathbbm{1}_{\eta_2-\eta_1 \geq 0}, F_{\eta_2}^{c}\mathbbm{1}_{\eta_2\geq 0}\right] d\eta_1 d\eta_2  \right\|_{L_y^1 L_x^2}^{\frac{1}{2}} \times \\ & \left\|x\,\mathcal{F}_{\xi,\eta}^{-1} \int_{\mathbb{R}^2}\frac{\mathrm{e}^{-2 i t \eta_1}-1}{i \eta_1} \mathcal{O}_{2}^{t}\left[F_{\eta-\eta_1}^{a}\mathbbm{1}_{\eta-\eta_1 \leq 0}, F_{\eta_2-\eta_1}^{b}\mathbbm{1}_{\eta_2-\eta_1 \geq 0}, F_{\eta_2}^{c}\mathbbm{1}_{\eta_2\geq 0}\right] d\eta_1 d\eta_2 \right\|_{L_y^1 L_x^2}^{\frac{1}{2}} \\ + & \left\|\partial_y^2\mathcal{F}_{\xi,\eta}^{-1} \int_{\mathbb{R}^2} \frac{\mathrm{e}^{-2i t \eta_1}-1}{i \eta_1} \mathcal{O}_{2}^{t}\left[F_{\eta-\eta_1}^{a}\mathbbm{1}_{\eta-\eta_1 \leq 0}, F_{\eta_2-\eta_1}^{b}\mathbbm{1}_{\eta_2-\eta_1 \geq 0}, F_{\eta_2}^{c}\mathbbm{1}_{\eta_2\geq 0}\right] d\eta_1 d\eta_2 \right\|_{ L_y^1 L_x^2 }^\frac{1}{2} \times \\ & \left\|x\,\partial_y^2 \mathcal{F}_{\xi,\eta}^{-1} \int_{\mathbb{R}^2}  \frac{\mathrm{e}^{-2i t \eta_1}-1}{i \eta_1} \mathcal{O}_{2}^{t}\left[F_{\eta-\eta_1}^{a}\mathbbm{1}_{\eta-\eta_1 \leq 0}, F_{\eta_2-\eta_1}^{b}\mathbbm{1}_{\eta_2-\eta_1 \geq 0}, F_{\eta_2}^{c}\mathbbm{1}_{\eta_2\geq 0}\right] d\eta_1 d\eta_2 \right\|_{L_y^1 L_x^2}^\frac{1}{2}\\ \lesssim & (1+|t|)^{-\frac{1}{2}+ \frac{\delta}{100}}\|F^a\|_S \|F^b\|_S \|F^c\|_S.
\end{align*}
The last inequality can be deduced as we did in the estimate of (\ref{3.34}).
So we have
\begin{equation}
\label{3.2507}
\left\|\mathcal{E}_{3}(t)\right\|_{Z} \lesssim (1+|t|)^{-\frac{1}{2}+\frac{\delta}{100}} \left\|F^a\right\|_{S} \left\|F^b\right\|_{S}\left\|F^c\right\|_{S}.
\end{equation}
According to (\ref{3.2505}), (\ref{3.2506}) and (\ref{3.2507}), we have 
\begin{equation}
\label{3.2508}
\left\|\mathcal{E}_{3}(t)\right\|_{\mathcal{Y}} \lesssim (1+|t|)^{-\frac{1}{2}+\frac{\delta}{100}} \left\|F^a\right\|_{S} \left\|F^b\right\|_{S}\left\|F^c\right\|_{S}.
\end{equation}
$\textbf{Case 3}$. For $\omega\left(\eta, \eta_{1}, \eta_{2}\right)=\pm\left(2 \eta-2 \eta_{1}\right) \text{ or } \pm 2 \eta_{2} \text{ or } \pm\left(2 \eta_{1}-2 \eta_{2}\right)$. We only have to deal with the case $\omega((\eta,\eta_1,\eta_2) =  -2(\eta- \eta_1)$, and other cases can be treated in the same way. Firstly we estimate $\left\|\mathcal{E}_{3}(t)\right\|_{L_{x, y}^{2}}$. By (\ref{3.191}), we have
\begin{align*}
& \left\|\mathcal{E}_{3}(t)\right\|_{L_{x, y}^{2}} \\  \lesssim & \left\|\int_{\omega(\eta,\eta_1\eta_2) = - 2(\eta-\eta_1)} \frac{\mathrm{e}^{-2i t (\eta-\eta_1})-1}{i (\eta-\eta_1)} \mathcal{O}_{2}^{t}\left[F_{\eta-\eta_1}^{a}, F_{\eta_2-\eta_1}^{b}, F_{\eta_2}^{c}\right] d\eta_1 d\eta_2 \right\|_{ L_{\xi,\eta}^{2}} \\ \lesssim & \int_{\mathbb{R}^{2}}\left|\mathcal{F}_{\mu,\kappa}^{-1} m(z_1, z_2)\right| \left\|\left\|\frac{\mathrm{e}^{-2it D_y}-1}{i D_y} \mathrm{e}^{it\partial_x^2}  T_{z_2}F_{+}^{a}\right\|_{L_x^2} \left\| \overline{\mathrm{e}^{it\partial_x^2} T_{z_1+z_2}F_{-}^b} \mathrm{e}^{it\partial_x^2} T_{z_1}F_{-}^c \right\|_{L_x^\infty}  \right\|_{L_y^2} dz_1 dz_2 \\ \lesssim & \int_{\mathbb{R}^{2}}\left|\mathcal{F}_{\mu,\kappa}^{-1} m(z_1, z_2)\right|\left\|\frac{\mathrm{e}^{-2it D_y}-1}{i D_y}\mathrm{e}^{it\partial_x^2} T_{z_2}F_{+}^a\right\|_{L_{x}^2 L_y^\infty} \left\| \overline{\mathrm{e}^{it\partial_x^2} T_{z_1+z_2}F_{-}^b} \mathrm{e}^{it\partial_x^2} T_{z_1}F_{-}^c \right\|_{L_y^2 L_{x}^\infty } dz_1 dz_2\\ \lesssim & \int_{\mathbb{R}^{2}}\left|\mathcal{F}_{\mu,\kappa}^{-1} m(z_1, z_2)\right| \left\|\frac{\mathrm{e}^{-2it \eta}-1}{i \eta}\right\|_{L_\eta^2} \left\| F^a\right\|_{L_{x,y}^2} \left\| \overline{\mathrm{e}^{it\partial_x^2} T_{z_1+z_2}F_{-}^b} \mathrm{e}^{it\partial_x^2} T_{z_1}F_{-}^c \right\|_{L_y^2 L_{x}^\infty }dz_1 dz_2 \\ \lesssim & (1+|t|)^{-\frac{1}{2}+ \frac{\delta}{100}} \|F^a\|_{L_{x,y}^2} \|F^b\|_S \|F^c\|_S,
\end{align*}
and we also have
\begin{align*}
& \left\|\mathcal{E}_{3}(t)\right\|_{L_{x, y}^{2}} \\  \lesssim & \left\|\int_{\omega(\eta,\eta_1\eta_2) = - 2(\eta-\eta_1)} \frac{\mathrm{e}^{-2i t (\eta-\eta_1})-1}{i (\eta-\eta_1)} \mathcal{O}_{2}^{t}\left[F_{\eta-\eta_1}^{a}, F_{\eta_2-\eta_1}^{b}, F_{\eta_2}^{c}\right] d\eta_1 d\eta_2 \right\|_{ L_{\xi,\eta}^{2}} \\ \lesssim & \int_{\mathbb{R}^{2}}\left|\mathcal{F}_{\mu,\kappa}^{-1} m(z_1, z_2)\right| \left\|\left\| \frac{\mathrm{e}^{-2it D_y}-1}{i D_y} \mathrm{e}^{it\partial_x^2} T_{z_2}F_{+}^{a}\right\|_{L_x^\infty} \left\| \overline{\mathrm{e}^{it\partial_x^2} T_{z_1+z_2}F_{-}^b} \mathrm{e}^{it\partial_x^2} T_{z_1}F_{-}^c \right\|_{L_x^2}  \right\|_{L_y^2} dz_1 dz_2\\ \lesssim & \int_{\mathbb{R}^{2}}\left|\mathcal{F}_{\mu,\kappa}^{-1} m(z_1, z_2)\right| \left\|\frac{\mathrm{e}^{-2it D_y}-1}{i D_y}\mathrm{e}^{it\partial_x^2} T_{z_2}F_{+}^a\right\|_{L_{x,y}^\infty} \left\| \overline{\mathrm{e}^{it\partial_x^2} T_{z_1+z_2}F_{-}^b} \mathrm{e}^{it\partial_x^2} T_{z_1}F_{-}^c \right\|_{L_{x,y}^2 } dz_1 dz_2 \\ \lesssim & \int_{\mathbb{R}^{2}}\left|\mathcal{F}_{\mu,\kappa}^{-1} m(z_1, z_2)\right| \left\|\frac{\mathrm{e}^{-2it \eta}-1}{i \eta}\right\|_{L_\eta^2} \left\| \mathrm{e}^{it\partial_x^2}T_{z_2}F_{+}^a\right\|_{L_y^2 L_{x}^\infty} \left\| \overline{\mathrm{e}^{it\partial_x^2} T_{z_1+z_2}F_{-}^b} \mathrm{e}^{it\partial_x^2} T_{z_1} F_{-}^c \right\|_{L_{x,y}^2}dz_1dz_2 \\ \lesssim & (1+|t|)^{-\frac{1}{2}+ \frac{\delta}{100}} \|F^b\|_{L_{x,y}^2} \|F^a\|_S \|F^c\|_S.
\end{align*}
The inequality above holds by replacing $F^b$ with $F^c$. So we have
\begin{equation}
\label{3.39}
\|\mathcal{E}_3(t) \|_{L_{x,y}^2} \lesssim (1+|t|)^{-\frac{1}{2}+\frac{\delta}{100}}  \min_{\{\alpha,\beta,\gamma\} = \{a, b, c\} } \|F^\alpha\|_{L_{x,y}^2} \|F^\beta\|_S \|F^\gamma\|_S.
\end{equation}
According to Lemma \ref{Lemma 6.2}, we have
\begin{equation}
\label{3.2509}
\left\|\mathcal{E}_{3}(t)\right\|_{S'} \lesssim (1+|t|)^{-\frac{1}{2}+\frac{\delta}{100}} \max _{\{\alpha, \beta, \gamma\}=\{a, b, c\}}\left\|F^{\alpha}\right\|_{S'} \| F^{\beta}\|_S  \|F^{\gamma}\|_{S}
\end{equation}
and
\begin{equation}
\label{3.260}
\left\|\mathcal{F}_{x \rightarrow \xi} \mathcal{E}_{3}^{t}[F^a, F^b, F^c]\right\|_{L_{\xi}^{\infty} L_{y}^{2}} \lesssim   (1+|t|)^{-\frac{1}{2}+ \frac{\delta}{100}} \max _{\{\alpha, \beta, \gamma\}=\{a, b, c\}} \|F^\alpha\|_{S'} \|F^\beta\|_S \|F^\gamma\|_S.
\end{equation}
Then we estimate $\left\|x \, \mathcal{E}_{3}(t)\right\|_{L_x^2 H_y^2}$. Similarly as we did in the estimate of (\ref{3.39}), we have
\begin{equation}
\label{3.2510}
\begin{aligned}
& \left\|x\, \mathcal{E}_{3}(t)\right\|_{L_{x}^{2} H_y^2} \\  \lesssim & \left\|x\, \mathcal{F}_{\xi,\eta}^{-1} \int_{\omega(\eta,\eta_1\eta_2) = - 2(\eta-\eta_1)} \frac{\mathrm{e}^{-2i t (\eta-\eta_1)}-1}{i (\eta-\eta_1)} \mathcal{O}_{2}^{t}\left[F_{\eta-\eta_1}^{a}, F_{\eta_2-\eta_1}^{b}, F_{\eta_2}^{c}\right] d\eta_1 d\eta_2 \right\|_{L_x^2 H_y^2} \\ \lesssim & (1+|t|)^{-\frac{1}{2}+ \frac{\delta}{100}} \|F^a\|_{S} \|F^b\|_S \|F^c\|_S.
\end{aligned}
\end{equation}
To estimate $\left\|\mathcal{E}_{3}(t)\right\|_{Z}$, since we have (\ref{3.260}),
we only need to estimate $\left\|\mathcal{F}_{x \rightarrow \xi} \mathcal{E}_{3}^{t}[F, G, H]\right\|_{L_{\xi}^{\infty} \dot{B}_{y}^{1}}$. By (\ref{6.14}), we have
\begin{align*} 
&  \|\mathcal{F}_{x\rightarrow\xi}{\mathcal{E}^{t}}[F^a, F^b, F^c](\xi,y)\|_{L_{\xi}^\infty \dot{B}_y^1 } \\ \lesssim & \left\|\mathcal{F}_{\xi,\eta}^{-1} \int_{\mathbb{R}^2}\frac{\mathrm{e}^{-2 i t (\eta-\eta_1)}-1}{i (\eta-\eta_1)} \mathcal{O}_{2}^{t}\left[F_{\eta-\eta_1}^{a}\mathbbm{1}_{\eta-\eta_1 \geq 0}, F_{\eta_2-\eta_1}^{b}\mathbbm{1}_{\eta_2-\eta_1 \leq 0}, F_{\eta_2}^{c}\mathbbm{1}_{\eta_2\leq 0}\right] d\eta_1 d\eta_2  \right\|_{L_{y}^1 L_x^2}^{\frac{1}{2}} \times \\ & \left\|x\,\mathcal{F}_{\xi,\eta}^{-1} \int_{\mathbb{R}^2}\frac{\mathrm{e}^{-2 i t (\eta-\eta_1)}-1}{i (\eta-\eta_1)} \mathcal{O}_{2}^{t}\left[F_{\eta-\eta_1}^{a}\mathbbm{1}_{\eta-\eta_1 \geq 0}, F_{\eta_2-\eta_1}^{b}\mathbbm{1}_{\eta_2-\eta_1 \leq 0}, F_{\eta_2}^{c}\mathbbm{1}_{\eta_2\leq 0}\right] d\eta_1 d\eta_2  \right\|_{L_{y}^1 L_x^2}^{\frac{1}{2}} \\ + & \left\|\partial_y^2 \mathcal{F}_{\xi,\eta}^{-1} \int_{\mathbb{R}^2}  \frac{\mathrm{e}^{-2i t (\eta-\eta_1)}-1}{i (\eta-\eta_1)} \mathcal{O}_{2}^{t}\left[F_{\eta-\eta_1}^{a}\mathbbm{1}_{\eta-\eta_1 \geq 0}, F_{\eta_2-\eta_1}^{b}\mathbbm{1}_{\eta_2-\eta_1 \leq 0}, F_{\eta_2}^{c}\mathbbm{1}_{\eta_2\leq 0}\right] d\eta_1 d\eta_2 \right\|_{ L_{y}^1 L_x^2}^\frac{1}{2} \times \\ & \left\|x\,\partial_y^2 \mathcal{F}_{\xi,\eta}^{-1} \int_{\mathbb{R}^2} \frac{\mathrm{e}^{-2i t (\eta-\eta_1)}-1}{i (\eta-\eta_1)} \mathcal{O}_{2}^{t}\left[F_{\eta-\eta_1}^{a}\mathbbm{1}_{\eta-\eta_1 \geq 0}, F_{\eta_2-\eta_1}^{b}\mathbbm{1}_{\eta_2-\eta_1 \leq 0}, F_{\eta_2}^{c}\mathbbm{1}_{\eta_2\leq 0}\right] d\eta_1 d\eta_2 \right\|_{ L_{y}^1 L_x^2}^\frac{1}{2}\\ \lesssim & (1+|t|)^{-\frac{1}{2}+ \frac{\delta}{100}}\|F^a\|_S \|F^b\|_S \|F^c\|_S.
\end{align*}
The last inequality can be deduced as we did in the estimate of (\ref{3.39}).
So we have
\begin{equation}
\label{3.2511}
\left\|\mathcal{E}_{3}(t)\right\|_{Z} \lesssim (1+|t|)^{-\frac{1}{2}+\frac{\delta}{100}} \left\|F^a\right\|_{S} \left\|F^b\right\|_{S}\left\|F^c\right\|_{S}.
\end{equation}
According to (\ref{3.2509}), (\ref{3.2510}) and (\ref{3.2511}), we have
\begin{equation}
\label{3.2512}
\left\|\mathcal{E}_{3}(t)\right\|_{\mathcal{Y}} \lesssim (1+|t|)^{-\frac{1}{2}+\frac{\delta}{100}} \left\|F^a\right\|_{S} \left\|F^b\right\|_{S}\left\|F^c\right\|_{S}.
\end{equation}
In general, after the analysis of $\textbf{Case 1}$, $\textbf{Case 2}$ and $\textbf{Case 3}$, according to (\ref{3.2504}), (\ref{3.2508}) and (\ref{3.2512}), with the assumption (\ref{3.3}) and $\frac{T}{2}\leq t \leq T$, we have
\begin{equation}
\label{3.2513}
\left\|\mathcal{E}_{3}(t)\right\|_{\mathcal{Y}}  \lesssim |T|^{-\frac{1}{3}}.
\end{equation}
$\textbf { 2. Estimation of } \widetilde{\mathcal{E}}_{1}(t) $. We need to control the $\mathcal{Y}$ norm. $\widetilde{\mathcal{E}}_{1}(t)$ is divided in two parts, one is from $\mathcal{O}_{1}^{t}$, and another one $\mathcal{L}^{t}$ is from the last four terms in (\ref{3.23}). We recall (\ref{3.25}) with 
\begin{align*} i \omega \mathcal{L}^{t}\left[f^{a}, f^{b}, f^{c}\right]:=&-\left(\partial_{t} \mathcal{O}_{2}^{t}\right)\left[f^{a}, f^{b}, f^{c}\right]-\mathcal{O}_{2}^{t}\left[\partial_{t} f^{a}, f^{b}, f^{c}\right]-\mathcal{O}_{2}^{t}\left[f^{a}, \partial_{t} f^{b}, f^{c}\right] \\ &-\mathcal{O}_{2}^{t}\left[f^{a}, f^{b}, \partial_{t} f^{c}\right]. \end{align*}
The term $\int_{\omega(\eta,\eta_1,\eta_2) \neq 0}(\mathrm{e}^{i t \omega(\eta,\eta_1,\eta_2)}-1)\mathcal{L}^{t}\left[F_{\eta-\eta_1}^{a}, F_{\eta_2-\eta_1}^{b}, F_{\eta_2}^{c}\right]d\eta_1 d\eta_2 $ can be estimated similarly as $\left\|\mathcal{E}_{3}(t)\right\|_{\mathcal{Y}}$. In fact, we can gain a better estimate here. For the first term, as $\frac{T}{2}\leq t \leq T$, we get an extra $T^{-3 / 4}$  which comes from the $t$ derivative of the multiplier, while for the other three terms, by the definition of $X_T$ norm, we have $\left\|\partial_{t} F\right\|_{\mathcal{Y}} \leq T^{-1+3 \delta}\|F\|_{X_{T}}$.\\\\
We then estimate $\int_{\omega(\eta,\eta_1,\eta_2) \neq 0}\mathrm{e}^{i t \omega(\eta,\eta_1,\eta_2)}\mathcal{O}_{1}^{t}\left[F_{\eta-\eta_1}^{a}, F_{\eta_2-\eta_1}^{b}, F_{\eta_2}^{c}\right]d\eta_1 d\eta_2$.\\\\
Let $\frac{T}{2}\leq t\leq T$. For $\mathcal{O}_{1}^{t}$, we still have
\begin{equation}
\label{3.56}
\left\|\mathcal{O}_{1}^{t}\left[f^{a}, f^{b}, f^{c}\right]\right\|_{L_\xi^{2}} \lesssim(1+|t|)^{\frac{\delta}{100}} \min _{\{\alpha, \beta, \gamma\}=\{a, b, c\}}\left\|f^{\alpha}\right\|_{L_{x}^{2}}\left\|\mathrm{e}^{i t \partial_{x}^2} f^{\beta}\right\|_{L_{x}^{\infty}}\left\|\mathrm{e}^{i t \partial_{x}^2} f^{\gamma}\right\|_{L_{x}^{\infty}}.
\end{equation}
Here $\mathcal{O}_{1}^{t}$ enjoys a similar estimate as $\mathcal{O}_{2}^{t}$ because $\mathcal{O}_{1}^{t}+\mathcal{O}_{2}^{t}$ enjoys a bound better than (\ref{3.56}).

We then estimate $\left\|\mathrm{e}^{i t \partial_{x}^2} f\right\|_{L_{x}^{\infty}}$. We notice that for all $\frac{1}{2}<\alpha \leq 1$,
\begin{equation}
\label{3.30}
\left\|\mathrm{e}^{i t \partial_{x}^2} f\right\|_{L_x^{\infty}(\mathbb{R})} \lesssim\langle t\rangle^{-\frac{1}{2}}\|f\|_{L_x^{1}(\mathbb{R})} \lesssim\langle t\rangle^{-\frac{1}{2}}\left\|\langle x\rangle^{-\alpha}\langle x\rangle^{\alpha} f\right\|_{L_x^{1}(\mathbb{R})} \lesssim\langle t\rangle^{-\frac{1}{2}}\left\|\langle x\rangle^{\alpha} f\right\|_{L_x^{2}(\mathbb{R})}.
\end{equation}
By taking $\alpha = \frac{7}{9}$, for $f$ supported on $|x| \geq R$, we have
\begin{equation}
\label{3.31}
\left\|\mathrm{e}^{i t \partial_{x}^2} f\right\|_{L_x^{\infty}} \lesssim\langle t\rangle^{-\frac{1}{2}} R^{-\frac{1}{9}}\left\|\langle x\rangle^{\frac{8}{9}} f\right\|_{L_x^{2}}.
\end{equation}
Then we decompose $f=f_{c}+f_{e} \text { with } f_{c}(x):=\varphi\left(\frac{x}{T^{\frac{1}{24}}}\right) f(x)$, we have
\begin{align*}
\mathcal{O}_{1}^{t}\left[f^{a}, f^{b}, f^{c}\right]=\mathcal{O}_{1}^{t}\left[f_{c}^{a}+f_{e}^{a}, f_{c}^{b}+f_{e}^{b}, f_{c}^{c}+f_{e}^{c}\right].
\end{align*}
If one of $f^{a}, f^{b}, f^{c} \text { is supported on }|x| \geq 2 T^{\frac{1}{24}}$, in this case we may assume $f^{b}=f_{e}^{b}$. By (\ref{3.31}), we have
\begin{equation}
\label{3.571}
\begin{aligned}\left\|\mathcal{O}_{1}^{t}\left[f^{a}, f_{e}^{b}, f^{c}\right]\right\|_{L_{x}^{2}} & \lesssim(1+|t|)^{\frac{\delta}{100}}\left\|f^{a}\right\|_{L_x^{2}}\left\|\mathrm{e}^{i t \partial_{x}^2} f_{e}^{b}\right\|_{L_x^{\infty}}\left\|\mathrm{e}^{i t \partial_{x}^2} f^{c}\right\|_{L_x^{\infty}} \\ & \lesssim(1+|t|)^{\frac{\delta}{100}}\left\|f^{a}\right\|_{L_x^{2}}\left\|\mathrm{e}^{i t \partial_{x}^2} f_{e}^{b}\right\|_{L_x^{\infty}}\left\|\mathrm{e}^{i t \partial_{x}^2} f^{c}\right\|_{L_x^{\infty}} \\ & \lesssim T^{-1-\frac{1}{216}+\frac{\delta}{100}}\left\|f^{a}\right\|_{L_x^{2}}\left\|\langle x\rangle^{\frac{8}{9}} f_{e}^{b}\right\|_{L_x^{2}}\left\|\langle x\rangle^{\frac{7}{9}} f^{c}\right\|_{L_x^{2}}.  
\end{aligned}
\end{equation}
Now we only have to deal with the case that two of $f^a, f^b, f^c$ is supported on $|x| \leq 2 T^{\frac{1}{24}}$. We can consider the case $f^b = f_c^b, f^c = f_c^c$. By replacing $\mathrm{e}^{2 i t \eta \kappa}$ by $(2 i t \mu)^{-1} \partial_{\kappa}\left(\mathrm{e}^{2 i t \mu \kappa}\right)$, we rewrite $\mathcal{O}_{1}^{t}$ as 
\begin{equation}
\begin{aligned} \mathcal{O}_{1}^{t} & {\left[f^{a}, f^{b}, f^{c}\right](\xi)=\int_{\mathbb{R}^{2}} \mathrm{e}^{2 i t \mu \kappa}\left(1-\varphi\left(t^{\frac{3}{4}} \mu \kappa\right)\right) \widehat{f^{a}}(\xi-\mu) \overline{\widehat{f^{b}}}(\xi-\mu-\kappa) \widehat{f^{c}}(\xi-\kappa) d \mu d \kappa } \\ &=\int_{\mathbb{R}^{2}}(2 i t \mu)^{-1} \partial_{\kappa}\left(\mathrm{e}^{2 i t \mu \kappa}\right)\left(1-\varphi\left(t^{\frac{3}{4}} \mu \kappa\right)\right) \widehat{f^{a}}(\xi-\mu) \overline{\widehat{f^{b}}}(\xi-\mu-\kappa) \widehat{f^{c}}(\xi-\kappa) d \mu d \kappa \\ &=-\int_{\mathbb{R}^{2}}(2 i t \mu)^{-1} \mathrm{e}^{2 i t \mu \kappa} \partial_{\kappa}\left(\left(1-\varphi\left(t^{\frac{3}{4}} \mu \kappa\right)\right) \widehat{f^{a}}(\xi-\mu) \overline{\widehat{f^{b}}}(\xi-\mu-\kappa) \widehat{f^{c}}(\xi-\kappa)\right) d \mu d \kappa. \end{aligned}
\end{equation}
In the case when the $\kappa$ derivative falls on $1-\varphi$, which is 
\begin{align*}
(2 i)^{-1} t^{-\frac{1}{4}} \int_{\mathbb{R}^{2}} \mathrm{e}^{2 i t \mu \kappa} \varphi^{\prime}\left(t^{\frac{3}{4}} \mu \kappa\right) \widehat{f^{a}}(\xi-\mu)\overline{\widehat{f^{b}}}(\xi-\mu-\kappa) \widehat{f^{c}}(\xi-\kappa) d \mu d \kappa,
\end{align*}
since $\varphi^{\prime}$ admits similar properties as $\varphi$, we have the similar estimate to the estimate we did for $\mathcal{O}_{2}^{t}$. 

For other cases, we deal with the case when $\kappa$ derivative falls on $f^b$ for example, which is denoted by $\mathcal{O}_{1, b}$,
\begin{equation}
\begin{aligned} \mathcal{O}_{1, b} &\left.:=-\int_{\mathbb{R}^{2}}(2 i t \eta)^{-1} \mathrm{e}^{2 i t \eta \kappa}\left(1-\varphi\left(t^{\frac{3}{4}} \mu \kappa\right)\right) \widehat{f^{a}}(\xi-\mu) \partial_{\kappa} \overline{\widehat{f^{b}}}(\xi-\mu-\kappa)\right) \widehat{f^{c}}(\xi-\kappa) d \mu d \kappa \\ &=-\int_{\mathbb{R}^{2}}(2 t \mu)^{-1} \mathrm{e}^{2 i t \mu \kappa}\left(1-\varphi\left(t^{\frac{3}{4}} \mu  \kappa\right)\right) \widehat{f^{a}}(\xi-\mu) \overline{\widehat{x f^{b}}}(\xi-\mu-\kappa) \widehat{f^{c}}(\xi-\kappa) d \mu d \kappa. \end{aligned}
\end{equation}
We observe that $|t||\mu| \gtrsim|t|^{\frac{1}{4}}|\kappa|^{-1} \gtrsim T^{\frac{1}{12}}$ on the support of the integration, so we still have an $L^2$ estimate
\begin{equation}
\label{3.359}
\left\|\mathcal{O}_{1, b}\right\|_{L_{\xi}^{2}} \lesssim T^{-\frac{1}{12}+\frac{\delta}{100}}\left\|f^{a}\right\|_{L_x^{2}} \cdot\left\|\mathrm{e}^{i t \partial_{x}^2}\left(x f^{b}\right)\right\|_{L_{x}^{\infty}} \cdot\left\|\mathrm{e}^{i t \partial_{x}^2} f^{c}\right\|_{L_{x}^{\infty}}.
\end{equation}
By (\ref{3.30}), if $f$ supported on $|x| \lesssim T^{\frac{1}{24}}$, we have
\begin{equation}
\label{3.36}
\left\|e^{i t \partial_{x}^2} x f\right\|_{L^{\infty}} \lesssim\langle t\rangle^{-\frac{1}{2}} T^{\frac{1}{24}}\left\|\langle x\rangle^{\frac{7}{9}} f\right\|_{L_{x}^{2}}.
\end{equation}
By (\ref{3.30}), (\ref{3.359}) and (\ref{3.36}), we have
\begin{equation}
\label{3.62}
\begin{aligned}
\left\|\mathcal{O}_{1, b}\right\|_{L_{\xi}^{2}} & \lesssim T^{-\frac{1}{12}+\frac{\delta}{100}}\left\|f^{a}\right\|_{L^{2}} \cdot\left\|\mathrm{e}^{i t \partial_{x}^2}\left(x f_{c}^{b}\right)\right\|_{L_{x}^{\infty}} \cdot\left\|\mathrm{e}^{i t \partial_{x}^2} f^{c}\right\|_{L_{x}^{\infty}} \\ & \lesssim T^{-\frac{25}{24}+\frac{\delta}{100}}\left\|f^{a}\right\|_{L_{x}^{2}}\left\|\langle x\rangle^{\frac{7}{9}} f^{b}\right\|_{L_{x}^{2}}\left\|\langle x\rangle^{\frac{7}{9}} f^{c}\right\|_{L_{x}^{2}}. 
\end{aligned}
\end{equation}
Combining (\ref{3.571}) and (\ref{3.62}), by replacement and Remark \ref{remark 6.41}, we have
\begin{equation}
\label{3.2601}
\begin{aligned}
 & \left\|\mathcal{F}_{\xi,\eta}^{-1}\int_{\omega(\eta,\eta_1,\eta_2) \neq 0}  \mathrm{e}^{i t \omega(\eta,\eta_1,\eta_2)} \mathcal{O}_{1}^{t}\left[F_{\eta-\eta_1}^{a}, F_{\eta_2-\eta_1}^{b}, F_{\eta_2}^{c}\right]d\eta_1 d\eta_2\right\|_{S} \\ \lesssim & \sum_{\{i,j,k\} = \{\pm,\pm,\pm\}}\left\|\mathcal{F}_{\xi \rightarrow x}^{-1}\mathcal{O}_{1}^{t}\left[\mathrm{e}^{-i t\left|D_{y}\right|}F_i^{a}, \mathrm{e}^{-i t\left|D_{y}\right|}F_j^{b}, \mathrm{e}^{-i t\left|D_{y}\right|}F_k^{c}\right]\right\|_{S} \\  \lesssim & \quad T^{-1-23\delta} \left\|F^{a}\right\|_{S}\left\|F^{b}\right\|_{S}\left\|F^{c}\right\|_{S}
\end{aligned}
\end{equation}
and
\begin{equation}
\label{3.2600}
\begin{aligned}
& \left\|\mathcal{F}_{\eta \to y}^{-1}\int_{\omega(\eta,\eta_1,\eta_2) \neq 0}  \mathrm{e}^{i t \omega(\eta,\eta_1,\eta_2)} \mathcal{O}_{1}^{t}\left[F_{\eta-\eta_1}^{a}, F_{\eta_2-\eta_1}^{b}, F_{\eta_2}^{c}\right]d\eta_1 d\eta_2\right\|_{L_\xi^\infty L_y^2} \\ \lesssim & \quad T^{-1-23\delta} \left\|F^{a}\right\|_{S}\left\|F^{b}\right\|_{S}\left\|F^{c}\right\|_{S}.
\end{aligned}
\end{equation}
To estimate $\left\|\mathcal{F}_{\xi, \eta}^{-1} \int_{\omega\left(\eta, \eta_{1}, \eta_{2}\right) \neq 0} \mathrm{e}^{i t \omega\left(\eta, \eta_{1}, \eta_{2}\right)} \mathcal{O}_{1}^{t}\left[F_{\eta-\eta_{1}}^{a}, F_{\eta-\eta_{1}-\eta_{2}}^{b}, F_{\eta-\eta_{2}}^{c}\right] d \eta_{1} d \eta_{2}\right\|_{Z}$, since we have (\ref{3.2600}), we only have to estimate 
\begin{align*}
\left\|\mathcal{F}_{\eta \to y}^{-1} \int_{\omega\left(\eta, \eta_{1}, \eta_{2}\right) \neq 0} \mathrm{e}^{i t \omega\left(\eta, \eta_{1}, \eta_{2}\right)} \mathcal{O}_{1}^{t}\left[F_{\eta-\eta_{1}}^{a}, F_{\eta-\eta_{1}-\eta_{2}}^{b}, F_{\eta-\eta_{2}}^{c}\right] d \eta_{1} d \eta_{2}\right\|_{L_\xi^\infty \dot{B}_y^1} .
\end{align*}
By (\ref{3.571}), (\ref{3.62}), Remark \ref{remark 6.41} and  (\ref{6.14}) in Remark \ref{remark 6.6}, we have
\begin{align*}
& \left\|\mathcal{F}_{\eta \to y}^{-1} \int_{\omega\left(\eta, \eta_{1}, \eta_{2}\right) \neq 0} \mathrm{e}^{i t \omega\left(\eta, \eta_{1}, \eta_{2}\right)} \mathcal{O}_{1}^{t}\left[F_{\eta-\eta_{1}}^{a}, F_{\eta-\eta_{1}-\eta_{2}}^{b}, F_{\eta-\eta_{2}}^{c}\right] d \eta_{1} d \eta_{2}\right\|_{L_\xi^\infty \dot{B}_y^1}
\\ \lesssim & \sum_{\{i, j, k\}=\{\pm, \pm, \pm\}}\left\|\mathcal{F}_{\xi \rightarrow x}^{-1}\mathcal{O}_{1}^{t}\left[\mathrm{e}^{-i t\left|D_{y}\right|}F_i^{a}, \mathrm{e}^{-i t\left|D_{y}\right|}F_j^{b}, \mathrm{e}^{-i t\left|D_{y}\right|}F_k^{c}\right]\right\|_{L_{y}^{1} L_{x}^{2}}^{\frac{1}{2}} \times \\ & \qquad \qquad \qquad \quad\left\|x \mathcal{F}_{\xi \rightarrow x}^{-1}\mathcal{O}_{1}^{t}\left[\mathrm{e}^{-i t\left|D_{y}\right|}F_i^{a}, \mathrm{e}^{-i t\left|D_{y}\right|}F_j^{b}, \mathrm{e}^{-i t\left|D_{y}\right|}F_k^{c}\right]\right\|_{L_{y}^{1} L_{x}^{2}}^{\frac{1}{2}} \\ + & \sum_{\{i, j, k\}=\{\pm, \pm, \pm\}} \left\|\mathcal{F}_{\xi \rightarrow x}^{-1}\partial_y^2 \mathcal{O}_{1}^{t}\left[\mathrm{e}^{-i t\left|D_{y}\right|}F_i^{a}, \mathrm{e}^{-i t\left|D_{y}\right|}F_j^{b}, \mathrm{e}^{-i t\left|D_{y}\right|}F_k^{c}\right]\right\|_{L_{y}^{1} L_{x}^{2}}^{\frac{1}{2}} \times \\ & \qquad \qquad\qquad \quad\left\| x\mathcal{F}_{\xi \rightarrow x}^{-1}\partial_y^2\mathcal{O}_{1}^{t}\left[x\mathrm{e}^{-i t\left|D_{y}\right|}F_i^{a}, \mathrm{e}^{-i t\left|D_{y}\right|}F_j^{b}, \mathrm{e}^{-i t\left|D_{y}\right|}F_k^{c}\right]\right\|_{L_{y}^{1} L_{x}^{2}}^{\frac{1}{2}} \\ \lesssim  & \quad T^{-1-23 \delta} \left\|F^{a}\right\|_{S}\left\|F^{b}\right\|_{S}\left\|F^{c}\right\|_{S}.
\end{align*}
So we have
\begin{equation}
\label{3.2603}
\begin{aligned}
& \left\|\mathcal{F}_{\xi, \eta}^{-1} \int_{\omega\left(\eta, \eta_{1}, \eta_{2}\right) \neq 0} \mathrm{e}^{i t \omega\left(\eta, \eta_{1}, \eta_{2}\right)} \mathcal{O}_{1}^{t}\left[F_{\eta-\eta_{1}}^{a}, F_{\eta-\eta_{1}-\eta_{2}}^{b}, F_{\eta-\eta_{2}}^{c}\right] d \eta_{1} d \eta_{2}\right\|_{Z} \\ \lesssim & \quad T^{-1-23 \delta}\left\|F^{a}\right\|_{S}\left\|F^{b}\right\|_{S}\left\|F^{c}\right\|_{S}.
\end{aligned}
\end{equation}
Then according to (\ref{3.2601}) and (\ref{3.2603}), with the assumption (\ref{3.3}) and $\frac{T}{2} \leq t \leq T$, we have
\begin{equation}
\label{3.2514}
\begin{aligned}
\left\|\int_{\omega\left(\eta, \eta_{1}, \eta_{2}\right) \neq 0} \mathrm{e}^{i t \omega\left(\eta, \eta_{1}, \eta_{2}\right)} \mathcal{O}_{1}^{t}\left[F_{\eta-\eta_{1}}^{a}, F_{\eta-\eta_{1}-\eta_{2}}^{b}, F_{\eta-\eta_{2}}^{c}\right] d \eta_{1} d \eta_{2}\right\|_{\mathcal{Y}} \lesssim & \, T^{-1-23 \delta} \left\|F^{a}\right\|_{S}\left\|F^{b}\right\|_{S}\left\|F^{c}\right\|_{S}\\ \lesssim & \, T^{-1-20\delta}.
\end{aligned}
\end{equation}
Thus, with the assumption (\ref{3.3}) and $\frac{T}{2} \leq t \leq T$, we have 
\begin{equation}
\label{3.2515}
\left\|\widetilde{\mathcal{E}_{1}}(t)\right\|_{\mathcal{Y}} \lesssim|T|^{-1-20\delta}.
\end{equation}
In summary, by the estimation of $\widetilde{\mathcal{E}}_{1}(t)$ and $\mathcal{E}_3(t)$, with the assumption (\ref{3.3}), we have (\ref{3.2513}) and (\ref{3.2515}), then we deduce (\ref{3.225}). We have completed the proof of Lemma {\ref{Lemma 3.5}}.
\end{proof}

\subsection{The resonant level sets}
We now consider the resonant part (\ref{3.18}),
\begin{align*}
\mathcal{F} \mathcal{N}_{0}^{t}(\xi, \eta):=  \int_{\omega(\eta,\eta_1,\eta_2) = 0} \mathcal{F}_{x}\left(\mathcal{I}^{t}\left[F_{\eta-\eta_{1}}, G_{\eta-\eta_1-\eta_{2}}, H_{\eta-\eta_{2}}\right]\right)(\xi) d \eta_{1} d \eta_{2}.
\end{align*}
We refer to Remark \ref{remark 3.1} for the description of the set $\omega(\eta,\eta_1,\eta_2) = 0$. This part yields the main contribution in Proposition \ref{Proposition 3.1} and in particular is responsible for the slowest $\frac{1}{t}$ decay. We show that it gives rise to a contribution which grows slowly in $S$ and that it can be well approximated by the resonant system in the $Z$ norm. \\\\
We also define a norm, which is very useful in the following lemma,
\begin{equation}
\|F\|_{\widetilde{Z}_{t}}:=\|F\|_{Z}+(1+|t|)^{-2\delta}\|F\|_{S}.
\end{equation}
We observe that $F(t)$ remains uniformly bounded in $\widetilde{Z}_{t}$ under the assumption of Proposition \ref{Proposition 3.1}.
\begin{lemma}
\label{Lemma 3.7}
Let $t \geq 1$, then we have
\begin{equation}
\label{3.38}
\left\|\mathcal{N}_{0}^{t}\left[F^{a}, F^{b}, F^{c}\right]\right\|_{S'} \lesssim(1+|t|)^{-1} \sum_{\{\alpha, \beta, \gamma\}=\{a, b, c\}}\left\|F^{\alpha}\right\|_{S'} \cdot\left\|F^{\beta}\right\|_{\widetilde{Z}_{t}} \cdot\left\|F^{\gamma}\right\|_{\widetilde{Z}_{t}},
\end{equation}
\begin{equation}
\label{3.385}
\begin{aligned}
\left\|x\mathcal{N}_{0}^{t}\left[F^{a}, F^{b}, F^{c}\right]\right\|_{L_x^2 H_y^2}  & \lesssim (1+|t|)^{-1} \left\| x F^a \right\|_{L_{x}^2 H_y^2} \left\|F^b \right\|_{\widetilde{Z}_{t}} \left\| F^c \right\|_{\widetilde{Z}_{t}} \\ & + (1+|t|)^{-1} \|F^a\|_{S'} \|F^b\|_{S^+} \|F^c\|_{\widetilde{Z}_{t}} \\ & + (1+|t|)^{-1} \|F^a\|_{S'}\|F^b\|_{\widetilde{Z}_{t}} \|F^c\|_{S^+}  \\ & + (1+|t|)^{-1} \|F^a\|_{\widetilde{Z}_{t}} \|F^b\|_{S} \|F^c\|_{S},
\end{aligned}
\end{equation}
here $a,b,c$ can be replaced by each other. We also have
\begin{equation}
\label{3.41}
\left\|\mathcal{N}_{0}^{t}[F, G, H]-\frac{\pi}{t} \mathcal{R}[F, G, H]\right\|_{Z} \lesssim(1+|t|)^{-1-23 \delta}\|F\|_{S}\|G\|_{S}\|H\|_{S},
\end{equation}
\begin{equation}
\label{3.42}
\left\|\mathcal{N}_{0}^{t}[F, G, H]-\frac{\pi}{t} \mathcal{R}[F, G, H]\right\|_{\mathcal{Y}} \lesssim(1+|t|)^{-1-23 \delta}\|F\|_{S^{+}}\|G\|_{S^{+}}\|H\|_{S^{+}}.
\end{equation}
\end{lemma}
\begin{proof}
By Lemma \ref{Lemma 4.1} and Proposition \ref{Proposition 4.2} in Section 4, we have
\begin{align*}
\left\|\mathcal{N}_{0}^{t}\left[F^{a}, F^{b}, F^{c}\right]\right\|_{L_{x, y}^{2}} & \lesssim \left\|\int_{\omega\left(\eta, \eta_{1}, \eta_{2}\right)=0} \mathcal{I}^{t}\left[F_{\eta-\eta_{1}}^a, F_{\eta-\eta_{1}-\eta_{2}}^b, F_{\eta-\eta_{2}}^c\right] d \eta_{1} d \eta_{2}\right\|_{L_{x,\eta}^2}\\ & \lesssim \left\|\mathcal{I}^{t}\left[F_{+}^a, F_{+}^b, F_{+}^c\right] \right\|_{L_{x,y}^2} + \left\|\mathcal{I}^{t}\left[F_{-}^a, F_{-}^b, F_{-}^c \right] \right\|_{L_{x,y}^2}\\ & \lesssim \min _{\{\alpha, \beta, \gamma\}=\{a,b,c\}} \left\| F^\alpha \right\|_{L_{x,y}^2} \left\|\mathrm{e}^{i t \partial_{x}^2}F_{+}^\beta\right\|_{L_{x,y}^\infty} \left\| \mathrm{e}^{i t \partial_{x}^2} F_{+}^\gamma \right\|_{L_{x,y}^\infty} \\ & + \min _{\{\alpha, \beta, \gamma\}=\{a,b,c\}} \left\| F^\alpha \right\|_{L_{x,y}^2} \left\|\mathrm{e}^{i t \partial_{x}^2}F_{-}^\beta\right\|_{L_{x,y}^\infty} \left\| \mathrm{e}^{i t \partial_{x}^2} F_{-}^\gamma \right\|_{L_{x,y}^\infty}. 
\end{align*}
For $t\geq1$, we have
\begin{equation}
\label{3.425}
\left|\mathrm{e}^{i t \partial_{x}^2} f(x)-c \frac{\mathrm{e}^{-i x^{2} /(4 t)}}{\sqrt{t}} \widehat{f}\left(-\frac{x}{2 t}\right)\right| \lesssim|t|^{-3 / 4}\|x f\|_{L_{x}^{2}}, \quad c \text { is a constant. }
\end{equation}
We refer to [\ref{[7]}, Lemma 7.3] for the proof of (\ref{3.425}). Then 
\begin{equation}
\label{3.426}
\left|\mathrm{e}^{i t \partial_{x}^2} f(x)\right| \lesssim|t|^{-1 / 2} \|\widehat{f}(\xi)\|_{L_\xi^\infty}+|t|^{-3 / 4}\|x f\|_{L_{x}^{2}}.
\end{equation}
Thus
\begin{align*}
\left\|\mathrm{e}^{i t \partial_{x}^2} F_{\pm}\right\|_{L_{x,y}^\infty}\lesssim & \, |t|^{-1 / 2} \|\widehat{F}_{\pm}(\xi, \cdot)\|_{L_\xi^\infty L_{y}^\infty} + |t|^{-3 / 4}\|x F_{\pm}\|_{L_y^\infty L_x^{2}} \\ \lesssim & \, |t|^{-1 / 2} \|\widehat{F}(\xi, y)\|_{L_\xi^\infty\mathcal{G}_y} + |t|^{-3 / 4}\|x F\|_{L_{x}^2 H_y^1}  \\ \lesssim &\,  |t|^{-1 / 2}\|F\|_{Z}+|t|^{-3 / 4}\|F\|_{S}, \end{align*}
where we use $\Pi_{\pm}(\mathcal{G}) \subset \mathcal{G}  \subset L^\infty$. 
Then we have
\begin{equation}
\left\|\mathcal{N}_{0}^{t}\left[F^{a}, F^{b}, F^{c}\right]\right\|_{L_{x, y}^{2}} \lesssim(1+|t|)^{-1} \min _{\{\alpha, \beta, \gamma\}=\{a, b, c\}}\left\|F^{\alpha}\right\|_{L_{x, y}^{2}}\left\|F^{\beta}\right\|_{\widetilde{Z}_{t}}\left\|F^{\gamma}\right\|_{\widetilde{Z}_{t}}.
\end{equation}
By Lemma \ref{Lemma 6.2}, we obtain (\ref{3.38}). \\\\
Then we estimate $\left\|x \mathcal{N}_{0}^{t}\left[F^{a}, F^{b}, F^{c}\right]\right\|_{L_x^2 H_y^2}$. We have
\begin{align*}
\left\|x \mathcal{N}_{0}^{t}\left[F^{a}, F^{b}, F^{c}\right]\right\|_{L_{x}^{2} H_y^2} & \lesssim \left\|(1+|\eta|^2)\int_{\omega\left(\eta, \eta_{1}, \eta_{2}\right)=0} x \mathcal{I}^{t}\left[F_{\eta-\eta_{1}}^a, F_{\eta-\eta_{1}-\eta_{2}}^b, F_{\eta-\eta_{2}}^c\right] d \eta_{1} d \eta_{2}\right\|_{L_{x,\eta}^2}\\ & \lesssim \left\|\mathcal{I}^{t}\left[x F_{+}^{a}, F_{+}^{b}, F_{+}^{c}\right]\right\|_{L_{x}^{2}H_y^2} + \left\|\mathcal{I}^{t}\left[x F_{-}^{a}, F_{-}^{b}, F_{-}^{c}\right]\right\|_{L_{x}^{2} H_y^2}\\ & + \left\|\mathcal{I}^{t}\left[F_{+}^{a}, x F_{+}^{b}, F_{+}^{c}\right]\right\|_{L_{x}^{2}H_y^2} + \left\|\mathcal{I}^{t}\left[ F_{-}^{a}, x F_{-}^{b}, F_{-}^{c}\right]\right\|_{L_{x}^{2} H_y^2}\\ & + \left\|\mathcal{I}^{t}\left[F_{+}^{a}, F_{+}^{b}, x F_{+}^{c}\right]\right\|_{L_{x}^{2}H_y^2} + \left\|\mathcal{I}^{t}\left[ F_{-}^{a}, F_{-}^{b}, x F_{-}^{c}\right]\right\|_{L_{x}^{2} H_y^2}.
\end{align*}
For the first two terms above, we estimate $\left\|\mathcal{I}^{t}\left[x F_{+}^{a}, F_{+}^{b}, F_{+}^{c}\right](\xi)\right\|_{L_{x}^{2}H_y^2}$ for example.  Again by (\ref{3.426}), we have
\begin{align*}
\left\|\mathcal{I}^{t}\left[x F_{+}^{a}, F_{+}^{b}, F_{+}^{c}\right]\right\|_{L_{x}^{2}H_y^2} \lesssim & \left\|\mathcal{I}^{t}\left[ (1-\partial_y^2)x F_{+}^{a}, F_{+}^{b}, F_{+}^{c}\right]\right\|_{L_{x,y}^2} \\ + & \left\|\mathcal{I}^{t}\left[x F_{+}^{a}, (1-\partial_y^2)F_{+}^{b}, F_{+}^{c}\right]\right\|_{L_{x,y}^2} \\ + & \left\|\mathcal{I}^{t}\left[x F_{+}^{a}, F_{+}^{b}, (1-\partial_y^2) F_{+}^{c}\right]\right\|_{L_{x,y}^2}  \\ \lesssim &  \|xF^a\|_{L_{x}^{2} H_y^2} \left\|\mathrm{e}^{i t \partial_{x}^2}F_{+}^b\right\|_{L_{x,y}^\infty}\left\|\mathrm{e}^{i t \partial_{x}^2}F_{+}^c \right\|_{L_{x,y}^\infty} \\ + & \|xF^a\|_{L_{x,y}^2} \left\|(1-\partial_y^2) \mathrm{e}^{i t \partial_{x}^2}F_{+}^b\right\|_{L_{x,y}^\infty}\left\|\mathrm{e}^{i t \partial_{x}^2}F_{+}^c \right\|_{L_{x,y}^\infty} \\ + & \|xF^a\|_{L_{x,y}^2} \left\| \mathrm{e}^{i t \partial_{x}^2}F_{+}^b\right\|_{L_{x,y}^\infty}\left\|(1-\partial_y^2)\mathrm{e}^{i t \partial_{x}^2}F_{+}^c \right\|_{L_{x,y}^\infty} \\ \lesssim & \frac{1}{|t|} \|xF^a\|_{L_{x}^{2} H_y^2} \left\|F^b\right\|_{\widetilde{Z}_{t}}\left\|F^c \right\|_{\widetilde{Z}_{t}} \\ + & \frac{1}{|t|} \|F^a\|_{S'} \left\|F^b\right\|_{S^+}\left\|F^c \right\|_{\widetilde{Z}_{t}} \\ + & \frac{1}{|t|} \|F^a\|_{S'} \left\|F^b\right\|_{\widetilde{Z}_{t}}\left\|F^c \right\|_{S^+} .
\end{align*}
For the term $\left\|\mathcal{I}^{t}\left[x F_{-}^{a}, F_{-}^{b}, F_{-}^{c}\right](\xi)\right\|_{L_{x}^{2} H_{y}^{2}}$, we can use the same method to get the same estimate. 

For the last four terms, we estimate $\left\|\mathcal{I}^{t}\left[F_{+}^{a}, x F_{+}^{b}, F_{+}^{c}\right](\xi)\right\|_{L_{x}^{2} H_{y}^{2}}$ for example. We have
\begin{align*}
\left\|\mathcal{I}^{t}\left[F_{+}^{a}, x F_{+}^{b}, F_{+}^{c}\right](\xi)\right\|_{L_{x}^{2} H_{y}^{2}} \lesssim & \left\|\mathcal{I}^{t}\left[ (1-\partial_y^2) F_{+}^{a}, x F_{+}^{b}, F_{+}^{c}\right](\xi)\right\|_{L_{x,y}^2} \\ + & \left\|\mathcal{I}^{t}\left[ F_{+}^{a}, (1-\partial_y^2)x F_{+}^{b}, F_{+}^{c}\right](\xi)\right\|_{L_{x,y}^2} \\ + & \left\|\mathcal{I}^{t}\left[ F_{+}^{a},x F_{+}^{b}, (1-\partial_y^2) F_{+}^{c}\right](\xi)\right\|_{L_{x,y}^2} \\ \lesssim & \|F^a\|_{L_x^2 H_y^2} \|\mathrm{e}^{i t \partial_{x}^2}(xF_{+}^{b})\|_{L_{x,y}^\infty} \|\mathrm{e}^{i t \partial_{x}^2} F_{+}^c\|_{L_{x,y}^\infty} \\ + &  \|\mathrm{e}^{i t \partial_{x}^2} F_{+}^a\|_{L_{x,y}^\infty} \|xF^{b}\|_{L_{x}^2 H_y^2} \|\mathrm{e}^{i t \partial_{x}^2} F_{+}^c\|_{L_{x,y}^\infty} \\ + & \|\mathrm{e}^{i t \partial_{x}^2} F_{+}^a\|_{L_{x,y}^\infty} \|xF^{b}\|_{L_y^\infty L_{x}^2 } \|(1-\partial_y^2)\mathrm{e}^{i t \partial_{x}^2} F_{+}^c\|_{L_y^2 L_{x}^\infty} \\ \lesssim & \frac{1}{|t|} \|F^a\|_{S'}  \|F^b\|_{S^+} \|F^c\|_{\widetilde{Z}_{t}} \\ + & \frac{1}{|t|} \|F^a\|_{\widetilde{Z}_{t}} \|F^b\|_{S}\|F^c\|_{S}.
\end{align*}
 Other three terms can be estimated in the same way and we get the same estimate as the estimate of $\left\|\mathcal{I}^{t}\left[F_{+}^{a}, x F_{+}^{b}, F_{+}^{c}\right](\xi)\right\|_{L_{x}^{2} H_{y}^{2}}$. Thus we get (\ref{3.385}), and we observe that $a,b,c$ can be replaced by each other so that the inequality still holds.\\\\
To prove (\ref{3.41}) and (\ref{3.42}), we decompose the functions as follows
\begin{align*}
F=F_{c}+F_{e}, \text { with } F_{c} \text { compactly supported as } F_{c}=\varphi\left(\frac{x}{t^{\frac{1}{4}}}\right) F.
\end{align*}
We start by estimating the $Z$ norm and the $\mathcal{Y}$ norm of 
\begin{align*}
\mathcal{N}_{0}^{t}[F, G, H]-\mathcal{N}_{0}^{t}\left[F_{c}, G_{c}, H_{c}\right] \quad \text { and } \quad \mathcal{R}[F, G, H]-\mathcal{R}\left[F_{c}, G_{c}, H_{c}\right].
\end{align*}
In this case, we only have to estimate $\mathcal{N}_{0}^{t}\left[F_{e}, G, H\right]$ and $\frac{1}{t} \mathcal{R}\left[F_{e}, G, H\right]$. According to (\ref{2.16}), (\ref{2.1601}), (\ref{4.701}) and (\ref{4.702}), for the norm $Z$ and the norm $\mathcal{Y}$, we have
\begin{equation}
\begin{aligned}
\left\|\mathcal{N}_{0}^{t}\left[F_{e}, G, H\right]\right\|_{Z}+\frac{1}{t}\left\|\mathcal{R}\left[F_{e}, G, H\right]\right\|_{Z} & \lesssim \left\|\mathcal{N}^{t}\left[F_{(e,+)}, G_{+}, H_{+}\right]\right\|_{Z} + \left\|\mathcal{N}^{t}\left[F_{(e,-)}, G_{-}, H_{-}\right]\right\|_{Z} \\ & + \frac{1}{t}\left\|\mathcal{R}\left[F_{e}, G, H\right]\right\|_{Z}
\\ & \lesssim(1+|t|)^{-1}\left\|F_{e}\right\|_{L_x^2 H_{y}^2}^{\frac{1}{4}}\|F\|_S^{\frac{3}{4}}\|G\|_{S}\|H\|_{S} \\ & \lesssim(1+|t|)^{-\frac{17}{16}}\|F\|_{S}\|G\|_{S}\|H\|_{S}, \end{aligned}
\end{equation}
\begin{equation}
\begin{aligned}\left\|\mathcal{N}_{0}^{t}\left[F_{e}, G, H\right]\right\|_{\mathcal{Y}}+\frac{1}{t}\left\|\mathcal{R}\left[F_{e}, G, H\right]\right\|_{\mathcal{Y}} & \lesssim(1+|t|)^{-1}\left\|F_{e}\right\|_{S}\|G\|_{S}\|H\|_{S} \\ & \lesssim(1+|t|)^{-\frac{5}{4}}\|F\|_{S^{+}}\|G\|_{S^{+}}\|H\|_{S^{+}}. \end{aligned}
\end{equation}
Then we need to prove the inequalities below to complete our proof of this lemma,
\begin{equation}
\label{3.52}
\left\|\mathcal{N}_{0}^{t}\left[F_{c}, G_{c}, H_{c}\right]-\frac{\pi}{t} \mathcal{R}\left[F_{c}, G_{c}, H_{c}\right]\right\|_{Z} \lesssim(1+|t|)^{-1-23 \delta}\|F\|_{S}\|G\|_{S}\|H\|_{S},
\end{equation}
\begin{equation}
\label{3.53}
\left\|\mathcal{N}_{0}^{t}\left[F_{c}, G_{c}, H_{c}\right]-\frac{\pi}{t} \mathcal{R}\left[F_{c}, G_{c}, H_{c}\right]\right\|_{\mathcal{Y}} \lesssim(1+|t|)^{-1-23 \delta}\|F\|_{S^{+}}\|G\|_{S^{+}}\|H\|_{S^{+}}.
\end{equation}
With the assumption $F=F_{c}, G=G_{c}, H=H_{c}$, we have
\begin{equation}
\begin{aligned}& \mathcal{F} \left(\mathcal{N}_{0}^{t}[F, G, H]-\frac{\pi}{t} \mathcal{R}[F, G, H]\right)(\xi, \eta) \\ = & \int_{\omega(\eta,\eta_1,\eta_2) = 0} \bigg(\int_{\mathbb{R}^{2}} e^{i t 2 \mu \kappa} \widehat{F}_{\eta-\eta_1}(\xi-\mu) \widehat{G}_{\eta - \eta_1-\eta_2}(\xi-\mu-\kappa) \widehat{H}_{\eta_2}(\xi-\kappa) d \kappa d \mu \\- & \frac{\pi}{t} \widehat{F}_{\eta-\eta_1}(\xi) \widehat{G}_{\eta_2-\eta_1}(\xi) \widehat{H}_{\eta_2}(\xi)\bigg)d\eta_1 d\eta_2. \end{aligned}
\end{equation}
By rewriting the integration part, we have
\begin{align*}
&\int_{\mathbb{R}^{2}} \mathrm{e}^{i t 2 \mu \kappa} \widehat{F}_{\eta-\eta_1}(\xi-\mu) \overline{\widehat{G}}_{\eta_2-\eta_1}(\xi-\mu-\kappa) \widehat{H}_{\eta_2}(\xi-\kappa) d \kappa d \mu \\ = &\frac{1}{8\pi^3} \int_{\mathbb{R}^{3}} F_{\eta-\eta_1}\left(x_{1}\right) \overline{G}_{\eta_2-\eta_1}\left(x_{2}\right) H_{\eta_2}\left(x_{3}\right) \int_{\mathbb{R}^{2}} \mathrm{e}^{i t 2 \mu \kappa} \mathrm{e}^{-i x_{1}(\xi-\mu)+i x_{2}(\xi-\mu-\kappa)-i x_{3}(\xi-\kappa)} d \kappa d \mu d x_{1} d x_{2} d x_{3} \\ = & \frac{1}{16\pi^3 t} \int_{\mathbb{R}^{3}} F_{\eta-\eta_1}\left(x_{1}\right) \overline{G}_{\eta_2-\eta_1}\left(x_{2}\right) H_{\eta_2}\left(x_{3}\right) \mathrm{e}^{-i \xi\left(x_{1}-x_{2}+x_{3}\right)} \mathrm{e}^{-i \frac{x_{1}-x_{2}}{\sqrt{2 t}} \frac{x_{3}-x_{2}}{\sqrt{2 t}}} \\ & \times\left\{\int_{\mathbb{R}^{2}} \mathrm{e}^{i\left[\mu+\frac{x_{3}-x_{2}}{\sqrt{2 t}}\right]\left[\kappa+\frac{x_{1}-x_{2}}{\sqrt{2 t}}\right]} d \mu d \kappa\right\} d x_{1} d x_{2} d x_{3} \\ = & \frac{1}{8\pi^2t}\int_{\mathbb{R}^{3}} F_{\eta-\eta_1}\left(x_{1}\right) \overline{G}_{\eta_2-\eta_1}\left(x_{2}\right) H_{\eta_2}\left(x_{3}\right) \mathrm{e}^{-i \xi\left(x_{1}-x_{2}+x_{3}\right)} \mathrm{e}^{-i \frac{x_{1}-x_{2}}{\sqrt{2 t}} \frac{x_{3}-x_{2}}{\sqrt{2 t}}} d x_{1} d x_{2} d x_{3},
\end{align*}
then 
\begin{align*} &\left|\frac{1}{8\pi^2t}\int_{\mathbb{R}^{2}} \mathrm{e}^{i t 2 \mu \kappa} \widehat{F}_{\eta-\eta_1}(\xi-\mu) \overline{\widehat{G}}_{\eta_2-\eta_1}(\xi-\mu-\kappa) \widehat{H}_{\eta_2}(\xi-\kappa) d \kappa d \mu-\frac{\pi}{t} \widehat{F}_{\eta-\eta_1}(\xi) \widehat{G}_{\eta_2-\eta_1}(\xi) \widehat{H}_{\eta_2}(\xi)\right| \\ & =\frac{1}{8\pi^2|t|}\left|\int_{\mathbb{R}^{3}} F_{\eta-\eta_1}\left(x_{1}\right) \overline{G}_{\eta_2-\eta_1}\left(x_{2}\right) H_{\eta_2-\eta_1}\left(x_{3}\right) \mathrm{e}^{-i \xi\left(x_{1}-x_{2}+x_{3}\right)}\left(\mathrm{e}^{-i \frac{x_{1}-x_{2}}{\sqrt{2 t}} \frac{x_{3}-x_{2}}{\sqrt{2 t}}}-1\right) d x_{1} d x_{2} d x_{3}\right|\\ & \lesssim |t|^{-1} \left\|F_{\eta-\eta_1}\right\|_{L_{x}^{1}} \left\|G_{\eta_2-\eta_1}\right\|_{L_{x}^{1}}\left\|H_{\eta_2}\right\|_{L_{x}^{1}} \left\|\varphi\left(\frac{x_1}{t^{\frac{1}{4}}}\right) \varphi\left(\frac{x_2}{t^{\frac{1}{4}}}\right) \varphi\left(\frac{x_3}{t^{\frac{1}{4}}}\right)\left(\mathrm{e}^{-i \frac{x_{1}-x_{2}}{\sqrt{2 t}} \frac{x_{3}-x_{2}}{\sqrt{2 t}}}-1\right)\right\|_{L_{x_1,x_2,x_3}^\infty} \\ & \lesssim |t|^{-\frac{5}{8}} \left\|F_{\eta-\eta_1}\right\|_{L_{x}^{2}}\left\|G_{\eta_2-\eta_1}\right\|_{L_{x}^{2}}\left\|H_{\eta_2}\right\|_{L_{x}^{2}} \left\|\varphi\left(\frac{x_1}{t^{\frac{1}{4}}}\right) \varphi\left(\frac{x_2}{t^{\frac{1}{4}}}\right) \varphi\left(\frac{x_3}{t^{\frac{1}{4}}}\right)\left(\mathrm{e}^{-i \frac{x_{1}-x_{2}}{\sqrt{2 t}} \frac{x_{3}-x_{2}}{\sqrt{2 t}}}-1\right)\right\|_{L_{x_1,x_2,x_3}^\infty}
\\ & \lesssim|t|^{-\frac{9}{8}}\left\|F_{\eta-\eta_1}\right\|_{L_{x}^{2}}\left\|G_{\eta_2-\eta_1}\right\|_{L_{x}^{2}}\left\|H_{\eta_2}\right\|_{L_{x}^{2}}. \end{align*}
In fact, by using the previous method, we can obtain for any integer $m$,
\begin{equation}
\label{3.55}
\begin{aligned} & |\xi|^m\left|\frac{1}{8\pi^2 t}\int_{\mathbb{R}^{2}} \mathrm{e}^{i t 2 \mu \kappa} \widehat{F}_{\eta-\eta_1}(\xi-\mu) \overline{\widehat{G}}_{\eta_2-\eta_1}(\xi-\mu-\kappa) \widehat{H}_{\eta_2}(\xi-\kappa) d \kappa d \mu-\frac{\pi}{t} \widehat{F}_{\eta-\eta_1}(\xi) \widehat{G}_{\eta_2-\eta_1}(\xi) \widehat{H}_{\eta_2}(\xi)\right|\\ & = \frac{1}{8\pi^2|t|}\left|\int_{\mathbb{R}^{3}} F_{\eta-\eta_1}\left(x_{1}\right) \overline{G}_{\eta_2-\eta_1}\left(x_{2}\right) H_{\eta_2-\eta_1}\left(x_{3}\right) \xi^m \mathrm{e}^{-i \xi\left(x_{1}-x_{2}+x_{3}\right)}\left(\mathrm{e}^{-i \frac{x_{1}-x_{2}}{\sqrt{2 t}} \frac{x_{3}-x_{2}}{\sqrt{2 t}}}-1\right) d x_{1} d x_{2} d x_{3}\right| \\ & = \frac{1}{8\pi^2|t|}\left|\int_{\mathbb{R}^{3}} F_{\eta-\eta_1}\left(x_{1}\right) \overline{G}_{\eta_2-\eta_1}\left(x_{2}\right) H_{\eta_2-\eta_1}\left(x_{3}\right) \frac{d^m( \mathrm{e}^{-i \xi\left(x_{1}-x_{2}+x_{3}\right)})}{dx_1^m}\left(\mathrm{e}^{-i \frac{x_{1}-x_{2}}{\sqrt{2 t}} \frac{x_{3}-x_{2}}{\sqrt{2 t}}}-1\right) d x_{1} d x_{2} d x_{3}\right| \\ & \lesssim |t|^{-1} \sum_{k=0}^m \left|\int_{\mathbb{R}^{3}} \frac{d^{k}F_{\eta-\eta_1}\left(x_{1}\right)}{dx_1^k} \overline{G}_{\eta_2-\eta_1}\left(x_{2}\right) H_{\eta_2-\eta_1}\left(x_{3}\right)\mathrm{e}^{-i \xi\left(x_{1}-x_{2}+x_{3}\right)}\frac{d^{m-k}(\mathrm{e}^{-i \frac{x_{1}-x_{2}}{\sqrt{2 t}} \frac{x_{3}-x_{2}}{\sqrt{2 t}}})}{dx_1^{m-k}} d x_{1} d x_{2} d x_{3}\right|   \\ & \lesssim|t|^{-\frac{9}{8}}\left\|F_{\eta-\eta_1}\right\|_{H_{x}^{m}}\left\|G_{\eta_2-\eta_1}\right\|_{L_{x}^{2}}\left\|H_{\eta_2}\right\|_{L_{x}^{2}}. 
\end{aligned}
\end{equation}
Here $F$, $G$ and $H$ can be replaced by each other. By the definition of $Z$ norm and $S$ norm, for (\ref{3.52}) and (\ref{3.53}), the terms 
\begin{align*}
\left\|\mathcal{N}_{0}^{t}\left[F_{c}, G_{c}, H_{c}\right]-\frac{\pi}{t} \mathcal{R}\left[F_{c}, G_{c}, H_{c}\right]\right\|_{S} \quad \text { and } \quad\left\|\mathcal{N}_{0}^{t}\left[F_{c}, G_{c}, H_{c}\right]-\frac{\pi}{t} \mathcal{R}\left[F_{c}, G_{c}, H_{c}\right]\right\|_{Z}
\end{align*}
are easy to estimate by using (\ref{3.55}). In fact, when $x$ derivative falls on $\varphi\left(\frac{x}{t^{\frac{1}{4}}}\right)$, since $\varphi'$ holds the similar properties as $\varphi$, (\ref{3.55}) still works, and we get the estimate (\ref{3.52}) and the estimate (\ref{3.53}). The proof is complete.
\end{proof}
\subsection{Proof of Proposition \ref{Proposition 3.1}}
Here we give the proof of Proposition \ref{Proposition 3.1}.
\begin{proof}
We have
\begin{align*} \mathcal{N}^{t}[F, G, H]=& \sum_{ \max (A, B, C) \geq T^{\frac{1}{6}} } \mathcal{N}^{t}\left[Q_{A} F(t), Q_{B} G(t), Q_{C} H(t)\right] \\ &+\widetilde{\mathcal{N}}^{t}\left[Q_{\leq T^{\frac{1}{6}}} F(t), Q_{\leq T^{\frac{1}{6}}} G(t), Q_{\leq T^{\frac{1}{6}}} H(t)\right] \\ &+\mathcal{N}_{0}^{t}\left[Q_{\leq T^{\frac{1}{6}}} F(t), Q_{\leq T^{\frac{1}{6}}} G(t), Q_{\leq T^{\frac{1}{6}}} H(t)\right]. \end{align*}
We rewrite the last term as
\begin{align*} \mathcal{N}_{0}^{t} & {\left[Q_{\leq T^{\frac{1}{6}}} F(t), Q_{\leq T^{\frac{1}{6}}} G(t), Q_{\leq T^{\frac{1}{6}}} H(t)\right]=\frac{\pi}{t} \mathcal{R}[F(t), G(t), H(t)] } \\ &+\left(\mathcal{N}_{0}^{t}\left[ F(t),  G(t), H(t)\right] -\frac{\pi}{t} \mathcal{R}\left[ F(t), G(t), H(t)\right]\right)\\ & -\sum_{\max (A, B, C) \geq T^{\frac{1}{6}}} \mathcal{N}_0^{t}\left[Q_{A} F(t), Q_{B} G(t), Q_{C} H(t)\right]. \end{align*}
So we have the formula for the remainder 
\begin{align*} \mathcal{E}^{t}[F, G, H]=& \sum_{\max (A, B, C) \geq T^{\frac{1}{6}}} \mathcal{N}^{t}\left[Q_{A} F(t), Q_{B} G(t), Q_{C} H(t)\right] \\ &+\widetilde{\mathcal{N}}^{t}\left[Q_{\leq T^{\frac{1}{6}}} F(t), Q_{\leq T^{\frac{1}{6}}} G(t), Q_{\leq T^{\frac{1}{6}}} H(t)\right] \\ &+\left(\mathcal{N}_{0}^{t}\left[ F(t),  G(t), H(t)\right] -\frac{\pi}{t} \mathcal{R}\left[ F(t), G(t), H(t)\right]\right)\\ & -\sum_{\max (A, B, C) \geq T^{\frac{1}{6}}} \mathcal{N}_0^{t}\left[Q_{A} F(t), Q_{B} G(t), Q_{C} H(t)\right]. \end{align*}
We observe that the first term contributes to $\mathcal{E}_{1}$ by Lemma \ref{Lemma 3.3}, and the second term which contains $\mathcal{E}_{2}$ can be weitten by Lemma \ref{Lemma 3.5} as $\widetilde{\mathcal{E}}_{1}+\mathcal{E}_{2}$ with $\widetilde{\mathcal{E}}_{1}$ contributing to $\mathcal{E}_{1}$. The third term contributes to $\mathcal{E}_{1}$ by Lemma \ref{Lemma 3.7}. For the last term, we observe that
\begin{align*}
\mathcal{N}_0^{t}\left[ F(t), G(t), H(t)\right] = \Pi_{+}\left(\mathcal{N}^{t}\left[ F_{+}(t), G_{+}(t), H_{+}(t)\right]\right) + \Pi_{-}\left(\mathcal{N}^{t}\left[ F_{-}(t), G_{-}(t), H_{-}(t)\right]\right),
\end{align*}
so we can easily deduce that it enjoys the same estimate as in Lemma {\ref{Lemma 3.3}}, which contributes to $\mathcal{E}_{1}$. The proof is complete.
\end{proof}
\section{The resonant system}
In this section, we study the following equation, which contains the resonant part of the nonlinearity,
\begin{equation}
\label{4.1}
i \partial_{t} G=\mathcal{R}[G, G, G],
\end{equation}
where 
\begin{align*}
\mathcal{F} \mathcal{R}[F, G, H](\xi, \eta)=\int_{\omega(\eta, \eta_1, \eta_2) = 0 }  \widehat{{F}}_{\eta-\eta_1}(\xi)\overline{\widehat{G}}_{\eta_2-\eta_1}(\xi)\widehat{{H}}_{\eta_2}(\xi)d\eta_1 d\eta_2.
\end{align*}
Firstly, we recall a useful result on the structure of the resonances. 
\begin{lemma}
[\ref{[6]}, Lemma 2.1]
\label{Lemma 4.1}
The set of $\left(\eta_{1}, \eta_{2}, \eta_{3}, \eta_{4}\right) \in \mathbb{R}^{4}$ such that $\eta_{1}-\eta_{2}+\eta_{3}-\eta_{4}=0$ and $|\eta_{1}|-|\eta_{2}|+|\eta_{3}|-|\eta_{4}|=0$, is 

\leftline{$\begin{aligned} &\text { (1) } \forall j, \eta_{j} \geq 0, \\ & \text { (2) } \forall j, \eta_{j} \leq 0, \\ & \text { (3) } \eta_{1}=\eta_{2}, \eta_{3}=\eta_{4}, \\ & \text { (4) } \eta_{1}=\eta_{4}, \eta_{3}=\eta_{2}. \end{aligned}$}

\end{lemma}
Then we introduce the following proposition, which allows us to transform (\ref{4.1}) to the non-chiral Szeg\H o equation.
\begin{proposition}
\label{Proposition 4.2}
Given $G_{0} \in L_{x}^{2} H_{y}^{s}, s>1,\left\|G_{0}\right\|_{L_{x}^{2} H_{y}^{s}}=\varepsilon, \varepsilon>0$. Denote the corresponding solution to the resonant system (\ref{4.1}) by $G(t)$. Then $G_\pm(t)$ saitisfy the following cubic Szeg\H o equation
\begin{equation}
\label{4.2}
i \partial_{t} G_{\pm}=\mathcal{R}_{\pm}\left[G_{\pm}, G_{\pm}, G_{\pm}\right],
\end{equation}
where
\begin{equation}
\mathcal{F}_{x} \mathcal{R}_{\pm}\left[G_{\pm}, G_{\pm}, G_{\pm}\right](\xi, y)=\Pi_{\pm}\left(|\widehat{G}_\pm|^{2} \widehat{G}_\pm \right)(\xi, y).
\end{equation}
with 
\begin{align*}
\mathcal{F}_{y}(G_{+})(\eta) = \mathcal{F}_{y}(\Pi_{+} G)(\eta)=\left\{\begin{array}{l}G_{\eta}, \text { if } \eta \geq 0, \\ 0, \text { if } \eta<0.\end{array}\right.\text{ and    } \quad \mathcal{F}_{y}(G_{-})(\eta) = \mathcal{F}_{y} (\Pi_{-} G)(\eta)=\left\{\begin{array}{l}G_{\eta}, \text { if } \eta \leq 0, \\ 0, \text { if } \eta > 0. \end{array}\right.
\end{align*}
\end{proposition}
\begin{proof}
By Lemma \ref{Lemma 4.1}, we have that $\omega(\eta,\eta_1,\eta_2) = 0$ in the following cases:
\begin{align*}
\begin{array}{l}\text { If } \eta>0 \text { and }(\eta_1,\eta_2) \in\left\{ \eta \geq \eta_1, \eta \geq \eta_1+\eta_2, \eta \geq \eta_1 \right\} \cup\{\eta=\eta-\eta_1\} \cup\{\eta=\eta_2\}, \\ \text { If } \eta<0 \text { and }(\eta_1, \eta_2) \in\left\{ \eta \leq \eta_1, \eta \leq \eta_1+\eta_2, \eta \leq \eta_2 \right\} \cup\{\eta=\eta-\eta_1\} \cup\{\eta=\eta_2\}. \end{array}
\end{align*}
Since, for $\eta \in \mathbb{R}$, the sets $\{(\eta_1,\eta_2) \in \mathbb{R}^2 | \eta=\eta-\eta_1\} $ and $\{(\eta_1,\eta_2) \in \mathbb{R}^2 | \eta=\eta_2\}$ are of measure zero in $\mathbb{R}^2$, they do not interfere in the integration in equation (\ref{4.1}), and so we can neglect them. We are therefore left with the following two terms:\\
1. The case $\eta > 0, \eta-\eta_1\geq0, \eta_2-\eta_1\geq0, \eta_2\geq0$:
\begin{align*}
\mathcal{F}_{x,y} \mathcal{R}[G, G, G](\xi, \eta)=  \mathcal{F}_{y}\left( \Pi_{+}\left(|\widehat{G}_{+}|^{2} \widehat{G}_{+}\right)\right)(\eta).
\end{align*}
2.The case $\eta < 0, \eta-\eta_1\leq0, \eta_2-\eta_1\leq0, \eta_2\leq0$:
\begin{align*}
\mathcal{F}_{x,y} \mathcal{R}[G, G, G](\xi, \eta)= \mathcal{F}_{y}\left( \Pi_{-}\left(|\widehat{G}_{-}|^{2} \widehat{G}_{-}\right)\right)(\eta).
\end{align*}
Thus we have the decoupling.
\end{proof}
\subsection{Lax pairs for the cubic Szeg\H o equation}
We introduce the following cubic Szeg\H o equation on the line
\begin{equation}
\label{4.4}
i \partial_{t} v(t,x)=\Pi_{+}\left(|v(t,x)|^{2} v(t,x)\right), \qquad t,x \in \mathbb{R}.
\end{equation}
We recall the Lax pair structure and its conserved quantities for the cubic Szeg\H o equation (\ref{4.4}). To define the Lax pairs, we introduce the Hankel $H_v$ operator and the Toeplitz operator $T_b$ with $v \in H_{+}^{\frac{1}{2}}(\mathbb{R}), b \in L^{\infty}(\mathbb{R})$,
\begin{equation}
\label{4.401}
H_{v} h:=\Pi_{+}(v \overline{h}), T_{b} h:=\Pi_{+}(b h), h \in L^{\infty}.
\end{equation}
We observe that $H_v$ is $\mathbb{C}$-antilinear, and is a Hilbert-Schmidt operator. Now we are able to introduce the Lax pair structure of the cubic Szeg\H o equation (\ref{4.4}).
\begin{theorem}
$\text { Let } v \in C\left(\mathbb{R}, H_{+}^{s}(\mathbb{R})\right) \text { for some } s>\frac{1}{2} . \text { The cubic Szeg\H o }$equation (\ref{4.4}) has a Lax pair $(H_v, B_v)$, if $v$ solves (\ref{4.4}), then
\begin{equation}
\frac{d H_{v}}{d t}=\left[B_{v}, H_{v}\right],
\end{equation}
where $B_{v}=\frac{i}{2} H_{v}^{2}-i T_{|v|^{2}}$.
\end{theorem}
An important property of this Lax pair structure is that the spectrum of the trace class operator $\sqrt{H_{v}^{2}}$, is conserved by the evolution, in particular, the trace norm of $\sqrt{H_{v}^{2}}$ is conserved by the flow. A theorem by Peller [\ref{[5]}] says that the trace norm of $\sqrt{H_{v}^{2}}$ is equivalent to the homogeneous Besov norm $\dot{B}_{1,1}^1(\mathbb{R})$ of $v$, and we can deduce that the norm $\dot{B}_{1,1}^1(\mathbb{R})\cap L^2(\mathbb{R})$ is conserved by the flow.
\subsection{Estimation of solutions to the resonant system}
Just as we did in Lemma \ref{Lemma 2.1}, we can deduce a similar estimate for $\mathcal{R}$.
\begin{lemma}
\begin{equation}
\label{4.701}
\left\|\mathcal{R}[G^1, G^2, G^3]\right\|_{\mathcal{Y}} \lesssim \|G^1\|_{S}\|G^2\|_{S}\|G^3\|_{S}.
\end{equation}
In particuler,
\begin{equation}
\label{4.702}
\begin{aligned}
& \left\|\mathcal{R}[G^1, G^2, G^3]\right\|_{Z} \\ \lesssim & \, \|G^1\|_{L_{x}^{2} H_{y}^{2}}^{\frac{1}{4}}\|x G^1\|_{L_{x}^{2} H_{y}^{2}}^{\frac{1}{4}}\|G^2\|_{L_{x}^{2} H_{y}^{2}}^{\frac{1}{4}}\|x G^2\|_{L_{x}^{2} H_{y}^{2}}^{\frac{1}{4}}\|G^3\|_{L_{x}^{2} H_{y}^{2}}^{\frac{1}{4}}\|x G^3\|_{L_{x}^{2} H_{y}^{2}}^{\frac{1}{4}} \|G^1\|_S^{\frac{1}{2}}\|G^2\|_S^{\frac{1}{2}}\|G^3\|_S^{\frac{1}{2}} \\ \lesssim & \,\|G^1\|_{S}\|G^2\|_{S}\|G^3\|_{S}.
\end{aligned}
\end{equation}
\end{lemma}
\begin{proof}
In view of Proposition \ref{Proposition 4.2}, we have
\begin{equation}
\label{4.91}
\mathcal{F}_{x\rightarrow\xi}\mathcal R[G^1,G^2,G^3] = 
\Pi_+(\widehat{G}_{+}^{1} \overline{\widehat{G}_{+}^{2}} \widehat{G}_{+}^{3})+
\Pi_-(\widehat{G}_{-}^{1} \overline{\widehat{G}_{-}^{2}} \widehat{G}_{-}^{3}).
\end{equation}
In fact, we have
\begin{equation}
\label{4.101}
\begin{aligned}
\left\|\widehat{G}^{1} \overline{\widehat{G}^{2}} \widehat{G}^{3}\right\|_{L_\xi^2}  & \lesssim \min_{\{j,k,\ell\} = \{1,2,3\}}  \|\widehat{G}^j\|_{L_\xi^2} \|\widehat{G}^k\|_{L_\xi^\infty}\|\widehat{G}^\ell\|_{L_\xi^\infty} \\ &  \lesssim \min_{\{j,k,\ell\} = \{1,2,3\}}  \|G^j\|_{L_x^2} \|G^k\|_{L_x^1}\|G^\ell\|_{L_x^1}  \\ & \lesssim \min_{\{j,k,\ell\} = \{1,2,3\}}  \|G^j\|_{L_x^2} \|G^k\|_{L_x^2}^{\frac{1}{2}}\|xG^k\|_{L_x^2}^{\frac{1}{2}}\|G^\ell\|_{L_x^2}^{\frac{1}{2}}\|xG^\ell\|_{L_x^2}^{\frac{1}{2}}.
\end{aligned}
\end{equation}
Then by (\ref{4.91}) and Lemma \ref{Lemma 6.21}, we have
\begin{equation}
\label{4.111}
\left\|\mathcal{R}\left[G^{1}, G^{2}, G^{3}\right]\right\|_{S} \lesssim \|G^1\|_{S} \|G^2\|_{S} \|G^3\|_{S}. 
\end{equation}
Since we have (\ref{4.101}), by (\ref{4.91}), Remark \ref{remark 6.4} and Remark \ref{remark 6.6}, we can deduce (\ref{4.702}). Combining (\ref{4.702}) and (\ref{4.111}), we have (\ref{4.701}).
\end{proof}
We also introduce results as follows which concern the long time behavior and stability of the equation (\ref{4.1}).
\begin{lemma}
\label{Lemma 4.4}
For every $G^{1}, G^{2}, G^{3}$, the following estimates hold true
\begin{equation}
\label{4.8}
\left\|\mathcal{R}\left[G^{1}, G^{2}, G^{3}\right]\right\|_{Z} \lesssim\left\|G^{1}\right\|_{Z}\left\|G^{2}\right\|_{Z}\left\|G^{3}\right\|_{Z},
\end{equation}
\begin{equation}
\left\|\mathcal{R}\left[G^{1}, G^{2}, G^{3}\right]\right\|_{L_{x, y}^{2}} \lesssim \min _{\{j, k, \ell\}=\{1,2,3\}}\left\|G^{j}\right\|_{L_{x, y}^{2}}\left\|G^{k}\right\|_{Z}\left\|G^{\ell}\right\|_{Z},
\end{equation}
\begin{equation}
\label{4.81}
\left\|\mathcal{R}\left[G^{1}, G^{2}, G^{3}\right]\right\|_{S'} \lesssim \max _{\{j, k, \ell\}=\{1,2,3\}}\left\|G^{j}\right\|_{S'}\left\|G^{k}\right\|_{Z}\left\|G^{\ell}\right\|_{Z},
\end{equation}
\begin{equation}
\label{4.85}
\begin{aligned}
\left\|x \mathcal{R}\left[G^{1}, G^{2}, G^{3}\right]\right\|_{L_x^2 H_y^2} & \lesssim \max _{\{j, k, \ell\}=\{1,2,3\}}\|x G^{j}\|_{L_{x}^2 H_y^2}\|G^{k}\|_{Z}\|G^{\ell}\|_{Z} \\ & + \max _{\{j, k, \ell\}=\{1,2,3\}}\| G^{j}\|_{L_{x}^2 H_y^2}\|x G^{k}\|_{Z}\|G^{\ell}\|_{Z}.
\end{aligned}
\end{equation}
\begin{equation}
\label{4.10}
\begin{aligned}
\left\|\mathcal{R}\left[G^{1}, G^{2}, G^{3}\right]\right\|_{\mathcal{Y}} & \lesssim \max _{\{j, k, \ell\}=\{1,2,3\}}\left\|G^{j}\right\|_{\mathcal{Y}}\left\|G^{k}\right\|_{Z}\left\|G^{\ell}\right\|_{Z} \\ & + \max _{\{j, k, \ell\}=\{1,2,3\}}\left\|G^{j}\right\|_{L_x^2 H_y^2}\left\|x G^{k}\right\|_{Z}\left\|G^{\ell}\right\|_{Z}.
\end{aligned}
\end{equation}
\end{lemma}
\begin{proof}
(\ref{4.8}) comes from the fact that $\mathcal{G}$ is an alegbra.\\\\
Since $\Pi_{\pm}(\mathcal{G})  \subset \mathcal{G} \subset L^\infty $, we have
\begin{align*}
\left\|\mathcal{R}\left[G^{1}, G^{2}, G^{3}\right]\right\|_{L_{x, y}^{2}} \lesssim & \left\|\int_{\omega(\eta,\eta_1,\eta_2) = 0} \widehat{G}_{\eta-\eta_1}^{1} \overline{\widehat{G}}_{\eta_2-\eta_1}^{2} \widehat{G}_{\eta_2}^{3} d\eta_1 d\eta_2 \right \|_{L_{\xi,\eta}^{2} } \\ \lesssim & \left\| \widehat{G}_{+}^{1} \overline{\widehat{G}_{+}^{2}} \widehat{G}_{+}^{3} \right \|_{L_{\xi,y}^{2} } + \left\| \widehat{G}_{-}^{1} \overline{\widehat{G}_{-}^{2}} \widehat{G}_{-}^{3} \right \|_{L_{\xi,y}^{2} } & \\ \lesssim & \min _{\{j, k, \ell\}=\{1,2,3\}}\left(\|G^{j}\|_{L_{x, y}^{2}}\|\widehat{G}_{+}^{k}\|_{L_{\xi,y}^\infty}\|\widehat{G}_{+}^{\ell}\|_{L_{\xi,y}^\infty}+\|G^{j}\|_{L_{x, y}^{2}}\|\widehat{G}_{-}^{k}\|_{L_{\xi,y}^\infty}\|\widehat{G}_{-}^{\ell}\|_{L_{\xi,y}^\infty}\right) \\ \lesssim & \min _{\{j, k, \ell\}=\{1,2,3\}}\left\|G^{j}\right\|_{L_{x, y}^{2}}\|\widehat{G}^{k}\|_{L_{\xi}^{\infty} \mathcal{G}_y}\|\widehat{G}^{\ell}\|_{L_{\xi}^{\infty} \mathcal{G}_y} \\ \lesssim & \min _{\{j, k, \ell\}=\{1,2,3\}}\|G^{j}\|_{L_{x, y}^{2}}\|G^{k}\|_{Z}\|G^{\ell}\|_{Z}.
\end{align*}
By Lemma \ref{Lemma 6.2}, we deduce that 
\begin{align*}
\left\|\mathcal{R}\left[G^{1}, G^{2}, G^{3}\right]\right\|_{S'} \lesssim \max _{\{j, k, \ell\}=\{1,2,3\}}\|G^{j}\|_{S'}\|G^{k}\|_{Z}\|G^{\ell}\|_{Z}.
\end{align*}
Similarly, we can deduce that
\begin{align*}
\left\|x \mathcal{R}\left[G^{1}, G^{2}, G^{3}\right]\right\|_{L_x^2 H_y^2} & \lesssim \max _{\{j, k, \ell\}=\{1,2,3\}}\|x G^{j}\|_{L_{x}^2 H_y^2}\|G^{k}\|_{Z}\|G^{\ell}\|_{Z} \\ & + \max _{\{j, k, \ell\}=\{1,2,3\}}\| G^{j}\|_{L_{x}^2 H_y^2}\|x G^{k}\|_{Z}\|G^{\ell}\|_{Z}.
\end{align*}
Combining the above inequalities, we obtain (\ref{4.10}).
\end{proof}
\begin{proposition}
\label{Proposition 4.5}
$\text { Let } G_{0} \in \mathcal{Y}^{+},\left\|G_{0}\right\|_{\mathcal{Y}^{+}}=\varepsilon \text { with } \varepsilon \text { small enough, }$ and $G$ evolves according to (\ref{4.1}). Then there holds that for $t \geq 1$,
\begin{equation}
\label{4.12}
\|G(\pi \ln t)\|_{Z} \simeq \left\|G_{0}\right\|_{Z},
\end{equation}
\begin{equation}
\label{4.121}
\|x G(\pi \ln t)\|_{Z} \lesssim (1+|t|)^{\delta_1} \left\|x G_{0}\right\|_{Z},
\end{equation}
\begin{equation}
\label{4.122}
\|G(\pi \ln t)\|_{S'} \lesssim(1+|t|)^{\delta_2}\left\|G_{0}\right\|_{S'},
\end{equation}
\begin{equation}
\label{4.123}
\|G(\pi \ln t)\|_{\mathcal{Y}} \lesssim(1+|t|)^{\delta_3}\left\|G_{0}\right\|_{\mathcal{Y}},
\end{equation}
\begin{equation}
\label{4.13}
\|G(\pi \ln t)\|_{\mathcal{Y}^{+}} \lesssim(1+|t|)^{\delta_4}\left\|G_{0}\right\|_{\mathcal{Y}^{+}},
\end{equation}
with $\delta_1, \delta_2, \delta_3, \delta_4 \simeq\left\|G_{0}\right\|_{Z}^{2}$.
\end{proposition}
\begin{proof}
For (\ref{4.12}), we use Proposition \ref{Proposition 4.2} to transform (\ref{4.1}) to the non chiral Szeg\H o equation, and then we use the integrability of the cubic Szeg\H o equation, which indicates the conservation of the $\mathcal{G}$ norm.  Combine with the Lax Pairs, we have
\begin{align*}
\|\widehat{G}(\xi, t)\|_{\mathcal{G}_y} \simeq \|\widehat{G}_{0}(\xi)\|_{\mathcal{G}_y},
\end{align*}
which implies (\ref{4.12}). \\\\
By (\ref{4.8}), we deduce that 
\begin{equation}
\label{4.131}
\left\|x \mathcal{R}\left[G, G, G\right]\right\|_{Z} \lesssim \left\|x G\right\|_{Z}\left\|G\right\|_{Z}\left\|G\right\|_{Z}.
\end{equation}
We take $\widetilde{G}(t)=G(\pi \ln t)$, then $\widetilde{G}$ satisfies
\begin{equation}
i \partial_{t} \widetilde{G}=\frac{\pi}{t} \mathcal{R}[\widetilde{G}, \widetilde{G}, \widetilde{G}]
\end{equation}
and 
\begin{equation}
i \partial_{t} ( x \widetilde{G})=\frac{\pi}{t} x \mathcal{R}[\widetilde{G}, \widetilde{G}, \widetilde{G}].
\end{equation}
By (\ref{4.81}) and (\ref{4.131}), we have
\begin{align*}
\partial_{t}\|\widetilde{G}\|_{S'} \lesssim \frac{1}{t}\|\mathcal{R}[\widetilde{G}, \widetilde{G}, \widetilde{G}]\|_{S'} \lesssim \frac{1}{t}\|\widetilde{G}\|_{Z}^{2}\|\widetilde{G}\|_{S'} \lesssim \frac{1}{t}\left\|G_{0}\right\|_{Z}^{2}\|\widetilde{G}\|_{S'}
\end{align*}
and
\begin{align*}
\partial_{t}\|x \widetilde{G}\|_{Z} \lesssim \frac{1}{t}\|x \mathcal{R}[\widetilde{G}, \widetilde{G}, \widetilde{G}]\|_{Z} \lesssim \frac{1}{t}\|\widetilde{G}\|_{Z}^{2}\|x\widetilde{G}\|_{Z} \lesssim \frac{1}{t}\left\|G_{0}\right\|_{Z}^{2}\|x\widetilde{G}\|_{Z}.
\end{align*}
By Gronwall's inequality, we get (\ref{4.121}) and (\ref{4.122}). Then by (\ref{4.10}), we deduce that
\begin{align*}
& \partial_{t}\| \widetilde{G}\|_{\mathcal{Y}} \\ \lesssim & \frac{1}{t}\| \mathcal{R}[\widetilde{G}, \widetilde{G}, \widetilde{G}]\|_{\mathcal{Y}} \\ \lesssim & \frac{1}{t}\|\widetilde{G}\|_{Z}^{2}\|\widetilde{G}\|_{\mathcal{Y}} +  \frac{1}{t}\|\widetilde{G}\|_{Z} \|x\widetilde{G}\|_{Z} \|\widetilde{G}\|_{S'}\\ \lesssim & \frac{1}{t}\|G_0\|_{Z}^{2}\|\widetilde{G}\|_{\mathcal{Y}} +  \frac{1}{t^{1-\delta_1-\delta_2}}\left\|G_{0}\right\|_{Z} \|x G_0\|_{Z} \|G_0\|_{S'}.
\end{align*}
By inhomogeneous Gronwall's ineqaulity, we obtain (\ref{4.123}). \\\\
To estimate $\|G(\pi \ln t)\|_{\mathcal{Y}^{+}}$, firstly we estimate $\|x^2 \widetilde{G}\|_Z$ and $\|(1+|D_x|) \widetilde{G}\|_Z$. We have
\begin{equation}
\label{4.20}
\partial_{t}\|x^2 \widetilde{G}\|_{Z} \lesssim \frac{1}{t}\|x^2 \mathcal{R}[\widetilde{G}, \widetilde{G}, \widetilde{G}]\|_{Z} \lesssim \frac{1}{t}\|G_0\|_{Z}^{2}\|x^2 \widetilde{G}\|_{Z} + \frac{1}{t^{1-2\delta_1}}\left\|G_{0}\right\|_{Z}\|x G_0\|_{Z}^2,
\end{equation}
and
\begin{equation}
\label{4.205}
\partial_{t}\|(1+|D_x|)  \widetilde{G}\|_{Z} \lesssim \frac{1}{t}\|(1+|D_x|)  \mathcal{R}[\widetilde{G}, \widetilde{G}, \widetilde{G}]\|_{Z} \lesssim \frac{1}{t}\|G_0\|_{Z}^{2}\|(1+|D_x|) \widetilde{G}\|_{Z}.
\end{equation}
Similary to (\ref{4.85}), we have
\begin{equation}
\label{4.21}
\partial_{t}\|x \widetilde{G}\|_{L_x^2 H_y^3} \lesssim \frac{1}{t}\|x \mathcal{R}[\widetilde{G}, \widetilde{G}, \widetilde{G}]\|_{L_x^2 H_y^3} \lesssim \frac{1}{t}\|G_0\|_{Z}^{2}\|x \widetilde{G}\|_{L_x^2 H_y^3} + \frac{1}{t^{1-\delta_1-\delta_2}}\|G_0\|_{S'}\left\|x G_{0}\right\|_{Z}\|G_0\|_{Z}.
\end{equation}
Again by (\ref{4.85}) and (\ref{4.10}), we have
\begin{equation}
\label{4.22}
\begin{aligned}
& \partial_{t}\| x \widetilde{G}\|_{\mathcal{Y}} \\ \lesssim & \frac{1}{t}\| x \mathcal{R}[\widetilde{G}, \widetilde{G}, \widetilde{G}]\|_{\mathcal{Y}} \\ \lesssim & \frac{1}{t}\|G_0\|_{Z}^{2}\|x\widetilde{G}\|_{\mathcal{Y}} +  \frac{1}{t^{1-\delta_1-\delta_3}}\left\|G_{0}\right\|_{Z} \|x G_0\|_{Z} \|G_0\|_{\mathcal{Y}}\\ + & \frac{1}{t^{1-2\delta_1-\delta_2}} \|G_0\|_{S'} \|x G_0\|_Z^2 +\frac{1}{t^{1-\delta_2}} \|G_0\|_{S'}\|x^2\widetilde{G}\|_Z \|G_0\|_Z
\end{aligned}
\end{equation}
and 
\begin{equation}
\label{4.225}
\begin{aligned}
& \partial_{t}\| (1+|D_x|) \widetilde{G}\|_{\mathcal{Y}} \\ \lesssim & \frac{1}{t}\| (1+|D_x|) \mathcal{R}[\widetilde{G}, \widetilde{G}, \widetilde{G}]\|_{\mathcal{Y}} \\ \lesssim & \frac{1}{t}\|G_0\|_{Z}^{2}\|(1+|D_x|)\widetilde{G}\|_{\mathcal{Y}} +  \frac{1}{t^{1-\delta_3}}\left\|G_{0}\right\|_{Z} \|(1+|D_x|) \widetilde{G}\|_{Z} \|G_0\|_{\mathcal{Y}}\\ + & \frac{1}{t^{1-\delta_1-\delta_2}} \|G_0\|_{S'} \|x G_0\|_Z \|G_0\|_Z.
\end{aligned}
\end{equation}
According to (\ref{4.20}), (\ref{4.205}), (\ref{4.21}), (\ref{4.22}), (\ref{4.225}) and inhomogeneous Gronwall's inequality, we get  (\ref{4.13}).
\end{proof}
\section{Proof of the main results}
In this section, we prove our main results. We will start with constructing a modified wave operators and then obtain the corresponding cascade result, which have been introduced in Section 1.
\subsection{Modified wave operators}
\begin{theorem}
\label{Theorem 5.1}
There exists $\varepsilon >0$ such that if $G_{0} \in \mathcal{Y}^{+}$ satisfies 
\begin{equation}
\label{5.1}
\left\|G_{0}\right\|_{\mathcal{Y}^{+}} \leq \varepsilon,
\end{equation}
and if $G$ is the solution of (\ref{1.7}) with initial data $G_0$, then there exists a unique solution $U$ of (\ref{1.1}) such that $\mathrm{e}^{-i t \mathcal{A}} U(t) \in C([0, \infty): \mathcal{Y})$ and 
\begin{align*}
\left\|\mathrm{e}^{-i t \mathcal{A}} U(t)-G(\pi \ln t)\right\|_{\mathcal{Y}} \rightarrow 0 \text { as } t \rightarrow \infty.
\end{align*}
\end{theorem}
\begin{proof}
Let 
\begin{align*}
\widetilde{G}(t)=G(\pi \ln t), K(t)=\mathrm{e}^{-i t \mathcal{A}} U(t)-\widetilde{G}(t),
\end{align*}
and define a mapping 
\begin{align*}
\Phi(K)(t)=i \int_{t}^{\infty}\left\{\mathcal{N}^{\sigma}[K+\widetilde{G}, K+\widetilde{G}, K+\widetilde{G}]-\frac{\pi}{\sigma} \mathcal{R}[\widetilde{G}(\sigma), \widetilde{G}(\sigma), \widetilde{G}(\sigma)]\right\} d \sigma.
\end{align*}
Then we define
\begin{align*}
\mathfrak{A} &:=\left\{K \in C^{1}([1, \infty): \mathcal{Y}):\|K\|_{\mathfrak{A}}<\infty\right\}, \\\|K\|_{\mathfrak{A}} &:=\sup _{t>1}\left\{(1+|t|)^{\delta}\|xK(t)\|_{L_x^2 H_y^2}+(1+|t|)^{2\delta}\|K(t)\|_{S'}+(1+|t|)^{3 \delta}\|K(t)\|_{Z}+(1+|t|)^{1-\delta}\left\|\partial_{t} K(t)\right\|_{\mathcal{Y}}\right\}.
\end{align*}
We want to prove $\Phi$ defines a contraction on the complete metric space $\left\{K \in \mathfrak{A}:\|K\|_{\mathfrak{A}} \leq \varepsilon_{1}\right\}$ with $\varepsilon$ and $\varepsilon_1$ small enough. We decompose
\begin{equation}
\mathcal{N}^{t}[K+\widetilde{G}, K+\widetilde{G}, K+\widetilde{G}]-\frac{\pi}{t} \mathcal{R}[\widetilde{G}, \widetilde{G}, \widetilde{G}]=\mathcal{E}^{t}[\widetilde{G}, \widetilde{G}, \widetilde{G}]+\mathcal{L}^{t}[K, \widetilde{G}]+\mathcal{Q}^{t}[K, \widetilde{G}],
\end{equation}
where
\begin{align*}
\left\{\begin{array}{l}\mathcal{E}^{t}[\widetilde{G}, \widetilde{G}, \widetilde{G}]:=\mathcal{N}^{t}[\widetilde{G}, \widetilde{G}, \widetilde{G}]-\frac{\pi}{t} \mathcal{R}[\widetilde{G}, \widetilde{G}, \widetilde{G}], \\ \mathcal{L}^{t}[K, \widetilde{G}]:=\mathcal{N}^{t}[\widetilde{G}, \widetilde{G}, K]+\mathcal{N}^{t}[K, \widetilde{G}, \widetilde{G}]+\mathcal{N}^{t}[\widetilde{G}, K, \widetilde{G}], \\ \mathcal{Q}^{t}[K, \widetilde{G}]:=\mathcal{N}^{t}[K, K, \widetilde{G}]+\mathcal{N}^{t}[\widetilde{G}, K, K]+\mathcal{N}^{t}[K, \widetilde{G}, K]+\mathcal{N}^{t}[K, K, K].\end{array}\right.
\end{align*}
For $K \in \mathfrak{A}$, we have
\begin{equation}
(1+|t|)^{3 \delta}\|K(t)\|_{Z}+(1+|t|)^{2\delta}\|K(t)\|_{S'}+(1+|t|)^{\delta}\|xK(t)\|_{L_x^2H_y^2}+(1+|t|)^{1-\delta}\left\|\partial_{t} K(t)\right\|_{\mathcal{Y}} \lesssim \varepsilon_{1}.
\end{equation}
By Proposition \ref{Proposition 4.5}, we take $\varepsilon \lesssim \delta^{1/2}$ such that
\begin{equation}
\label{5.5}
\begin{aligned}\|\widetilde{G}(t)\|_{\mathcal{Y}^{+}}+(1+|t|)\left\|\partial_{t} \widetilde{G}(t)\right\|_{\mathcal{Y}} & \lesssim \varepsilon(1+|t|)^{\delta / 100}, \\\|\widetilde{G}(t)\|_{Z} & \lesssim \varepsilon. \end{aligned}
\end{equation}
To prove the contraction, we only need to prove the following inequalities,
\begin{equation}
\label{5.6}
\left\|\int_{t}^{\infty} \mathcal{E}^{\sigma}[\widetilde{G}, \widetilde{G}, \widetilde{G}] d \sigma\right\|_{\mathfrak{A}} \lesssim \varepsilon^{3},
\end{equation}
\begin{equation}
\label{5.7}
\left\|\int_{t}^{\infty} \mathcal{L}^{\sigma}[K, \widetilde{G}] d \sigma\right\|_{\mathfrak{A}} \lesssim \varepsilon^{2}\|K\|_{\mathfrak{A}},
\end{equation}
\begin{equation}
\label{5.8}
\left\|\int_{t}^{\infty} \mathcal{Q}^{\sigma}[K, \widetilde{G}] d \sigma\right\|_{\mathfrak{A}} \lesssim \varepsilon \varepsilon_1\|K\|_{\mathfrak{A}},
\end{equation}
\begin{equation}
\label{5.9}
\left\|\int_{t}^{\infty}\left\{\mathcal{Q}^{\sigma}\left[K_{1}, \widetilde{G}\right]-\mathcal{Q}^{\sigma}\left[K_{2}, \widetilde{G}\right]\right\} d \sigma\right\|_{\mathfrak{A}} \lesssim \varepsilon \varepsilon_{1}\left\|K_{1}-K_{2}\right\|_{\mathfrak{A}}.
\end{equation}
$\textbf{Proof of }$ (\ref{5.6}). By the definition of $\mathcal{E}^{t}$, for $t > 1$, we have
\begin{equation}
\begin{aligned}\left\|\mathcal{E}^{t}[\widetilde{G}, \widetilde{G}, \widetilde{G}]\right\|_{\mathcal{Y}} &=\left\|\mathcal{N}^{t}[\widetilde{G}, \widetilde{G}, \widetilde{G}]-\frac{1}{t} \mathcal{R}[\widetilde{G}, \widetilde{G}, \widetilde{G}]\right\|_{\mathcal{Y}} \\ & \leq\left\|\mathcal{N}^{t}[\widetilde{G}, \widetilde{G}, \widetilde{G}]\right\|_{\mathcal{Y}}+\frac{1}{t}\|\mathcal{R}[\widetilde{G}, \widetilde{G}, \widetilde{G}]\|_{\mathcal{Y}}. \end{aligned}
\end{equation}
By (\ref{2.16}), we have
\begin{align*}
\left\|\mathcal{N}^{t}[\widetilde{G}, \widetilde{G}, \widetilde{G}]\right\|_{\mathcal{Y}} \lesssim t^{-1}\|\widetilde{G}\|_{S}^{3} \lesssim t^{-1+\delta} \varepsilon^{3},
\end{align*}
and by (\ref{4.701}),
\begin{align*}
\left\|\mathcal{R}^{t}[\widetilde{G}, \widetilde{G}, \widetilde{G}]\right\|_{\mathcal{Y}} \lesssim \|\widetilde{G}\|_{S}^3  \lesssim t^{\delta} \varepsilon^{3},
\end{align*}
then
\begin{align*}
\left\|\mathcal{E}^{t}[\widetilde{G}, \widetilde{G}, \widetilde{G}]\right\|_{\mathcal{Y}} \lesssim t^{-1+\delta} \varepsilon^{3}.
\end{align*}
Thus we have
\begin{align*}
t^{1-\delta}\left\|\partial_{t}\left(\int_{t}^{\infty} \mathcal{E}^{\sigma}[\widetilde{G}, \widetilde{G}, \widetilde{G}] d \sigma\right)\right\|_{\mathcal{Y}} \lesssim \varepsilon^{3}.
\end{align*}
For the other three terms of the $\mathfrak{A}$ norm, by (\ref{5.5}), we have $\|\widetilde{G}\|_{X_{T}^{+}} \lesssim \varepsilon$ for any $T \geq 1$, so $\left\|\int_{t}^{\infty} \mathcal{E}^{\sigma}[\widetilde{G}, \widetilde{G}, \widetilde{G}] d \sigma\right\|_{S'}$, $\left\|\int_{t}^{\infty} \mathcal{E}^{\sigma}[\widetilde{G}, \widetilde{G}, \widetilde{G}] d \sigma\right\|_{L_x^2 H_y^2}$  and $\left\|\int_{t}^{\infty} \mathcal{E}^{\sigma}[\widetilde{G}, \widetilde{G}, \widetilde{G}] d \sigma\right\|_{Z}$ can be controlled by the estimates in Proposition \ref{Proposition 3.1}. \\\\
$\textbf{Proof of }$ (\ref{5.7}). We estimate the norm $\left\|\int_{t}^{\infty} \mathcal{N}^{\sigma}[K,\widetilde{G},\widetilde{G}] d \sigma\right\|_{\mathfrak{A}}$ for example, and other terms in $\mathcal{L}^{t}[K, \widetilde{G}]$ can be estimated in the same way.

By (\ref{2.16}) and (\ref{5.5}), we control the time derivative in the $\mathfrak{A}$ norm, 
\begin{align*}
\left\|\mathcal{N}^{t}[K,\widetilde{G}, \widetilde{G} ]\right\|_{\mathcal{Y}} \lesssim t^{-1}\|\widetilde{G}\|_{S}^{2}\|K\|_{S} \lesssim t^{-1+\delta} \varepsilon^{2}\|K\|_{\mathfrak{A}}.
\end{align*}
For the other three terms in the $\mathfrak{A}$ norm, we will do the decomposition as we did in the proof of Proposition \ref{Proposition 3.1} on $\mathcal{N}^{t}[G, G, K]$. By Lemma \ref{Lemma 3.3} and Lemma \ref{Lemma 3.5}, we only need to show
\begin{equation}
\label{5.10}
\|\mathcal{R}[K, \widetilde{G}, \widetilde{G}]\|_{Z} \lesssim (1+|t|)^{-3\delta} \varepsilon^{2} \|K\|_{\mathfrak{A}},
\end{equation}
\begin{equation}
\label{5.11}
\left\|\mathcal{N}_{0}^{t}[K, \widetilde{G},\widetilde{G}]-\frac{\pi}{t} \mathcal{R}[K, \widetilde{G}, \widetilde{G}]\right\|_{Z} \lesssim  (1+|t|)^{-1-3\delta} \varepsilon^{2} \|K\|_{\mathfrak{A}},
\end{equation}
\begin{equation}
\label{5.12}
\left\|\mathcal{N}_{0}^{t}[G, G, K]\right\|_{S'} \lesssim (1+|t|)^{-1-2\delta} \varepsilon^{2} \|K\|_{\mathfrak{A}}.
\end{equation}
\begin{equation}
\label{5.13}
\left\|x\mathcal{N}_{0}^{t}[K,\widetilde{G},\widetilde{G}]\right\|_{L_x^2 H_y^2 } \lesssim (1+|t|)^{-1-\delta} \varepsilon^{2} \|K\|_{\mathfrak{A}}.
\end{equation}
(\ref{5.10}) comes from (\ref{4.8}),
\begin{align*}
\|\mathcal{R}[K, \widetilde{G}, \widetilde{G}]\|_{Z} \lesssim\|\widetilde{G}\|_{Z}^{2}\|K\|_{Z} \lesssim (1+|t|)^{-3\delta} \varepsilon^2 \|K\|_{\mathfrak{A}}.
\end{align*}
(\ref{5.11}) comes from from (\ref{3.41}),
\begin{align*}
\left\|\mathcal{N}_{0}^{t}[K, \widetilde{G}, \widetilde{G}]-\frac{\pi}{t} \mathcal{R}[K, \widetilde{G}, \widetilde{G}]\right\|_{Z} \lesssim(1+|t|)^{-1-23 \delta}\|\widetilde{G}\|_{S}^{2}\|K\|_{S} \lesssim (1+|t|)^{-1-3\delta} \varepsilon^{2} \|K\|_{\mathfrak{A}}.
\end{align*}
(\ref{5.12}) comes from (\ref{3.38}),
\begin{align*}
\left\|\mathcal{N}_{0}^{t}[K,\widetilde{G}, \widetilde{G}]\right\|_{S'} & \lesssim (1+|t|)^{-1} \left(\|\widetilde{G}\|_{\widetilde{Z}_{t}}^{2}\|K\|_{S'} + \|\widetilde{G}\|_{\widetilde{Z}_{t}}\|K\|_{\widetilde{Z}_{t}}\|\widetilde{G}\|_{S'} \right)\\ & \lesssim (1+|t|)^{-1-2\delta} \varepsilon^{2} \|K\|_{\mathfrak{A}}.
\end{align*}
(\ref{5.13}) comes from (\ref{3.385}),
\begin{align*}
\left\|x\mathcal{N}_{0}^{t}\left[K, \widetilde{G}, \widetilde{G}\right]\right\|_{L_x^2 H_y^2}  & \lesssim (1+|t|)^{-1} \left\| x K \right\|_{L_{x}^2 H_y^2} \|\widetilde{G} \|_{\widetilde{Z}_{t}} \| \widetilde{G}\|_{\widetilde{Z}_{t}} \\ & + (1+|t|)^{-1} \|K\|_{S'} \|\widetilde{G}\|_{S^+} \|\widetilde{G}\|_{\widetilde{Z}_{t}}\\ & + (1+|t|)^{-1} \|K\|_{\widetilde{Z}_{t}} \|\widetilde{G}\|_{S} \|\widetilde{G}\|_{S} \\ & \lesssim (1+|t|)^{-1-\delta} \varepsilon^{2} \|K\|_{\mathfrak{A}}.
\end{align*}
$\textbf{Proof of }$ (\ref{5.8}). The proof of (\ref{5.8}) is similar to the proof of (\ref{5.7}), but we notice that we use (\ref{2.16}) directly to estimate $\left\|x\mathcal{N}_{0}^{t}\left[K, K, \widetilde{G}\right]\right\|_{L_x^2 H_y^2}$, 
\begin{align*}
\left\|x\mathcal{N}_{0}^{t}\left[K, K, \widetilde{G}\right]\right\|_{L_x^2 H_y^2}  & \lesssim (1+|t|)^{-1} \left\| K \right\|_{S} \left\|K \right\|_{S} \| \widetilde{G}\|_{S} \\ & \lesssim (1+|t|)^{-1-\delta} \varepsilon \varepsilon_1 \|K\|_{\mathfrak{A}}.
\end{align*}
$\textbf{Proof of }$ (\ref{5.9}). We have
\begin{align*}
\mathcal{N}^{t}\left[K_{1}, K_{1}, G\right]-\mathcal{N}^{t}\left[K_{2}, K_{2}, G\right]=& \mathcal{N}^{t}\left[K_{1}, K_{1}, G\right]-\mathcal{N}^{t}\left[K_{1}, K_{2}, G\right]+\mathcal{N}^{t}\left[K_{1}, K_{2}, G\right] \\ &-\mathcal{N}^{t}\left[K_{2}, K_{2}, G\right] \\=& \mathcal{N}^{t}\left[K_{1}, K_{1}-K_{2}, G\right]+\mathcal{N}^{t}\left[K_{1}-K_{2}, K_{2}, G\right],
\end{align*}
we take similar decompositions on the terms $\mathcal{N}^{t}\left[K_{1}, G, K_{1}\right]-\mathcal{N}^{t}\left[K_{2}, G, K_{2}\right]$, $\mathcal{N}^{t}\left[G, K_{1}, K_{1}\right]-\mathcal{N}^{t}\left[G, K_{2}, K_{2}\right]$ and $\mathcal{N}^{t}\left[K_{1}, K_{1}, K_{1}\right]-\mathcal{N}^{t}\left[K_{2}, K_{2}, K_{2}\right]$. Similar strategy we used to prove (\ref{5.8}) can be applied to obtain the estimate (\ref{5.9}). The proof of is complete.
\end{proof}
\subsection{Cascade result}
Before we go to prove Theorem \ref{Theorem 1.6}, we recall Proposition \ref{Proposition 1.4}.
\begin{proposition}
\label{Proposition 5.2}
Let $u \in H_{+}^1(\mathbb{R})$ be a rational solution to the cubic Szeg\H o equation on the line (\ref{1.51})  such that $H_u$ has singular values $\lambda_{1}, \cdots, \lambda_{N}$, with $\lambda_1$ being multiple and $\lambda_j$ being simple for every $j \geq 2$. Then
\begin{align*}
0<\liminf_{t \rightarrow\infty} \frac{\left\|\partial_{x} u(t)\right\|_{L^{2}}}{|t|} \leq \limsup_{t \rightarrow\infty} \frac{\left\|\partial_{x} u(t)\right\|_{L^{2}}}{|t|} < +\infty,
\end{align*}
where $H_u$ is an Hankel operator, and we say that $\lambda > 0$ is a singular value of $H_u$ if the corresponding Schmidt subspace
\begin{align*}
E_{H_{u}}(\lambda)=\operatorname{Ker}\left(H_{u}^{2}-\lambda^{2} I\right)
\end{align*}
is not $\{ 0 \}$. \\\\
Moreover, the set
\begin{align*}
\begin{array}{l}\mathcal{J}=\{u_{0} \in H_{+}^{1}\left(\mathbb{R}\right): \text {the solution to (\ref{1.51}) with initial data $u_0$ satisfying } \\  \liminf_{t\to\infty}\frac{\|u(t)\|_{H^{1}}}{|t|} >0  \}\end{array}
\end{align*}
is a dense subset of $H_{+}^{1}\left(\mathbb{R}\right)$.
\end{proposition}
Let $u(t,\cdot) \in H_{+}^1(\mathbb{R})$ be a rational solution to the cubic Szeg\H o equation on the line (\ref{1.51}) such that $H_{u}$ has singular values $\lambda_{1}, \cdots, \lambda_{N}$, with $\lambda_1$ being multiple and $\lambda_j$ being simple for every $j \geq 2$ and let $u_0(y) := u(0,y)$. According to Proposition \ref{Proposition 5.2}, we have
\begin{align*}
0<\liminf_{t \rightarrow\infty} \frac{\left\|\partial_{y} u(t)\right\|_{L^{2}}}{|t|} \leq \limsup_{t \rightarrow\infty} \frac{\left\|\partial_{y} u(t)\right\|_{L^{2}}}{|t|} < +\infty.
\end{align*}
We take a cutoff function $\psi \in C_c^\infty(\mathbb{R})$ as follows:  $\psi(\xi) = 1 $ when $|\xi| \leq 1$, $0<\psi(\xi)<1 $ when $1< |\xi| < 2 $ and $\psi(\xi)=0$ when $|\xi| \geq 2$. Let $\widehat{G}_0(\xi,y) : = \rho\,\psi(\xi)u_0(y)$, where $\rho>0$ is a constant to be determined. We are to verify that $G_0(x,y) \in \mathcal{Y}^{+}$.

We notice that the Fourier transform of a rational function in $H_{+}^1(\mathbb{R})$ is a linear combination of
\begin{align*}
\mathbbm{1}_{\eta \geq 0} \, \eta^{k} \mathrm{e}^{-\alpha \eta},
\end{align*}
where $k$ is a nonnegative integer and $\alpha$ is a complex number of positive real part. So we can easily verify that $u_0 \in H^p(\mathbb{R}),\, \forall p \geq 0$. Also, as $u_0 \in H_{+}^1(\mathbb{R}) \subset L^\infty(\mathbb{R})$ is a rational function, the corresponding Hankel operator $H_{u_0}$ has finite rank, so $\sqrt{H_{u_0}^2}$ has finite trace norm. By Peller's theorem in [\ref{[5]}], we have the equivalence between the trace norm of $\sqrt{H_{u_0}^2}$ and $\|u_0\|_{\dot{B}_{1,1}^1(\mathbb{R})}$, so we have $u_0 \in \mathcal{G}$. Moreover, since $\psi \in C_{c}^{\infty}(\mathbb{R})$, we have $(1+|\xi|)^n\psi(\xi)$ and $ \partial_\xi^m\psi(\xi)$ are all in $L^2(\mathbb{R}) \cap L^\infty(\mathbb{R})$ for any $n,m\geq0$. In summary, we verify that $G_0(x,y) \in \mathcal{Y}^{+}$.

For any $\varepsilon > 0$, by a good choice of $\rho_{\varepsilon}$, we can make $G_0$ saitisfy $\|G_0\|_{\mathcal{Y}^{+}} \leq \varepsilon$, which verifies the condition of Theorem \ref{Theorem 5.1}.  Then we observe that $G(t,x,y) = \mathcal{F}_{\xi\to x}^{-1}\left(\rho_{\varepsilon}\psi(\xi) u(\rho_{\varepsilon}^2\psi(\xi)^2 t,\xi,y)\right)$ satisfies the following equation,
\begin{equation}
\left\{\begin{array}{l}i \partial_{t} G=\mathcal{R}[G, G, G],\quad (x,y) \in \mathbb{R}\times\mathbb{R}, \\ G(0)=G_{0}.\end{array}\right.
\end{equation}
Then we have
\begin{equation}
\label{5.14}
\begin{aligned}
\liminf_{t\to\infty} \frac{\|G(t)\|_{L_{x}^{2} H_{y}^{1}}}{|t|} & \gtrsim \liminf_{t\to\infty} \frac{\left(\int_{|\xi|\leq 1}\|\widehat{G}(t, \xi,y)\|_{H_{y}^{1}}^{2} d \xi\right)^{1 / 2}}{|t|} \\ & = \sqrt{2} \rho_{\varepsilon}^3\liminf_{t\to\infty}\frac{\|u(\rho_{\varepsilon}^2t,y)\|_{H_{y}^{1}}}{\rho_{\varepsilon}^2|t|} >0
\end{aligned}
\end{equation}
and 
\begin{equation}
\label{5.145}
\begin{aligned}
\limsup_{t\to\infty} \frac{\|G(t)\|_{L_{x}^{2} H_{y}^{1}}}{|t|} & \lesssim \limsup_{t\to\infty} \frac{\left(\int_{|\xi| < 2}\|\widehat{G}(t, \xi,y)\|_{H_{y}^{1}}^{2} d \xi\right)^{1 / 2}}{|t|} \\ & \lesssim \rho_{\varepsilon}^3 \left( \int_{|\xi| < 2} \limsup_{t\to\infty} \frac{\|u(\rho_{\varepsilon}^2\psi(\xi)^2t,y)\|_{H_y^1}^2}{\rho_{\varepsilon}^4\psi(\xi)^4 |t|^2}\right)^{\frac{1}{2}}<+\infty.
\end{aligned}
\end{equation}
From (\ref{5.14}), (\ref{5.145}) and Theorem \ref{Theorem 5.1}, we deduce the corresponding cascade result for the solutions to (\ref{1.1}), which is Theorem \ref{Theorem 1.6}, and we rewrite it as follows.
\begin{theorem}
Given $N \geq 3$, then for any $\varepsilon > 0$, there exists $U_{0} \in \mathcal{Y}^{+}$ with $\left\|U_{0}\right\|_{\mathcal{Y}^{+}} \leq \varepsilon$, such that the corresponding solution to (\ref{1.1}) satisfies
\begin{equation}
0<\liminf _{t \rightarrow \infty} \frac{\|U(t)\|_{L_{x}^{2} H_{y}^{1}}}{(1+\log |t|)} \leq  \limsup _{t \rightarrow \infty} \frac{\|U(t)\|_{L_{x}^{2} H_{y}^{1}}}{(1+\log |t|)} < +\infty.
\end{equation}
\end{theorem}
\begin{remark}
As in Proposition \ref{Proposition 5.2}, we expect that there exists some Banach space $B$ such that the set
\begin{align*}
J:=\left\{U_{0} \in B: \text{ the solution to (\ref{1.1}) with initial data } U_0 \text{ satisfying } \liminf_{t \rightarrow \infty} \frac{\|U(t)\|_{L_x^2 H_y^1}}{1+\log |t|}>0\right\}
\end{align*}
is a dense subset of $B$. The difficulty comes from the gap between $\mathcal{Y}$ and $\mathcal{Y}^{+}$ in the modified wave operator argument, which also exists in the result of H. Xu [\ref{[2]}].
\end{remark}
\section{Appendix}
We now turn to our basic lemma allowing to transform suitable $L_{x,y}^2$ bounds to bounds in terms of the $L_{x,y}^2$-based spaces $S, \mathcal{Y}$ and $S^{+}, \mathcal{Y}^{+}$. We define an LP-family $\widetilde{Q}=\left\{\widetilde{Q}_{A}\right\}_{A}$ to be a family of operators (indexed by the dyadic integers) of the form
\begin{align*}
\widehat{\widetilde{Q}_1 f}(\xi)=\widetilde{\varphi}(\xi) \widehat{f}(\xi), \quad \widehat{\widetilde{Q}_{A} f}(\xi)=\widetilde{\phi}\left(\frac{\xi}{A}\right) \widehat{f}(\xi), \quad A \geq 2
\end{align*}
for two smooth functions $\widetilde{\varphi}, \widetilde{\phi} \in C_{c}^{\infty}(\mathbb{R}) \text { with } \widetilde{\phi} \equiv 0 \text { in a neighborhood of } 0$. \\\\
We define the set of admissible transformations to be the family of operators $\left\{T_{A}\right\}$ where for any dyadic number $A$,
\begin{center}
$T_{A}=\lambda_{A} \widetilde{Q}_{A}, \quad\left|\lambda_{A}\right| \leq 1$
\end{center}
for some LP-family $\widetilde{Q}$.

Let $\mathcal{B}$ be a $L_{x,y}^2$-based space. If $F \in \mathcal{B}$, then for any admissible transformation family $T=\left\{T_{A}: A \text { dyadic numbers }\right\}$, $\sum_{A} T_{A} F$ converges in $\mathcal{B}$. And this norm $\mathcal{B}$ is called admissible if 
\begin{equation}
\left\|\sum_{A} T_{A} F\right\|_{\mathcal{B}} \lesssim\|F\|_{\mathcal{B}}.
\end{equation}
We observe that $S$ and $Z$ are admissible.\\\\
Given a trilinear operator $\mathfrak{T}$ and a set $\Lambda$ of 4-tuples of dyadic integers, we define an admissible realization of $\mathfrak{T}$ at $\Lambda$ to be an operator of the form which converges in $L^2$,
\begin{equation}
\label{6.2}
\mathfrak{T}_{\Lambda}[F, G, H]=\sum_{(A, B, C, D) \in \Lambda} T_{D} \mathfrak{T}\left[T_{A}^{\prime} F, T_{B}^{\prime \prime} G, T_{C}^{\prime \prime \prime} H\right]
\end{equation}
for some admissible transformations $T, T^{\prime}, T^{\prime \prime}, T^{\prime \prime \prime}$. Then we introduce the following transfer lemma, which is also introduced in [\ref{[2]}].
\begin{lemma}
[\ref{[2]}, Lemma 5.2]
\label{Lemma 6.2}
Consider a trilinear operator $\mathfrak{T}$ which satisfies
\begin{equation}
\label{6.3}
Z \mathfrak{T}[F, G, H]=\mathfrak{T}[Z F, G, H]+\mathfrak{T}[F, Z G, H]+\mathfrak{T}[F, G, Z H],
\end{equation}
for $Z \in\left\{x, \partial_{x}, \partial_{y}\right\}$ and let $\Lambda$ be a set of 4-tuples of dyadic integers. Assume that for all admissible realizations of $\mathfrak{T}$ at $\Lambda$, 
\begin{equation}
\label{6.4}
\left\|\mathfrak{T}_{\Lambda}\left[F^{a}, F^{b}, F^{c}\right]\right\|_{L_{x,y}^{2}} \lesssim K \min _{\{\alpha, \beta, \gamma\}=\{a, b, c\}}\left\|F^{\alpha}\right\|_{L_{x,y}^{2}}\left\|F^{\beta}\right\|_{\mathcal{B}}\left\|F^{\gamma}\right\|_{\mathcal{B}}
\end{equation}
for some admissible norm $\mathcal{B}$ such that the Littlewood–Paley projectors $P_{\leq M}$(both in $x$ and in $y$) are uniformly bounded on $\mathcal{B}$. Then, for all admissible realizations of $\mathfrak{T}$ at $\Lambda$, we have
\begin{equation}
\label{6.5}
\left\|\mathfrak{T}_{\Lambda}\left[F^{a}, F^{b}, F^{c}\right]\right\|_{S'} \lesssim K \max _{\{\alpha, \beta, \gamma\}=\{a, b, c\}}\left\|F^{\alpha}\right\|_{S'}\left\|F^{\beta}\right\|_{\mathcal{B}}\left\|F^{\gamma}\right\|_{\mathcal{B}}.
\end{equation}
Moreover, for all admissible realizations of $\mathfrak{T}$ at $\Lambda$, we have
\begin{equation}
\label{6.51}
\left\|\mathcal{F}_{x \to \xi }\mathfrak{T}_{\Lambda}\left[F^{a}, F^{b}, F^{c}\right]\right\|_{L_{\xi}^\infty L_y^2} \lesssim K \max _{\{\alpha, \beta, \gamma\}=\{a, b, c\}}\left\|F^{\alpha}\right\|_{S^{\prime}}\left\|F^{\beta}\right\|_{\mathcal{B}}\left\|F^{\gamma}\right\|_{\mathcal{B}}.
\end{equation}
\end{lemma}
\begin{proof}
We recall that $S'$ consists of two norms: the $x$ weighed $L_{x,y}^2$ norm and the $H_{x,y}^N$ norm.\\\\
1. \textbf{The $x$ weighted $L_{x,y}^2$ norm. }By $x T_{A}=\left[x, T_{A}\right]+T_{A} x$ and using (\ref{6.3}), we have
\begin{align*} x \mathfrak{T}_{\Lambda}\left[F^{a}, F^{b}, F^{c}\right] &=\sum_{(A, B, C, D) \in \Lambda} x T_{D} \mathfrak{T}\left[T_{A}^{\prime} F^{a}, T_{B}^{\prime \prime} F^{b}, T_{C}^{\prime \prime \prime} F^{c}\right] \\ &=\sum_{(A, B, C, D) \in \Lambda}\left[x, T_{D}\right] \mathfrak{T}\left[T_{A}^{\prime} F^{a}, T_{B}^{\prime \prime} F^{b}, T_{C}^{\prime \prime \prime} F^{c}\right] \\ & +\sum_{(A, B, C, D) \in \Lambda} T_{D} \mathfrak{T}\left[\left[x, T_{A}^{\prime}\right] F^{a}, T_{B}^{\prime \prime} F^{b}, T_{C}^{\prime \prime \prime} F^{c}\right] \\ & +\sum_{(A, B, C, D) \in \Lambda} T_{D} \mathfrak{T}\left[T_{A}^{\prime} F^{a},\left[x, T_{B}^{\prime \prime}\right] F^{b}, T_{C}^{\prime \prime \prime} F^{c}\right]\\ &+\sum_{(A, B, C, D) \in \Lambda} T_{D} \mathfrak{T}\left[T_{A}^{\prime} F^{a}, T_{B}^{\prime \prime} F^{b},\left[x, T_{C}^{\prime \prime \prime}\right] F^{c}\right] \\ & +\mathfrak{T}_{\Lambda}\left[x F^{a}, F^{b}, F^{c}\right]+\mathfrak{T}_{\Lambda}\left[F^{a}, x F^{b}, F^{c}\right]+\mathfrak{T}_{\Lambda}\left[F^{a}, F^{b}, x F^{c}\right].\end{align*}
We also have
\begin{align*}
\left[x, Q_A\right]=A^{-1} Q_{A}^{\prime}.
\end{align*}
We observe that if $Q_{A}$ is an LP-family, $Q_{A}^{\prime}$ is also an LP-family, then $\left[x, T_{A}\right]$ is also anadmissible transformation. Thus, we consider $x\mathfrak{T}_{\Lambda}\left[F^{a}, F^{b}, F^{c}\right]$ as the following summation,
\begin{equation}
\label{6.13}
\mathfrak{T}_{\Lambda}\left[F^{a}, F^{b}, F^{c}\right]+\mathfrak{T}_{\Lambda}\left[x F^{a}, F^{b}, F^{c}\right]+\mathfrak{T}_{\Lambda}\left[F^{a}, x F^{b}, F^{c}\right]+\mathfrak{T}_{\Lambda}\left[F^{a}, F^{b}, x F^{c}\right],
\end{equation}
then $\left\|x \mathfrak{T}_{\Lambda}\left[F^{a}, F^{b}, F^{c}\right]\right\|_{L_{x,y}^{2}}$ follows from (\ref{6.4}). \\\\
2.\textbf{ The $H_{x,y}^N$ norm. }We have
\begin{align*}
\|F\|_{H_{x,y}^{N}}^{2}:=\sum_{M \text { dyadic }} M^{2 N}\left\|P_{M} F\right\|_{L_{x,y}^{2}}^{2},
\end{align*}
with $P_{M}$ as the Littlewood–Paley projections on $\mathbb{R} \times \mathbb{R}$. Then we decompose
\begin{align*}
P_{M} \mathfrak{T}_{\Lambda}\left[F^{a}, F^{b}, F^{c}\right]=P_{M} \mathfrak{T}_{\Lambda, l o w}\left[F^{a}, F^{b}, F^{c}\right]+P_{M} \mathfrak{T}_{\Lambda, h i g h}\left[F^{a}, F^{b}, F^{c}\right],
\end{align*}
with $\mathfrak{T}_{\Lambda, l o w}\left[F^{a}, F^{b}, F^{c}\right]=\mathfrak{T}_{\Lambda}\left[P_{\leq M} F^{a}, P_{\leq M} F^{b}, P_{\leq M} F^{c}\right]$.\\\\
Firstly we have
\begin{equation}
\sum_{M \text { dyadic }} M^{2 N}\left\|P_{M} \mathfrak{T}_{\Lambda, h i g h}\left[F^{a}, F^{b}, F^{c}\right]\right\|_{L_{x,y}^{2}}^{2} \lesssim K^{2} \max _{\{\alpha, \beta, \gamma\}=\{a, b, c\}}\left\|F^{\alpha}\right\|_{H_{x,y}^{N}}^{2}\left\|F^{\beta}\right\|_{\mathcal{B}}^{2}\left\|F^{\gamma}\right\|_{\mathcal{B}}^{2},
\end{equation}
since
\begin{equation}
\begin{array}{l}\sum_{M}|M|^{2 N}\left\|P_{M} \mathfrak{T}_{\Lambda}\left[P_{\geq 2 M} F^{a}, F^{b}, F^{c}\right]\right\|_{L_{x,y}^{2}}^{2} \lesssim K^{2} \sum_{M}|M|^{2 N}\left\|P_{\geq 2 M} F^{a}\right\|_{L_{x,y}^{2}}^{2}\left\|F^{b}\right\|_{\mathcal{B}}^{2}\left\|F^{c}\right\|_{\mathcal{B}}^{2} \\ \quad \lesssim K^{2}\left\|F^{a}\right\|_{H_{x,y}^{N}}^{2}\left\|F^{b}\right\|_{\mathcal{B}}^{2}\left\|F^{c}\right\|_{\mathcal{B}}^{2}.\end{array}
\end{equation}
Let $Z \in\left\{\partial_{x}, \partial_{y}\right\}$, we bound the contribution of $\mathfrak{T}_{\Lambda, low}$ as follows,
\begin{equation}
\begin{aligned}
& M^{N}\left\|P_{M} \mathfrak{T}_{\Lambda, low }\right\|_{L_{x,y}^{2}} \\ \lesssim & M^{-N}\left\|Z^{2 N} P_{M} \mathfrak{T}_{\Lambda, l o w}\left[P_{\leq M} F^{a}, P_{\leq M} F^{b}, P_{\leq M} F^{c}\right]\right\|_{L_{x,y}^{2}} \\ = & M^{-N} \left\|\sum_{\alpha+\beta+\gamma \leq 2 N} \sum_{M_{1}, M_{2}, M_{3} \leq M} P_{M} \mathfrak{T}_{\Lambda, l o w}\left[Z^{\alpha} P_{M_{1}} F^{a}, Z^{\beta} P_{M_{2}} F^{b}, Z^{\gamma} P_{M_{3}} F^{c}\right] \right\|_{L_{x,y}^{2}}.
\end{aligned}
\end{equation}
Without loss of generality, we assume $M_{1}=\max \left(M_{1}, M_{2}, M_{3}\right) \leq M$, then we have
\begin{equation}
\begin{aligned} M^{N}\left\|P_{M} \mathfrak{T}_{\Lambda, low}\right\|_{L_{x,y}^{2}} & \lesssim \sum_{M_{1} \leq M} M^{-N} M_{1}^{2 N} \sum_{M_{2}, M_{3} \leq M_{1}}\left\|\mathfrak{T}_{\Lambda, low}\left[P_{M_{1}} F^{a}, P_{M_{2}} F^{b}, P_{M_{3}} F^{c}\right]\right\|_{L_{x,y}^{2}} \\ & \lesssim K \sum_{M_{1} \leq M}\left(\frac{M_{1}}{M}\right)^{-N} M_{1}^{N}\left\|P_{M_{1}} F^{a}\right\|_{L_{x,y}^{2}}\left\|F^{b}\right\|_{\mathcal{B}}\left\|F^{c}\right\|_{\mathcal{B}}. \end{aligned}
\end{equation}
By Schur test, the above sum is in $\ell_{M}^{2}$, then we have
\begin{equation}
\sum_{M \text { dyadic }} M^{2 N}\left\|P_{M} \mathfrak{T}_{\Lambda, low }\left[F^{a}, F^{b}, F^{c}\right]\right\|_{L_{x,y}^{2}}^{2} \lesssim K^{2} \max _{\{\alpha, \beta, \gamma\}=\{a, b, c\}}\left\|F^{\alpha}\right\|_{H_{x,y}^{N}}^{2}\left\|F^{\beta}\right\|_{\mathcal{B}}^{2}\left\|F^{\gamma}\right\|_{\mathcal{B}}^{2}.
\end{equation}
Thus we bound the $H_{x,y}^N$ norm, which indicates the estimate (\ref{6.5}). \\\\
Moreover, since we have (\ref{6.4}) and (\ref{6.5}), we can deduce that
\begin{equation}
\begin{aligned}
 \left\|\mathcal{F}_{x \to \xi }\mathfrak{T}_{\Lambda}\left[F^{a}, F^{b}, F^{c}\right]\right\|_{L_{\xi}^\infty L_y^2} \lesssim & \left\|\mathcal{F}_{x \to \xi }\mathfrak{T}_{\Lambda}\left[F^{a}, F^{b}, F^{c}\right]\right\|_{L_y^2 L_{\xi}^\infty }  \\  \lesssim & \left\|\mathfrak{T}_{\Lambda}\left[F^{a}, F^{b}, F^{c}\right]\right\|_{L_y^2 L_{x}^1 }  \\ \lesssim & \left\|\mathfrak{T}_{\Lambda}\left[F^{a}, F^{b}, F^{c}\right]\right\|_{L_{x,y}^2 }^{\frac{1}{2}}\left\|x\, \mathfrak{T}_{\Lambda}\left[F^{a}, F^{b}, F^{c}\right]\right\|_{L_{x,y}^2 }^{\frac{1}{2}} \\ \lesssim & K \max _{\{\alpha, \beta, \gamma\}=\{a, b, c\}}\left\|F^{\alpha}\right\|_{S^{\prime}}\left\|F^{\beta}\right\|_{\mathcal{B}}\left\|F^{\gamma}\right\|_{\mathcal{B}},
\end{aligned}
\end{equation}
\end{proof}
According to Lemma \ref{Lemma 6.2}, if we can give an estimate such as (\ref{6.4}) for a trilinear operator $\mathfrak{T}$ which satisfies (\ref{6.2}), then we can have the estimates such as (\ref{6.5}) and (\ref{6.51}) for this trilinear operator $\mathfrak{T}$. The following lemma is to give one of the condition for the derivation of (\ref{6.4}), and from this we have even more precise estimates.
\begin{lemma}
\label{Lemma 6.21}
Consider a trilinear operator $\mathfrak{T}$ which satisfies
\begin{equation}
\label{6.7}
Z \mathfrak{T}[F, G, H]=\mathfrak{T}[Z F, G, H]+\mathfrak{T}[F, Z G, H]+\mathfrak{T}[F, G, Z H],
\end{equation}
for $Z \in\left\{x, \partial_{x}, \partial_{y}\right\}$ and let $\Lambda$ be a set of 4-tuples of dyadic integers. Assume that for all admissible realizations of $\mathfrak{T}$ at $\Lambda$, we have
\begin{equation}
\label{6.8}
\left\|\mathfrak{T}_{\Lambda}\left[F^{a}, F^{b}, F^{c}\right]\right\|_{L_{x}^{2}} \lesssim K \min _{\{\alpha, \beta, \gamma\}=\{a, b, c\}}\left\|F^{\alpha}\right\|_{L_{x}^{2}}\left\|F^{\beta}\right\|_{L_x^2}^{\frac{1}{2}}\left\|x F^{\beta}\right\|_{L_x^2}^{\frac{1}{2}} \left\|F^{\gamma}\right\|_{L_x^2}^{\frac{1}{2}}\left\|x F^{\gamma}\right\|_{L_x^2}^{\frac{1}{2}},
\end{equation}
then we have
\begin{equation}
\label{6.9}
\begin{aligned}
\left\|\mathfrak{T}_{\Lambda}\left[F^{a}, F^{b}, F^{c}\right]\right\|_{L_{x,y}^{2}} & \lesssim  K \min _{\{\alpha, \beta, \gamma\}=\{a, b, c\}}\left\|F^{\alpha}\right\|_{L_{x,y}^{2}}\left\|F^{\beta}\right\|_{L_x^2 H_y^1}^{\frac{1}{2}}\left\|x F^{\beta}\right\|_{L_x^2 H_y^1}^{\frac{1}{2}} \left\|F^{\gamma}\right\|_{L_x^2 H_y^1}^{\frac{1}{2}}\left\|x F^{\gamma}\right\|_{L_x^2 H_y^1}^{\frac{1}{2}} \\ & \lesssim K \min _{\{\alpha, \beta, \gamma\}=\{a, b, c\}}\left\|F^{\alpha}\right\|_{L_{x,y}^{2}}\left\|F^{\beta}\right\|_{S}\left\|F^{\gamma}\right\|_{S},
\end{aligned}
\end{equation}
and
\begin{equation}
\label{6.10}
\begin{aligned}
\left\|x^i(\partial_y^2)^{j}\,\mathfrak{T}_{\Lambda}\left[F^{a}, F^{b}, F^{c}\right]\right\|_{L_y^k L_{x}^{2}} & \lesssim K \left\|x^i F^{a}\right\|_{L_{x}^{2}H_y^2}\left\|F^{b}\right\|_{L_x^2 H_y^2}^{\frac{1}{2}}\left\|x F^{b}\right\|_{L_x^2 H_y^2}^{\frac{1}{2}} \left\|F^{c}\right\|_{L_x^2 H_y^2}^{\frac{1}{2}}\left\|x F^{c}\right\|_{L_x^2 H_y^2}^{\frac{1}{2}} \\ & + K \left\|x^i F^{b}\right\|_{L_{x}^{2}H_y^2}\left\|F^{a}\right\|_{L_x^2 H_y^2}^{\frac{1}{2}}\left\|x F^{a}\right\|_{L_x^2 H_y^2}^{\frac{1}{2}} \left\|F^{c}\right\|_{L_x^2 H_y^2}^{\frac{1}{2}}\left\|x F^{c}\right\|_{L_x^2 H_y^2}^{\frac{1}{2}} \\ & + K \left\|x^i F^{c}\right\|_{L_{x}^{2} H_y^2}\left\|F^{a}\right\|_{L_x^2 H_y^2}^{\frac{1}{2}}\left\|x F^{a}\right\|_{L_x^2 H_y^2}^{\frac{1}{2}} \left\|F^{b}\right\|_{L_x^2 H_y^2}^{\frac{1}{2}}\left\|x F^{b}\right\|_{L_x^2 H_y^2}^{\frac{1}{2}} \\ & \lesssim K \|F^a\|_{S}\|F^b\|_{S} \|F^c\|_{S}.
\end{aligned}
\end{equation}
where $i,j \in \{0,1\}, k \in \{1,2\}$.\\\\
In particular, we have
\begin{equation}
\label{6.171}
\begin{aligned}
\left\|x\,\mathfrak{T}_{\Lambda}\left[F^{a}, F^{b}, F^{c}\right]\right\|_{L_{x}^{2} H_y^2} & \lesssim K \left\|x F^{a}\right\|_{L_{x}^{2}H_y^2}\left\|F^{b}\right\|_{L_x^2 H_y^2}^{\frac{1}{2}}\left\|x F^{b}\right\|_{L_x^2 H_y^2}^{\frac{1}{2}} \left\|F^{c}\right\|_{L_x^2 H_y^2}^{\frac{1}{2}}\left\|x F^{c}\right\|_{L_x^2 H_y^2}^{\frac{1}{2}} \\ & + K \left\|x F^{b}\right\|_{L_{x}^{2}H_y^2}\left\|F^{a}\right\|_{L_x^2 H_y^2}^{\frac{1}{2}}\left\|x F^{a}\right\|_{L_x^2 H_y^2}^{\frac{1}{2}} \left\|F^{c}\right\|_{L_x^2 H_y^2}^{\frac{1}{2}}\left\|x F^{c}\right\|_{L_x^2 H_y^2}^{\frac{1}{2}} \\ & + K \left\|x F^{c}\right\|_{L_{x}^{2} H_y^2}\left\|F^{a}\right\|_{L_x^2 H_y^2}^{\frac{1}{2}}\left\|x F^{a}\right\|_{L_x^2 H_y^2}^{\frac{1}{2}} \left\|F^{b}\right\|_{L_x^2 H_y^2}^{\frac{1}{2}}\left\|x F^{b}\right\|_{L_x^2 H_y^2}^{\frac{1}{2}} \\ & \lesssim K \|F^a\|_{S}\|F^b\|_{S} \|F^c\|_{S}.
\end{aligned}
\end{equation}
Moreover, we have
\begin{equation}
\label{6.181}
\left\|\mathfrak{T}_{\Lambda}\left[F^{a}, F^{b}, F^{c}\right]\right\|_{S} \lesssim K \|F^a\|_{S}\|F^b\|_{S} \|F^c\|_{S}.
\end{equation}
\end{lemma}
\begin{proof}
Since we have (\ref{6.7}), we can deduce that
\begin{align*}
\left\|\mathfrak{T}_{\Lambda}\left[F^{a}, F^{b}, F^{c}\right]\right\|_{L_{x, y}^{2}} & \lesssim K \min _{\{\alpha, \beta, \gamma\}=\{a, b, c\}}\left\|\left\|F^{\alpha}\right\|_{L_{x}^{2}}\left\|F^{\beta}\right\|_{L_x^2}^{\frac{1}{2}}\left\|x F^{\beta}\right\|_{L_x^2}^{\frac{1}{2}} \left\|F^{\gamma}\right\|_{L_x^2}^{\frac{1}{2}}\left\|x F^{\gamma}\right\|_{L_x^2}^{\frac{1}{2}}\right\|_{L_y^2} \\  & \lesssim K \min _{\{\alpha, \beta, \gamma\}=\{a, b, c\}}\left\|F^{\alpha}\right\|_{L_{x,y}^{2}}\left\|F^{\beta}\right\|_{L_x^2 L_y^\infty}^{\frac{1}{2}}\left\|x F^{\beta}\right\|_{L_x^2 L_y^\infty}^{\frac{1}{2}} \left\|F^{\gamma}\right\|_{L_x^2 L_y^\infty}^{\frac{1}{2}}\left\|x F^{\gamma}\right\|_{L_x^2 L_y^\infty}^{\frac{1}{2}} \\ & \lesssim  K \min _{\{\alpha, \beta, \gamma\}=\{a, b, c\}}\left\|F^{\alpha}\right\|_{L_{x,y}^{2}}\left\|F^{\beta}\right\|_{L_x^2 H_y^1}^{\frac{1}{2}}\left\|x F^{\beta}\right\|_{L_x^2 H_y^1}^{\frac{1}{2}} \left\|F^{\gamma}\right\|_{L_x^2 H_y^1}^{\frac{1}{2}}\left\|x F^{\gamma}\right\|_{L_x^2 H_y^1}^{\frac{1}{2}} \\ & \lesssim K \min _{\{\alpha, \beta, \gamma\}=\{a, b, c\}}\left\|F^{\alpha}\right\|_{L_{x,y}^{2}}\left\|F^{\beta}\right\|_{S}\left\|F^{\gamma}\right\|_{S}.
\end{align*}
Then by Lemma \ref{Lemma 6.2}, we can deduce that
\begin{equation}
\label{6.191}
\left\|\mathfrak{T}_{\Lambda}\left[F^{a}, F^{b}, F^{c}\right]\right\|_{S'} \lesssim \max _{\{\alpha, \beta, \gamma\}=\{a, b, c\}}\left\|F^{\alpha}\right\|_{S'}\left\|F^{\beta}\right\|_{S}\left\|F^{\gamma}\right\|_{S}.
\end{equation}
To prove (\ref{6.10}), we prove the estimate on $\left\|x\,\mathfrak{T}_{\Lambda}\left[\partial_y^2F^{a}, F^{b}, F^{c}\right]\right\|_{L_y^1 L_{x}^{2}}$ for example, other cases can be deduced in the same way. By (\ref{6.7}) and (\ref{6.8}), we have
\begin{align*}
& \left\|x\,\mathfrak{T}_{\Lambda}\left[\partial_y^2F^{a}, F^{b}, F^{c}\right]\right\|_{L_y^1 L_{x}^{2}} \\ \lesssim & \, K \left\|x\partial_y^2 F^a\right\|_{L_{x,y}^2} \left\|F^b\right\|_{L_x^2 L_y^4}^{\frac{1}{2}}\left\|xF^b\right\|_{L_x^2 L_y^4}^{\frac{1}{2}}\left\|F^c\right\|_{L_x^2 L_y^4}^{\frac{1}{2}}\left\|xF^c\right\|_{L_x^2 L_y^4}^{\frac{1}{2}} \\ + & \, K \left\|x F^b\right\|_{L_{x}^2 L_y^4} \left\|\partial_y^2F^a\right\|_{L_{x,y}^2}^{\frac{1}{2}}\left\|x\partial_y^2F^a\right\|_{L_{x,y}^2}^{\frac{1}{2}}\left\|F^c\right\|_{L_x^2 L_y^4}^{\frac{1}{2}}\left\|xF^c\right\|_{L_x^2 L_y^4}^{\frac{1}{2}} \\ + & \,K \left\|x F^c\right\|_{L_{x}^2 L_y^4} \|\partial_y^2F^a\|_{L_{x,y}^2 }^{\frac{1}{2}}\left\|x\partial_y^2F^a\right\|_{L_{x,y}^2}^{\frac{1}{2}}\left\|F^b\right\|_{L_x^2 L_y^4}^{\frac{1}{2}}\left\|xF^b\right\|_{L_x^2 L_y^4}^{\frac{1}{2}} \\ \lesssim & \, K \left\|x F^{a}\right\|_{L_{x}^{2}H_y^2}\left\|F^{b}\right\|_{L_x^2 H_y^2}^{\frac{1}{2}}\left\|x F^{b}\right\|_{L_x^2 H_y^2}^{\frac{1}{2}} \left\|F^{c}\right\|_{L_x^2 H_y^2}^{\frac{1}{2}}\left\|x F^{c}\right\|_{L_x^2 H_y^2}^{\frac{1}{2}} \\  + & \, K \left\|x F^{b}\right\|_{L_{x}^{2}H_y^2}\left\|F^{a}\right\|_{L_x^2 H_y^2}^{\frac{1}{2}}\left\|x F^{a}\right\|_{L_x^2 H_y^2}^{\frac{1}{2}} \left\|F^{c}\right\|_{L_x^2 H_y^2}^{\frac{1}{2}}\left\|x F^{c}\right\|_{L_x^2 H_y^2}^{\frac{1}{2}} \\ + & \, K \left\|x F^{c}\right\|_{L_{x}^{2} H_y^2}\left\|F^{a}\right\|_{L_x^2 H_y^2}^{\frac{1}{2}}\left\|x F^{a}\right\|_{L_x^2 H_y^2}^{\frac{1}{2}} \left\|F^{b}\right\|_{L_x^2 H_y^2}^{\frac{1}{2}}\left\|x F^{b}\right\|_{L_x^2 H_y^2}^{\frac{1}{2}} \\ \lesssim & \, K \|F^a\|_{S}\|F^b\|_{S} \|F^c\|_{S}.
\end{align*}
So we have proved (\ref{6.10}), and (\ref{6.10}) implies (\ref{6.171}). (\ref{6.171}) and (\ref{6.191}) also imply (\ref{6.181}). The proof is complete.
\end{proof}
\begin{remark}
\label{remark 6.4}
By (\ref{6.9}) and (\ref{6.10}), we have more precise estimate on $\left\|\mathcal{F}_{x \to \xi }\mathfrak{T}_{\Lambda}\left[F^{a}, F^{b}, F^{c}\right]\right\|_{L_{\xi}^\infty L_y^2}$, which is
\begin{equation}
\label{6.23}
\begin{aligned}
& \left\|\mathcal{F}_{x \to \xi }\mathfrak{T}_{\Lambda}\left[F^{a}, F^{b}, F^{c}\right]\right\|_{L_{\xi}^\infty L_y^2} \\ \lesssim &  \left\|\mathcal{F}_{x \to \xi }\mathfrak{T}_{\Lambda}\left[F^{a}, F^{b}, F^{c}\right]\right\|_{L_y^2 L_{\xi}^\infty }  \\  \lesssim & \left\|\mathfrak{T}_{\Lambda}\left[F^{a}, F^{b}, F^{c}\right]\right\|_{L_y^2 L_{x}^1 }  \\ \lesssim & \left\|\mathfrak{T}_{\Lambda}\left[F^{a}, F^{b}, F^{c}\right]\right\|_{L_{x,y}^2 }^{\frac{1}{2}}\left\|x\, \mathfrak{T}_{\Lambda}\left[F^{a}, F^{b}, F^{c}\right]\right\|_{L_{x,y}^2 }^{\frac{1}{2}} \\ \lesssim & \, K \|F^a\|_{L_x^2 H_y^2 }^{\frac{1}{4}}\|xF^a\|_{L_x^2 H_y^2 }^{\frac{1}{4}} \|F^b\|_{L_x^2 H_y^2 }^{\frac{1}{4}}\|xF^b\|_{L_x^2 H_y^2 }^{\frac{1}{4}} \|F^c\|_{L_x^2 H_y^2 }^{\frac{1}{4}} \|xF^c\|_{ L_x^2 H_y^2 }^{\frac{1}{4}} \|F^a\|_S^{\frac{1}{2}}\|F^b\|_S^{\frac{1}{2}}\|F^c\|_S^{\frac{1}{2}} \\ \lesssim & \, K \|F^a\|_S \|F^b\|_S \|F^c\|_S.
\end{aligned}
\end{equation}
\end{remark}
\begin{remark}
\label{remark 6.41}
Let $k_1, k_2 \in [0,1]$. If we change the assumptions (\ref{6.8}) to
\begin{equation}
\left\|\mathfrak{T}_{\Lambda}\left[F^{a}, F^{b}, F^{c}\right]\right\|_{L_{x}^{2}} \lesssim K \min _{\{\alpha, \beta, \gamma\}=\{a, b, c\}}\left\|F^{\alpha}\right\|_{L_{x}^{2}}\left\|\langle x\rangle^{k_1} F^{\beta}\right\|_{L_x^2} \left\|\langle x\rangle^{k_2} F^{\gamma}\right\|_{L_x^2},
\end{equation}
we can still deduce that
\begin{equation}
\left\|\mathfrak{T}_{\Lambda}\left[F^{a}, F^{b}, F^{c}\right]\right\|_{L_{x,y}^2} \lesssim K \min _{\{\alpha, \beta, \gamma\}=\{a, b, c\}}\left\|F^{\alpha}\right\|_{L_{x,y}^{2}}\left\|F^{\beta}\right\|_{S}\left\|F^{\gamma}\right\|_{S},
\end{equation}
\begin{equation}
\left\|x^i(\partial_y^2)^{j}\,\mathfrak{T}_{\Lambda}\left[F^{a}, F^{b}, F^{c}\right]\right\|_{L_y^k L_{x}^{2}} \lesssim K \|F^a\|_{S}\|F^b\|_{S} \|F^c\|_{S},
\end{equation}
\begin{equation}
\left\|\mathfrak{T}_{\Lambda}\left[F^{a}, F^{b}, F^{c}\right]\right\|_{S} \lesssim K \|F^a\|_{S}\|F^b\|_{S} \|F^c\|_{S}
\end{equation}
and
\begin{equation}
\left\|\mathcal{F}_{x \to \xi }\mathfrak{T}_{\Lambda}\left[F^{a}, F^{b}, F^{c}\right]\right\|_{L_{\xi}^\infty L_y^2} \lesssim K \|F^a\|_{S}\|F^b\|_{S} \|F^c\|_{S},
\end{equation}
where $i,j \in \{0,1\}, k \in \{1,2\}$.
\end{remark}
\begin{remark}
\label{remark 6.6}
Let $\psi : \mathbb{R} \to \mathbb{R}$ be a bounded function. In this paper, we usually give the estimate on some trilinear operator $\psi(D_y)\mathfrak{T}$ with the norm $Z$. For the part of $\left\|\mathcal{F}_{x \rightarrow \xi} \psi(D_y) \mathfrak{T}_{\Lambda}\left[F^{a}, F^{b}, F^{c}\right]\right\|_{L_{\xi}^{\infty} \dot{B}_y^1}$, we have the following estmate,
\begin{align*}
& \left\|\mathcal{F}_{x \rightarrow \xi} \psi(D_y) \mathfrak{T}_{\Lambda}\left[F^{a}, F^{b}, F^{c}\right]\right\|_{L_{\xi}^{\infty} \dot{B}_y^1} \\ \lesssim & \sum_{k \leq 0} 2^{k}\left\|\int_{\mathbb{R}} \mathrm{e}^{i y \eta}\psi(\eta) \phi\left(\frac{\eta}{2^{k}}\right) d \eta\right\|_{L_{y}^{1}} \left\|\mathcal{F}_{x \rightarrow \xi}\mathfrak{T}_{\Lambda}\left[F^{a}, F^{b}, F^{c}\right]\right\|_{L_{\xi}^{\infty} L_y^1}\\ + & \sum_{k > 0} 2^{-k}\left\|\int_{\mathbb{R}} \mathrm{e}^{i y \eta} \psi(\eta) \phi\left(\frac{\eta}{2^{k}}\right) d \eta\right\|_{L_{y}^{1}} \left\|\mathcal{F}_{x \rightarrow \xi} \partial_y^2 \mathfrak{T}_{\Lambda}\left[F^{a}, F^{b}, F^{c}\right]\right\|_{L_{\xi}^{\infty} L_y^1} \\ \lesssim & \sum_{k \leq 0} 2^{k}\left\|\int_{\mathbb{R}} \mathrm{e}^{i y \eta}\psi(\eta) \phi\left(\frac{\eta}{2^{k}}\right) d \eta\right\|_{L_{y}^{1}} \left\|\mathfrak{T}_{\Lambda}\left[F^{a}, F^{b}, F^{c}\right]\right\|_{L_{x,y}^1}\\  + & \sum_{k > 0} 2^{-k}\left\|\int_{\mathbb{R}} \mathrm{e}^{i y \eta} \psi(\eta) \phi\left(\frac{\eta}{2^{k}}\right) d \eta\right\|_{L_{y}^{1}}  \left\|\partial_y^2\mathfrak{T}_{\Lambda}\left[F^{a}, F^{b}, F^{c}\right]\right\|_{L_{x,y}^1} \\ \lesssim & \sum_{k \leq 0} 2^{k}\left\|\int_{\mathbb{R}} \mathrm{e}^{i y \eta}\psi(\eta) \phi\left(\frac{\eta}{2^{k}}\right) d \eta\right\|_{L_{y}^{1}}\left\|\mathfrak{T}_{\Lambda}\left[F^{a}, F^{b}, F^{c}\right]\right\|_{L_{y}^1 L_x^2}^{\frac{1}{2}}\left\|x\mathfrak{T}_{\Lambda}\left[F^{a}, F^{b}, F^{c}\right]\right\|_{L_{y}^1 L_x^2}^{\frac{1}{2}} \\ + & \sum_{k > 0} 2^{-k}\left\|\int_{\mathbb{R}} \mathrm{e}^{i y \eta} \psi(\eta) \phi\left(\frac{\eta}{2^{k}}\right) d \eta\right\|_{L_{y}^{1}}  \left\|\partial_y^2\mathfrak{T}_{\Lambda}\left[F^{a}, F^{b}, F^{c}\right]\right\|_{L_{y}^1 L_x^2}^{\frac{1}{2}}\left\|x\partial_y^2\mathfrak{T}_{\Lambda}\left[F^{a}, F^{b}, F^{c}\right]\right\|_{L_{y}^1 L_x^2}^{\frac{1}{2}},
\end{align*}
where the second term after the first inequality above is deduced by Bernstein's inequality. And we rewrite the above estimate as 
\begin{equation}
\label{6.14}
\begin{aligned}
& \left\|\mathcal{F}_{x \rightarrow \xi}\psi(D_y) \mathfrak{T}_{\Lambda}\left[F^{a}, F^{b}, F^{c}\right]\right\|_{L_{\xi}^{\infty} \dot{B}_y^1}\\ \lesssim & \sum_{k \leq 0} 2^{k}\left\|\int_{\mathbb{R}} \mathrm{e}^{i y \eta}\psi(\eta) \phi\left(\frac{\eta}{2^{k}}\right) d \eta\right\|_{L_{y}^{1}}\left\|\mathfrak{T}_{\Lambda}\left[F^{a}, F^{b}, F^{c}\right]\right\|_{L_{y}^1 L_x^2}^{\frac{1}{2}}\left\|x\mathfrak{T}_{\Lambda}\left[F^{a}, F^{b}, F^{c}\right]\right\|_{L_{y}^1 L_x^2}^{\frac{1}{2}} \\ + & \sum_{k > 0} 2^{-k}\left\|\int_{\mathbb{R}} \mathrm{e}^{i y \eta} \psi(\eta) \phi\left(\frac{\eta}{2^{k}}\right) d \eta\right\|_{L_{y}^{1}}  \left\|\partial_y^2\mathfrak{T}_{\Lambda}\left[F^{a}, F^{b}, F^{c}\right]\right\|_{L_{y}^1 L_x^2}^{\frac{1}{2}}\left\|x\partial_y^2\mathfrak{T}_{\Lambda}\left[F^{a}, F^{b}, F^{c}\right]\right\|_{L_{y}^1 L_x^2}^{\frac{1}{2}}.
\end{aligned}
\end{equation}
Then by (\ref{6.10}) in Lemma \ref{Lemma 6.21},  we have
\begin{align*}
& \left\|\mathfrak{T}_{\Lambda}\left[F^{a}, F^{b}, F^{c}\right]\right\|_{L_{y}^1 L_x^2}^{\frac{1}{2}}\left\|x\mathfrak{T}_{\Lambda}\left[F^{a}, F^{b}, F^{c}\right]\right\|_{L_{y}^1 L_x^2}^{\frac{1}{2}} \\  \lesssim & \, K \|F^a\|_{L_x^2 H_y^2 }^{\frac{1}{4}}\|xF^a\|_{L_x^2 H_y^2 }^{\frac{1}{4}} \|F^b\|_{L_x^2 H_y^2 }^{\frac{1}{4}}\|xF^b\|_{L_x^2 H_y^2 }^{\frac{1}{4}} \|F^c\|_{L_x^2 H_y^2 }^{\frac{1}{4}} \|xF^c\|_{ L_x^2 H_y^2 }^{\frac{1}{4}} \|F^a\|_S^{\frac{1}{2}}\|F^b\|_S^{\frac{1}{2}}\|F^c\|_S^{\frac{1}{2}} 
\end{align*}
and
\begin{align*}
& \left\|\partial_y^2\mathfrak{T}_{\Lambda}\left[F^{a}, F^{b}, F^{c}\right]\right\|_{L_{y}^1 L_x^2}^{\frac{1}{2}}\left\|x\partial_y^2\mathfrak{T}_{\Lambda}\left[F^{a}, F^{b}, F^{c}\right]\right\|_{L_{y}^1 L_x^2}^{\frac{1}{2}} \\  \lesssim & \, K \|F^a\|_{L_x^2 H_y^2 }^{\frac{1}{4}}\|xF^a\|_{L_x^2 H_y^2 }^{\frac{1}{4}} \|F^b\|_{L_x^2 H_y^2 }^{\frac{1}{4}}\|xF^b\|_{L_x^2 H_y^2 }^{\frac{1}{4}} \|F^c\|_{L_x^2 H_y^2 }^{\frac{1}{4}} \|xF^c\|_{ L_x^2 H_y^2 }^{\frac{1}{4}} \|F^a\|_S^{\frac{1}{2}}\|F^b\|_S^{\frac{1}{2}}\|F^c\|_S^{\frac{1}{2}}. 
\end{align*}
So we have 
\begin{equation}
\label{6.21}
\begin{aligned}
& \left\|\mathcal{F}_{x \rightarrow \xi} \mathfrak{T}_{\Lambda}\left[F^{a}, F^{b}, F^{c}\right]\right\|_{L_{\xi}^{\infty} \dot{B}_y^1} \\ \lesssim & \, K \left(\sum_{k \leq 0} 2^{k}\left\|\int_{\mathbb{R}} \mathrm{e}^{i y \eta}\psi(\eta) \phi\left(\frac{\eta}{2^{k}}\right) d \eta\right\|_{L_{y}^{1}} + \sum_{k > 0} 2^{k}\left\|\int_{\mathbb{R}} \mathrm{e}^{i y \eta}\psi(\eta) \phi\left(\frac{\eta}{2^{k}}\right) d \eta\right\|_{L_{y}^{1}} \right) \times \\ & \|F^a\|_{L_x^2 H_y^2 }^{\frac{1}{4}}\|xF^a\|_{L_x^2 H_y^2 }^{\frac{1}{4}} \|F^b\|_{L_x^2 H_y^2 }^{\frac{1}{4}}\|xF^b\|_{L_x^2 H_y^2 }^{\frac{1}{4}} \|F^c\|_{L_x^2 H_y^2 }^{\frac{1}{4}} \|xF^c\|_{ L_x^2 H_y^2 }^{\frac{1}{4}} \|F^a\|_S^{\frac{1}{2}}\|F^b\|_S^{\frac{1}{2}}\|F^c\|_S^{\frac{1}{2}} \\ \lesssim & \, K  \left(\sum_{k \leq 0} 2^{k}\left\|\int_{\mathbb{R}} \mathrm{e}^{i y \eta}\psi(\eta) \phi\left(\frac{\eta}{2^{k}}\right) d \eta\right\|_{L_{y}^{1}} + \sum_{k > 0} 2^{k}\left\|\int_{\mathbb{R}} \mathrm{e}^{i y \eta}\psi(\eta) \phi\left(\frac{\eta}{2^{k}}\right) d \eta\right\|_{L_{y}^{1}} \right) \|F^a\|_S\|F^b\|_S\|F^c\|_S.
\end{aligned}
\end{equation}
\end{remark}


\begin{thebibliography}{}
\bibitem{1} \label{[2.5]} Bahri, Y., Ibrahim, S., Kikuchi, H.: Remarks on solitary waves and Cauchy problem for a Half-wave-Schrödinger equations. Communications in Contemporary Mathematics (2020), p. 2050058.
\bibitem{1} \label{[2.1]} Burq, N., Gérard, P., and Tzvetkov, N. Bilinear eigenfunction estimates and the nonlinear Schrödinger equation on surfaces. Invent. Math. 159.1 (2005), pp. 187–223. 
\bibitem{1} \label{[3]} Colliander, J., Keel, M., Staffilani, G., Takaoka, H., Tao, T.: Global well-posedness for Schrödinger equations with derivative. SIAM J. Math. Anal. 33, 649–669, 2001.
\bibitem{1}  \label{[3.5]} Gérard, P., Grellier, S.: The cubic Szeg\H o equation, Ann. Sci. Éc. Norm. Supér. (4), 43(5) :761–810, 2010.
\bibitem{1} \label{[5.5]} Gérard, P., Grellier, S.: Effective integrable dynamics for some nonlinear wave equation, Anal. PDE, 5(5) :1139–1155, 2012.
\bibitem{1} \label{[4]} Gérard, P., Grellier, S.: The Cubic Szeg\H o equation and Hankel operators, volume 389 of Astérisque. Soc. Math. de France, 2017.
\bibitem{1} \label{[[1]]} Gérard, P., Pushnitski, A.: An inverse problem for Hankel operators and turbulent solutions of the cubic Szeg\H o equation on the line. \href{https://arxiv.org/abs/2202.03783}{arXiv:2202.03783, 2022}.
\bibitem{1}\label{[7]} Hani, Z., Pausader, B., Tzvetkov, N., Visciglia, N.: Modified scattering for the cubic Schrödinger equation on product spaces and applications. Forum Math. Pi 3 (2015).
\bibitem{1} \label{[2.2]} Kato, I.: Ill-posedness for the Half wave Schrödinger equation. \href{https://arxiv.org/abs/2112.10326}{arXiv:2112.10326}
\bibitem{1} \label{[2.201]} Kato, J., Pusateri, F.: A new proof of long-range scattering for critical nonlinear Schrödinger equations .Differential Integral Equations 24 (2011), no. 9-10, 923-940.
\bibitem{1} \label{[10]} Miller, P.D., Perry, P.A., Saut, J.C., Sulem, C.: Nonlinear Dispersive Partial Differential Equations and Inverse Scattering. Springer, ISBN 978-1-4939-9805-0, 2019.
\bibitem{1} \label{[5]} Peller, V.: Hankel Operators and Their Applications. Springer, ISBN 0-387-95548-8, 2002.
\bibitem{1}  \label{[3.6]} Pocovnicu, O.: Traveling waves for the cubic Szegö equation on the real line, Anal. PDE, 4 (2011), no. 3, 379-404.
\bibitem{1}  \label{[3.61]} Pocovnicu, O.: Explicit formula for the solution of the Szegö equation on the real line and applications, Discrete Contin. Dyn. Syst. A 31 (2011) no. 3, 607-649. 
\bibitem{1} \label{[6]} Pocovnicu, O.: First and second order approximations for a nonlinear wave equation, J.Dyn.Differ.Equ. 25(2), 305–333 (2013).
\bibitem{1} \label{[2]} Xu, H.: Unbounded Sobolev trajectories and modified scattering for a wave guide non-linear Schrödinger equation. Mathematische Zeitschrift, 286:443–489, 2017.

\end{thebibliography}
\end{document}